\documentclass[11 pt,leqno]{amsart}

\usepackage[all]{xy}
\usepackage{amsfonts}
\usepackage{amsthm}
\usepackage{amssymb}
\usepackage{mathabx}
\usepackage{latexsym}
\usepackage{amsmath}
\usepackage{mathrsfs}
\usepackage{dsfont}
\usepackage{hyperref}
\usepackage{a4wide}
\usepackage{ifthen}

\usepackage{tcolorbox}

\usepackage{mathabx}
\usepackage[all]{xy}
\xyoption{rotate}

\usepackage{enumitem} 
\setlist[enumerate,1]{label=\textup{(\arabic*)}}
\setlist[enumerate,2]{label=\textup{(\alph*)}}

\usepackage{fancybox}
\usepackage{lmodern}
\usepackage{tikz}
\tikzset{rndblock/.style={rounded corners,rectangle,draw,outer sep=0pt}}
\newcommand{\tframed}[2][]{\tikz[baseline=(h.base)]\node[rndblock,#1] (h) {#2};}

\theoremstyle{plain}
\newtheorem{thmx}{Theorem}

\newtheorem{corx}[thmx]{Corollary}

\newtheorem{thm}{Theorem}[section]

\newtheorem {lem} [thm]{Lemma}
\newtheorem {prop}[thm] {Proposition}

\newtheorem{cor}[thm]{Corollary}

\theoremstyle{definition}
\newtheorem {defn}[thm] {Definition}
\newtheorem {rem} [thm]{Remark}
\newtheorem{ex}[thm]{Example}

\newcommand{\eqn}{\begin{equation}}
\newcommand{\eqne}{\end{equation}}

\newcommand{\anti}{\stackrel{\mathrm{anti}}{\cong}} 

\DeclareMathOperator{\id}{id}

\DeclareMathOperator{\PHomeo}{PHomeo}

\DeclareMathOperator{\PAut}{PAut}

\DeclareMathOperator{\red}{r}
\DeclareMathOperator{\ess}{e}
\DeclareMathOperator{\Hau}{H}

\DeclareMathOperator{\tr}{tr}

\DeclareMathOperator{\BB}{\mathbb{B}}

\DeclareMathOperator{\SPIso}{SPIso}

\DeclareMathOperator{\clsp}{\overline{span}}
\DeclareMathOperator{\spane}{span}

\DeclareMathOperator{\supp}{supp}
\DeclareMathOperator{\Null}{\mathcal{N}}
\DeclareMathOperator{\Bis}{Bis}
\DeclareMathOperator{\E}{\mathbb{E}}

\DeclareMathOperator{\tight}{tight}

\newcommand{\FS}{F^{S}}
\newcommand{\FR}{F_{\RR}}

\renewcommand{\H}{\mathcal H}

\newcommand{\s}{d}

\newcommand{\B}{\mathcal B}
\newcommand{\LL}{\mathcal L}

\newcommand{\G}{\mathcal G} 

\newcommand{\EE}{\mathcal E}
\newcommand{\RR}{\mathcal R}

\newcommand{\OO}{\mathcal{O}}

\newcommand{\F}{\mathbb F}

\renewcommand{\L}{\mathcal L}

\newcommand{\D}{\mathcal D}
\newcommand{\C}{\mathbb C}
\newcommand{\R}{\mathbb R}
\newcommand{\Z}{\mathbb Z}
\newcommand{\N}{\mathbb N}

\newcommand{\cst}{\ifmmode\mathrm{C}^*\else{$\mathrm{C}^*$}\fi}

\newcommand{\Falg}[4]{F^{#2}_{#3}({#1} \ifthenelse{\equal{#4}{}}{}{, {#4}})} 
\newcommand{\Falgg}{\Falg} 
\newcommand{\FFalg}[4]{\widetilde{F}_{#2}^{#3}({#1} \ifthenelse{\equal{#4}{}}{}{, {#4}})}

\subjclass[2000]{47L10, 22A22, 46H10}

\begin{document}
\author{Krzysztof Bardadyn}
\address{Faculty of Mathematics, University of Bia\l ystok, ul. K. Cio\l kowskiego 1M, 15-245	Bia\l ystok, Poland}
\email{kbardadyn@math.uwb.edu.pl}

\author{Bartosz Kwa\'sniewski}
\address{Faculty of Mathematics, University of Bia\l ystok, ul. K. Cio\l kowskiego 1M, 15-245	Bia\l ystok, Poland}
\email{b.kwasniewski@uwb.edu.pl}

\author{Andrew McKee}
\address{Faculty of Mathematics, University of Bia\l ystok, ul. K. Cio\l kowskiego 1M, 15-245Bia\l ystok, Poland}
\email{a.mckee@uwb.edu.pl}


\title{\small Banach algebras associated to  twisted \'etale groupoids:\\
 simplicity and pure infiniteness}

\begin{abstract} 
We define  reduced and essential Banach algebras associated to a twisted \'etale (not necessarily Hausdorff) groupoid $(\G,\LL)$
 and extend some fundamental results from $C^*$-algebras to this context. 
We prove that for topologically free groupoids the associated essential Banach algebras have the ideal intersection property,
and thus such an algebra is simple if and only if the groupoid is minimal. We give conditions under which reduced algebras are essential (for example Hausdorffness of $\G$ is sufficient). 
This in particular solves the simplicity problem posed  by Gardella--Lupini for $L^p$-operator algebras associated to $\G$. 
In addition, using either the $n$-filling or locally contracting condition we give pure infiniteness criteria for essential simple Banach algebras associated to  $(\G,\LL)$.  
This extends the corresponding $C^*$-algebraic results that were previously known to hold in the untwisted Hausdorff case.
The results work nicely, and allow for characterization of the generalized intersection property, in the realm of $L^P$-operator algebras  where $P\subseteq [1,\infty]$ is a non-empty set of parameters. Such algebras cover in particular
$L^p$-operator algebras, for $p\in [1,\infty]$, and  their Banach $*$-algebra versions.

We apply our results to Banach algebra crossed products by twisted partial group actions, Roe-type Banach algebras with coefficients in finite-rank operators on a Banach space, twisted tight $L^P$-operator algebras of inverse semigroups, graph $L^P$-operator algebras, and $L^P$-versions of Exel-Pardo algebras. 
\end{abstract}

\maketitle

%
%

%
%
\section{Introduction}
\label{sec:Introduction}

The characterization of simplicity of the crossed product of a discrete group action on a topological space, \cite{Kawa-Tomi}, \cite{Arch_Spiel}, is one of the most fundamental facts from the theory of $C^*$-algebras associated to dynamical systems. 
The dynamical conditions appearing there, \emph{topological freeness} and \emph{minimality}, are easy to check, applicable to a number of actions, and, most importantly, they are also sufficient for simplicity of twisted reduced crossed products. 
This classical result was generalized to twisted actions of second countable Hausdorff groupoids in \cite[Corollary 4.6]{Renault91}, though it seemed to gain more attention when phrased in terms of second countable Hausdorff \'etale groupoids in \cite{Exel_Non-Hausdorff}, \cite{BCFS}. 
One boost came from Renault's theory of Cartan inclusions \cite{Re}, as they are modeled by twisted \'etale Hausdorff groupoids satisfying a version of topological freeness for groupoids, which in the separable case is \emph{topological principality}. 
Since the example of Skandalis from the Appendice to \cite{Renault91}, see also \cite{Exel_Non-Hausdorff},  it was unclear how to formulate simplicity criteria in the non-Hausdorff case. 
On the other hand, a number of important constructions lead to non-Hausdorff \'etale groupoids, cf.\ \cite{Nekrashevych2}, \cite{Katsura}, \cite{EP16}, \cite{Exel-Pardo:Self-similar}, and recently, after a significant breakthrough  obtained in \cite{CEPSS}, \cite{Kwa-Meyer}, the number of papers on this topic has increased rapidly. 
In this article, we present a cleaned version of this modern theory. 
The main achievement is that we  develop a well-oiled machinery that works not only for $C^*$-algebras but in the much broader  setting of Banach algebras. 
By removing a few important stumbling blocks in this setting, such as the simplicity problem posed in \cite{Gardella_Lupini17}, this should open the door to many new applications and research directions connecting the current theories of $C^*$-algebras, Banach algebras and harmonic analysis. 

Our initial motivation came from Phillips's program to develop a theory of \emph{$L^p$-operator algebras}, that is, Banach algebras that can be isometrically represented on a Banach space $L^p(\mu)$ for some measure $\mu$, see \cite{Gardella} and/or \cite{PhLp1}, \cite{PhLp2a}, \cite{Phillips}, \cite{Gardella_Thiel15}, \cite{PH}, \cite{Gardella_Lupini17}, \cite{Phillips19}, \cite{cortinas_rodrogiez}, \cite{cgt},  \cite{Austad_Ortega}, \cite{Hetland_Ortega}, \cite{cortinas_Montero_rodrogiez} and the references there. 
In particular, in our previous paper \cite{BKM} we presented a thorough study of representations of the universal Banach algebras $F^S(\G,\LL)$ associated with a twisted \'etale groupoid $(\G,\LL)$, which are graded by some fixed unital inverse semigroup $S$ of bisections of $\G$. 
Representations of $F^S(\G,\LL)$ on $L^p$-spaces do not see this inverse semigroup. 
In fact for any $p\in [1,\infty]$ the \emph{universal $L^p$-operator algebras} $F^p(\G,\LL)$ are graded by the semigroup $\Bis(\G)$ of all bisections of $\G$. 
For $p\neq 2$, their representations can be described in terms of spatial partial isometries \cite{BKM}, and $F^2(\G,\LL)$ coincides with the usual universal $C^*$-algebra $C^*(\G,\LL)$ associated to $(\G,\LL)$ (or a real version of it). 
To cover other relevant Banach algebras appearing in the literature, such as group Banach algebras \cite{Gardella_Thiel15}, \cite{LiYu}, \cite{Phillips19}, Banach algebra crossed products \cite{DDW}, \cite{BK}, Cuntz Banach algebras \cite{Daws_Horwath}, Roe type Banach algebras \cite{Spakula_Willett} or $*$-Banach algebras associated to groupoids \cite{Austad_Ortega}, \cite{Hetland_Ortega}, in general we consider Banach algebras $F_{\RR}(\G,\LL)$ that are completions of the underlying convolution algebra $\mathfrak{C}_c(\G , \L)$ of quasi-continuous sections in some seminorm $\|\cdot \|_{\RR}$ bounded by the norm coming from some $S$. 
We view the twist $\LL$ over $\G$ as a Fell line bundle over $\F=\R,\C$, and the reader not interested in twists may simply ignore it.

In this paper, we turn to the study of simplicity and pure infiniteness of algebras of the form $F_{\RR}(\G,\LL)$. 
This requires introducing \emph{reduced groupoid Banach algebras} of $(\G,\LL)$ and in the non-Hausdorff case also \emph{essential groupoid Banach algebras} of $(\G,\LL)$.
As it is simpler, we discuss first the case when the \'etale groupoid $\G$ is Hausdorff. 
This holds if and only if its locally compact Hausdorff unit space $X := \G^{(0)}$ is a closed  subset of $\G$. 
Then the convolution algebra $\mathfrak{C}_c(\G , \L) = C_c(\G,\LL)$ consists of compactly supported continuous sections. 
We say that a Banach algebra $\FR(\G,\LL)$ is \emph{reduced} if it is a completion of $C_c(\G,\LL)$ such that the inclusion map $C_c(\G,\LL)\subseteq C_0(\G,\LL)$ extends to an injective contractive map $j_{\RR} : \FR(\G,\LL) \to C_0(\G,\LL)$ which isometric on $C_0(X)$, and so $C_0(X)$ is naturally a subalgebra of $\FR(\G,\LL)$. 
This definition covers reduced group algebras \cite{Phillips19}, reduced crossed products \cite{BK}, as well as the reduced $L^p$-algebras $F^p_{\red}(\G,\LL)$, $p\in [1,\infty]$ studied in \cite{Gardella_Lupini17}, \cite{cgt}, \cite{Austad_Ortega}, \cite{Hetland_Ortega}, \cite{BKM}. 
Following \cite{BK}, we generalize $L^p$-operator algebras to $L^P$-operator algebras where $P\subseteq [1,\infty]$ is a non-empty set of parameters. 
In particular, we introduce the full $F^P(\G,\LL)$ and reduced $F^P_{\red}(\G,\LL)$ groupoid $L^P$-algebras.   

The main difficulty  is to establish the relationship between \emph{topological freeness} of $\G$, as defined in \cite{Kwa-Meyer}, and the \emph{intersection property} for  $C_0(X) \subseteq \FR(\G,\LL)$, which means that $I\cap C_0(X) \neq \{ 0 \}$ for every non-zero (closed two-sided) ideal $I$ in $\FR(\G,\LL)$; equivalentely, every representation of $\FR(\G,\LL)$ which is injective on $C_0(X)$ is injective on $\FR(\G,\LL)$.
We achieve this by establishing the relationships between the canonical conditional expectation onto $C_0(X)$ and the map $j_{\RR}$. 
Here the main idea is to extrapolate the conditional expectation through bisections and the crucial technical tool is the pinching lemma (Lemma \ref{lem:pinching_property}). 
For a Hausdorff groupoid, topological freeness is equivalent to \emph{effectiveness}, that is, the interior of the isotropy bundle being only the unit space. 
In particular, combining Proposition~\ref{prop:Cartan_subalgebra} and Theorems~\ref{thm:Top_Free_Intersection_Property}, \ref{enu:topological_freeness_twisted_intersection}, \ref{thm:topological_freeness_untwisted} we have the following result: 

\begin{thmx}\label{thm:topological_freeness_intro} 
Let $(\G,\LL)$ be a twisted groupoid where $\G$ is a locally compact Hausdorff \'etale groupoid   with unit space $X$. 
Let $P\subseteq [1,\infty]$ be any non-empty set of parameters and let $\FR(\G,\LL)$ be  any reduced twisted groupoid Banach algebra, which 
is graded by some   unital inverse semigroup $S$ of bisections covering $\G$.
The following conditions are equivalent:
\begin{enumerate}
    \item\label{enu:topological_freeness_intro1} $\G$ is topologically free (equivalently effective);
    \item $C_0(X)$ is a maximal abelian subalgebra in $F_{\RR}(\G,\LL)$;
    \item\label{enu:topological_freeness_intro2} for every ideal $I$ in the (untwisted) full algebra $F^P(\G)$, if  $I\cap C_0(X) = 0$ then  $I\subseteq \ker\Lambda_P$, where $\Lambda:F^P(\G)\to F^P_{\red}(\G)$ is the regular representation;
	 \item\label{enu:topological_freeness_intro4}  $C_0(X)\subseteq F^{P_0}(\G)$ has the intersection property for some (and hence any) $P_0\subseteq \{1,\infty\}$; 
	\item\label{enu:topological_freeness_intro5} the inclusion $C_0(X)\subseteq F^P_{\red}(\H, \LL|_{\H})$ has the intersection property for  every open subgroupoid $\mathcal{H}\subseteq \G$ containing $X$.
\end{enumerate}
If the above equivalent conditions hold, then $C_0(X)\subseteq F_{\RR}(\G,\LL)$ has the intersection property.
\end{thmx}

The equivalence between \ref{enu:topological_freeness_intro1} and \ref{enu:topological_freeness_intro2} is a far reaching generalization of the Kawamura--Tomiyama theorem \cite{Kawa-Tomi}, \cite{Arch_Spiel}, see also \cite[Theorem 7.29]{Kwa-Meyer}. 
Condition \ref{enu:topological_freeness_intro4} is a special case of \ref{enu:topological_freeness_intro2} because for $P\subseteq \{ 1,\infty \}$ we have 
$F^{P}(\G,\LL) = F^{P}_{\red}(\G,\LL)$ -- these algebras are both full and reduced, see \cite{BKM}. 
Condition \ref{enu:topological_freeness_intro5} is reminiscent of the recently discovered fact that topological freeness (aperiodicity) is equivalent to the intersection property for all intermediate $C^*$-algebras between $C_0(X)$ and the reduced $C^*$-algebra $C^*_{\red}(\G,\LL)=F_{\red}(\G,\LL)$, see  \cite[Theorem 7.3]{Kwa-Meyer2}. 
As an immediate corollary we get the following simplicity criteria (where by simplicity of a Banach algebra we mean lack of non-trivial closed two-sided ideals): 

\begin{corx}\label{cor:simplicity_intro} 
Retaining the notation of Theorem~\ref{thm:topological_freeness_intro}, the reduced twisted groupoid Banach algebra $F_{\RR}(\G,\LL)$ is simple and contains $C_0(X)$ as a maximal abelian subalgebra if and only if $\G$ is topologically free and minimal.
\end{corx}

Corollary~\ref{cor:simplicity_intro} applies to reduced algebras $F^P_{\red}(\G,\LL)$ for $P\subseteq [1,\infty]$. 
In particular, it solves in positive the problem posed in \cite[Problem 8.2]{Gardella_Lupini17}, which asked whether minimality and topological principality of a (second countable Hausdorff) groupoid $\G$ implies simplicity of $F^p_{\red}(\G)$ for $p\in(1,\infty)$. 
Combining Theorem~\ref{thm:topological_freeness_intro} with \cite[Theorem 6.19]{Gardella_Lupini17}, we get the following characterization of simplicity of full $L^P$-operator algebras for amenable (untwisted) groupoids. 

\begin{corx}\label{cor:simplicity_intro2} 
If $\G$ is an amenable second countable locally compact Hausdorff \'etale groupoid, then for any non-empty set $P\subseteq [1,\infty]$ we have that 
\begin{enumerate}
    \item\label{enu:simplicity_intro1}   the inclusion $C_0(X) \subseteq F^P(\G)$ has the intersection property
        if and only if $\G$ is topologically free;
    \item\label{enu:simplicity_intro2}   the algebra $F^P(\G)$ is simple
        if and only if $\G$ is topologically free and minimal.
\end{enumerate}
\end{corx}

In the non-Hausdorff setting, we generalize the above results by replacing reduced Banach algebras with essential ones. 
Reduced algebras $F_{\RR}(\G,\LL)$ are completions of $\mathfrak{C}_c(\G , \L)$ equipped with a canonical injective contractive map $j_{\RR} : F_{\RR}(\G,\LL)\to \mathfrak{C}_0(\G , \L)$ where $\mathfrak{C}_0(\G , \L)$ is the completion of $\mathfrak{C}_c(\G , \L)$ in $\|\cdot\|_{\infty}$ (see Definition~\ref{de:GroupoidBanachAlgebrasExpectations}). 
By an \emph{essential Banach algebra} of $(\G , \L)$ we mean a Hausdorff completion $F_{\RR}(\G,\LL)$ of $\mathfrak{C}_c(\G , \L)$ equipped with an injective contractive map $j_{\RR}^{\ess} : F_{\RR}(\G,\LL)\to \mathfrak{C}_0(\G , \L)/\mathfrak{M}_0(\G , \L)$, where $\mathfrak{M}_0(\G , \L)$ is the subspace of $\mathfrak{C}_0(\G , \L)$ consisting of functions  with meager open support (see Definition~\ref{de:ExoticGroupoidBanachAlgebra}). 
Below we phrase these  definitions in terms of existence and faithfulness of canonical generalized expectations, which makes clear that they  are consistent with existing $C^*$-algebraic notions. 
Every reduced Banach algebra can be quotiented by an ideal to get an essential Banach algebra. There are cases when this ideal is zero, for instance when $\G$ is ample and every compact open set in $\G$ is regular (see Lemma~\ref{lem:ample_essential_reduced_coincide} below).

Finally, we introduce the notion of pure infiniteness for simple (not necessarily unital) Banach algebras, which is consistent with the previous notions of this type (see Definition \ref{def:purely_infinite} and Proposition~\ref{prop:various_purely_infnite}). 
We modify the classical conditions of locally contracting \cite{A-D}, and $n$-filling \cite{Jolissaint-Robertson}, \cite{Suzuki}, to include the inverse semigroup $S$ of bisections in $\G$ (Definitions~\ref{def:n_feeling},~\ref{def:locally contractive}). 
Combining Theorems~\ref{thm:Top_Free_Intersection_Property}, \ref{thm:filling_implies_pure_infinite} and \ref{thm:Anantharaman-Delaroche} we get: 

\begin{thmx}\label{thm:pure_infinteness_intro} 
Let $(\G,\LL)$ be a twisted \'etale groupoid where $\G$ is topologically free and minimal.
Let $S$ be any unital inverse semigroup of bisections covering $\G$.
Then every $S$-graded essential Banach algebra $\FR(\G,\LL)$ is simple. 
If in addition $S$ consists of bisections where the bundle $\LL$ is trivial and $\G$ is
locally contracting or $n$-filling with respect to $S$, then $\FR(\G,\LL)$ is purely infinite simple. 
\end{thmx}

Groupoid $C^*$-algebras are always graded by the full inverse semigroup of bisections $\Bis(\G)$. 
Already in this context, the above pure infiniteness criteria are new, as until now they were known to work only in the Hausdorff untwisted case. 
In general, we need to assume local contractiveness or $n$-filling with respect to the inverse semigroup $S(\LL)$ consisting of bisections where the bundle $\LL$ is trivial. 
But when the bundle $\LL$ is trivial, equivalently $\LL$ is coming from a groupoid cocycle, we do not need to mention inverse semigroups at all, and we get for instance: 

\begin{corx}\label{cor:pure_infiniteness_intro}
Retaining the assumptions on $(\G,\LL)$ from Theorem~\ref{thm:pure_infinteness_intro},  assume that $\LL$ comes from a cocycle and $\G$ is either locally contracting or $n$-filling. 
Let $\emptyset \neq P\subseteq [1,\infty]$. 
Then $F^P_{\red}(\G,\LL)$ has a simple quotient $F^P_{\ess}(\G,\LL)$ which is purely infinite. 
This quotient is just $F^P_{\red}(\G,\LL)$ if $\G$ is Hausdorff or more generally when  $\mathfrak{M}_0(\G , \L) = \{ 0 \}$. 
\end{corx}

If $\G$ can be covered by a countable family of open bisections, then the quotient $F^P_{\ess}(\G,\LL)$ in the above corollary has a natural injective representation on a direct sum of $L^p$-spaces, $p\in P$, and it is natural to view the corresponding Banach algebra as the \emph{essential $L^P$-groupoid algebra} of $(\G,\LL)$ (see Corollary~\ref{cor:essential_L^P_operator_algebra}).

We illustrate our results by using them to generalize and extend the main results of \cite{BK} from global to partial actions (Subsection~\ref{ssec:ExamplesPartialActions}). 
We present a general groupoid picture for Roe-type Banach algebras of uniformly locally finite coarse structures with coefficients in the compact operators on some Banach space. 
Such algebras occur naturally, for instance in \cite{Spakula_Willett0}, \cite{Spakula_Willett}, \cite{Chung_Li}, in connection with criteria for Fredholmness or rigidity. 
They generalize standard uniform Roe and Roe $C^*$-algebras, and the underlying groupoid is always principal so that the equivalent conditions in Theorem~\ref{thm:topological_freeness_intro} hold (see Subsection~\ref{ssec:BanachRoeAlgebras}).
We introduce reduced and essential twisted tight $L^P$-algebras associated to inverse semigroups, and phrase our results in terms of the inverse semigroup. 
This extends, and sometimes improves, parts of \cite{Exel}, \cite{EP16} where the untwisted $C^*$-algebraic case was considered. 
It also covers all $L^P$-groupoid algebras for twisted ample groupoids (Subsection~\ref{ssec:TightBanachAlgebrasInverseSgrp}). 
Using groupoid methods we extend the main results of \cite{cortinas_Montero_rodrogiez}  from $L^p$-graph algebras to $L^P$-graph algebras (Subsection~\ref{subsect:graph_algebras}). 
Finally, we generalize Exel--Pardo algebras \cite{Exel-Pardo:Self-similar} associated to self-similar actions of groups on graphs to $L^P$-operator algebras, and discuss how the recent non-trivial results \cite{Exel-Pardo:Self-similar}, \cite{CEPSS} extend to this setting (Subsection~\ref{ssec:AlgebrasSelfSimilarActionsGraphs}). 
The groupoid modeling these algebras, as a rule, is non-Hausdorff. 
Usually the $L^p$-algebras for different $p$ are not isomorphic, thus the above results allow us to produce lots of essentially different simple Banach algebras.

The paper is organized as follows. 
After short preliminaries where we recall some basic  facts and notions from \cite{BKM}, in Section~\ref{sec:ReducedBanachAlgs} we introduce reduced groupoid Banach algebras. 
In Section~\ref{sec:Essential_and_exotic} we introduce exotic and essential groupoid Banach algebras. 
Section~\ref{sec:IntersectionAperiodicTopFree} is devoted to topological freeness and its relationship with other conditions. 
Section~\ref{sec:PureInfiniteness} discusses our pure infiniteness results and Section~\ref{sec:ApplicationsExamples} presents applications and examples. 

%
\subsection*{Acknowledgements}

The research of KB and BK was supported by the National Science Center (NCN) grant no.~2019/35/B/ST1/02684.  BK was partially supported by the National Science Centre, Poland, through the WEAVE-UNISONO grant no. 2023/05/Y/ST1/00046.
BK thanks Agnieszka Stocka whose comments helped formulate the  definition of reduced algebras  in a correct way.
We all thank the anonymous referees for pointing us to \cite{Dixmier2} and \cite{Elkiaer}, and for helpful suggestions which improved the manuscript. 

\section{Preliminaries}
\label{sec:Prelims}

We recall here some notation and facts from \cite{BKM}. 
Throughout this paper $\G$ denotes an \'etale groupoid with locally compact Hausdorff unit space $X$. 
Thus the range and domain maps $r,d:\G\to X$ are local homeomorphisms, $X$ is an open subset of $\G$, and $\G$ is Hausdorff if and only if $X$ is also closed in $\G$. 
We consider spaces, and algebras, over the field $\F=\R, \C$ of real or complex numbers. 
By an ideal in a Banach algebra, unless stated otherwise, we mean a closed and two-sided ideal. 

%
\subsection{Twisted groupoids}
\label{ssec:TwistedGroupoids} 
As  in \cite[Section 4]{BKM}, we define twists of $\G$  as Fell line bundles. 
A \emph{line bundle} $\L$ over $\G$ is a topological space $\L = \bigsqcup_{\gamma \in \G} L_\gamma$ such that the canonical map $\L \to \G$ is continuous and open, \
and each fiber $L_\gamma \cong \F$ is a one dimensional Banach space with a structure consistent with the topology of $\LL$ in the sense  that the maps
    $\bigsqcup_{\gamma\in \G} L_\gamma \times L_\gamma \ni (z, w) \mapsto z + w\in \LL$, 
    $\F\times \LL\ni (\lambda,z) \mapsto \lambda z\in \LL$ and  
    $\LL\ni z\mapsto  |z|\in \R
$ are continuous. 
For any open $U \subseteq \G$, we write $C_c(U , \L)$, $C_0(U , \L)$, $C(U , \L)$ and $C_u(U , \L)$ for the spaces of continuous sections of $\LL|_{U} := \bigsqcup_{\gamma \in U} L_\gamma$  that are compactly supported, vanishing at infinity, bounded and unitary, respectively. 
We use similar notation $\B_c(U , \L)$, $\B_0(U , \L)$, $\B(U , \L)$ when continuous is replaced by Borel. 
In particular, $\B_0(U , \L)$ consists of Borel sections $f:U\to \LL$ such that for every $\varepsilon>0$ the set $\{\gamma\in U:|f(\gamma)|\geq \varepsilon\}$ is precompact. 
We consider these spaces equipped with the supremum norm $\|f\|_{\infty}=\sup_{\gamma\in U} |f(\gamma)|$. 
Every line bundle is  locally trivial, which means that every point has a neighborhood $U\subseteq \G$ such that $C_u(U , \L) \neq \emptyset$. 
Note that, for any $c\in C_u(U , \L)$, the map $f\mapsto f \cdot c$ establishes isometric isomorphisms $C_*(U,\LL)\cong C_*(U)$ and
 $\B_{*}(U , \L)\cong \B_{*}(U)$ where  $*$ stands for either $c$, $0$ or  an empty space. 



\begin{defn}\label{de:TwistedGroupoid}
A \emph{twist over the groupoid $\G$} is a line bundle $\L = \bigsqcup_{\gamma \in \G} L_\gamma$ over $\G$ equipped with two continuous maps $\cdot :\LL\times \LL\supseteq \bigsqcup_{\s(\gamma) = r(\eta)} L_\gamma \times L_\eta \to \LL$ and $^*: \L \to \L$ called \emph{multiplication} and \emph{involution}, such that: 
\begin{enumerate}
    \item $\cdot$ restricts to bilinear maps $L_\gamma \times L_\eta \to L_{\gamma \eta}$,  $(\gamma,\eta)\in \G^{(2)}$, that make the norm $|\cdot|$ multiplicative, and $\cdot$ is associative, that is, $(\LL,\cdot)$ is a semigroupoid;
    \item $^*$ restricts to conjugate-linear maps $L_\gamma \to L_{\gamma^{-1}}$, $\gamma\in \G$, that preserve the norm, are antimultiplicative, and we have  $z \cdot z^* = | z |^2$ for all $z \in L_\gamma ,\ \gamma \in \G$. 
\end{enumerate}
Then $\LL|_X$ is necessarily trivial and \emph{we shall always assume that $\LL|_X = X \times \F$}.
\end{defn}

\begin{rem}\label{ex:cocycle_twist}
The trivial bundle $\LL = \G\times \F$ with pointwise operations is the \emph{trivial twist}.
More generally, if $\sigma : \G^{(2)} \to \{ z\in \F:  |z| = 1 \}$ is a \emph{continuous normalised $2$-cocycle} on $\G$, that is, $\sigma \big( r(\gamma),\gamma \big) = 1 = \sigma \big( \gamma , d(\gamma) \big)$ and $\sigma(\alpha,\beta)\sigma(\alpha\beta,\gamma) = \sigma(\beta,\gamma)\sigma(\alpha,\beta\gamma)$ 
for every composable triple $\alpha,\beta,\gamma \in \G$, 
 then we may consider a twist $\LL_{\sigma}$ given by the trivial bundle $\G\times \F$ with operations given by 
$(\alpha,w) \cdot (\beta,z) := \big( \alpha \beta , \sigma(\alpha,\beta) wz \big)$,
$(\alpha,w)^{*} := \big( \alpha^{-1} , \overline{\sigma(\alpha^{-1},\alpha) w } \big).$
Every twist based on a trivial line bundle is of this form. 
In general, twists defined above are equivalent to Kumjian's twist~\cite{Kumjian0} defined using extensions of $\G$. 
\end{rem}

The set of \emph{open bisections} $\Bis(\G) := \{ U\subseteq \G \text{ open} : \text{$r|_{U}$ and $d|_{U}$ are injective} \}$ 
forms a unital inverse semigroup with operations defined by 
\[
    U \cdot V := \{ \gamma \eta : \text{$\gamma \in U,\ \eta\in V$ are composable} \},  \qquad  
    U^* = U^{-1} := \{ \gamma^{-1} : \gamma \in U \}, 
\]
where $U,V\in \Bis(\G)$. The unit space $X$ is the unit in  $\Bis(\G)$. 
For any twist $\LL$ over $\G$, the set 
\[
    S(\LL) := \{U\in \Bis(\G) : \text{the restricted bundle $\LL|_{U}$ is trivial}\}
\]
forms a unital inverse subsemigroup $S(\LL)\subseteq \Bis(\G)$ which is \emph{wide} in the sense that it covers $\G$ and the intersection of any two elements in $S(\LL)$ is a union of elements from $S(\LL)$, see \cite[Lemma 4.2]{BKM}. 

%
\subsection{Banach algebras associated to twisted groupoids}
\label{ssec:TwistedGroupoidBanAlgs}

Fix a twisted groupoid $(\G , \L)$. 
We will work with the set of \emph{quasi-continuous compactly supported functions}, defined as
\[
    \mathfrak{C}_c(\G , \L) := \spane \{ f \in C_c(U , \L) : U \in \Bis(\G) \} ,
\]
where a section in $C_c(U , \L)$ is viewed as a section defined on $\G$ which vanishes off $U$. 
This set is equipped with a $*$-algebra structure as follows: 
\[
    (f * g)(\gamma) := \sum_{r(\eta) = r(\gamma)} f(\eta) \cdot g(\eta^{-1} \gamma) , \quad f^*(\gamma) := \big( f(\gamma^{-1}) \big)^* , \quad f,g \in \mathfrak{C}_c(\G , \L) ,\ \gamma \in \G , 
\]
where $\cdot$ and the final $*$ indicate the product and involution from $\L$. 
When the groupoid $\G$ is Hausdorff this reduces to the usual $*$-algebra $C_c(\G , \L)$. 
We associate a norm on $\mathfrak{C}_c(\G , \L)$ with any unital inverse subsemigroup $S \subseteq \Bis(\G)$ which covers $\G$. 
By \cite[Lemma 4.5]{BKM}, the formula
\[
    \| f \|^S_{\textrm{max}} := \inf \left\{ \sum_{k = 1}^n \| f_k \|_\infty : f = \sum_{k = 1}^n f_k ,\ f_k \in C_c(U_k , \L) ,\ U_k \in S \right\} ,
\]
defines the maximal submultiplicative norm on $\mathfrak{C}_c(\G , \L)$ which agrees with $\|\cdot\|_\infty$ on each $C_c(U,\LL) \subseteq \mathfrak{C}_c(\G,\LL)$, for $U\in S$. 
The involution in $\mathfrak{C}_c(\G , \L)$ is isometric in $\| \cdot \|^S_{\textrm{max}}$. 
Thus 
\[
    \FS(\G , \L):= \overline{\mathfrak{C}_c(\G , \L)}^{\| \cdot \|^S_{\textrm{max}} }
\]
is a Banach $*$-algebra. 
The Banach spaces $C_0(U,\LL)\subseteq \FS(\G , \L)$, $U \in S$, naturally embed into  $\FS(\G , \L)$, and they form an \emph{$S$-grading} of $\FS(\G , \L)$ in the sense that for any $U,V \in S$ we have $\overline{C_0(U,\LL)\cdot C_0(V,\LL)} = C_0(UV,\LL)$ in $\FS(\G , \L)$ and these spaces span a dense $*$-subalgebra of $\FS(\G , \L)$. 
Unitality of $S$ means that $X\in S$. 
In particular, $C_0(X)\subseteq \FS(\G , \L)$ sits in $\FS(\G , \L)$ as a non-degenerate Banach subalgebra (any approximate unit in $C_0(X)$ is a two-sided approximate unit in $\FS(\G , \L)$). 
We also write $F(\G , \L) := F^{\Bis(\G)}(\G , \L)$ and $\| \cdot \|_{\textrm{max}} = \| \cdot \|^{\Bis(\G)}_{\textrm{max}}$. 
In particular, $\| \cdot \|_{\textrm{max}}\leq \| \cdot \|^{S}_{\textrm{max}}$ for any  $S\subseteq \Bis(\G)$ as above.

\begin{defn}\label{defn:S_graded_algebras}
By a \emph{representation} of a Banach algebra $A$ in a Banach algebra $B$ we mean a contractive algebra homomorphism $\psi : A \to B$. 
Given a unital inverse subsemigroup $S \subseteq \Bis(\G)$ which covers $\G$ and a class $\RR$ of representations of $\FS(\G , \L)$, we define $\FR(\G , \L)$ to be the Hausdorff completion of $\mathfrak{C}_c(\G , \L)$ in the seminorm
\begin{equation*}\label{eq:GBANormRepresentations}
    \| f \|_{\RR} := \sup \big\{ \| \psi(f) \| : \psi \in \RR \big\} .
\end{equation*}
The closures of images of the spaces $C_0(U,\LL)$, $U \in S$, form an $S$-grading of $\FR(\G , \L)$. 
To emphasize the role of $S$ we will say that $\FR(\G , \L)$ is \emph{$S$-graded}.
\end{defn} 

\begin{rem} 
The class of algebras $\FR(\G , \L)$ defined above consists of Banach algebra Hausdorff completions of $\mathfrak{C}_c(\G , \L)$ in seminorms $\|\cdot\|_{\RR}$ for which there exists a unital inverse subsemigroup $S \subseteq \Bis(\G)$ covering $\G$ such that $\|\cdot \|_{\RR}\leq \|\cdot\|_{\infty}$ on every $C_c(U,\LL)$, $U\in S$ (this is equivalent to $\|\cdot\|_{\RR}\leq \|\cdot \|_{\textrm{max}}^S$ on $\mathfrak{C}_c(\G , \L)$). 
Without loss of generality, we may always assume here that $S$ is closed under taking subsets, so that in particular it is wide, and that $S\subseteq S(\LL)$, see \cite[Remark 4.7]{BKM}.
\end{rem}


\begin{ex}[Hahn's completions]
The domain and range maps $d,r : \G \to X$ induce the following  norms on $\mathfrak{C}_c(\G , \L)$ dominated by $\| \cdot \|_{\textrm{max}}$:  
\[
    \| f \|_{*d} := \max_{x \in X} \sum_{d(\gamma) = x} \big| f(\gamma) \big| , \qquad \| f \|_{*r} := \max_{x \in X} \sum_{r(\gamma) = x} \big| f(\gamma) \big| , \qquad \| f \|_I := \max \{ \| f \|_{*d} , \| f \|_{*r} \} .
\]
The completions of $\mathfrak{C}_c(\G , \L)$ in these norms are denoted by $\Falg{\G , \L}{}{*d}{}$, $\Falg{\G , \L}{}{*r}{}$, and $\Falg{\G , \L}{}{I}{}$ respectively, and they are $\Bis(\G)$-graded. 
The above norms, especially $\|\cdot \|_{I}$, are common in the literature, cf.\ \cite[Section 2.2]{Paterson} or \cite[II, 1.4]{Renault_book}, 
and $\|\cdot \|_{I}$ is  often called \emph{Hahn's $I$-norm}, as it was introduced in~\cite{Hahn}. 
Thus we call $\Falg{\G , \L}{}{I}{}$ \emph{Hahn's algebra}. 
The algebras $\Falg{\G , \L}{}{*d}{}$ and $\Falg{\G , \L}{}{*r}{}$ coincide respectively with the groupoid $L^1$ and $L^\infty$-operator algebras that are introduced in the next example. 
The identity on $\mathfrak{C}_c(\G , \L)$ extends to representations (with dense ranges)
\[
    \Falg{\G , \L}{}{}{} \to \Falg{\G , \L}{}{I}{} \to \Falg{\G , \L}{}{*d}{} , \Falg{\G , \L}{}{*r}{} . 
\]
\end{ex}

\begin{ex}[Full groupoid $L^p$-operator algebras] \label{ex:FullLpAlgebra}
Let $p \in [1 , \infty]$. 
We recall (see e.g.\ \cite{Phillips}) that a Banach algebra $A$ is an \emph{$L^p$-operator algebra} if there is an isometric representation $\psi : A \to B(L^p(\mu))$, where $L^p(\mu)$ is the Banach $L^p$-space associated to some measure $\mu$. 
Following \cite[Definition 5.12]{BKM}, we define the \emph{full $L^p$-operator algebra of $(\G , \LL)$} as $F^p(\G , \L):= \FR(\G , \L)$ where $\RR$ consists of representations $\psi : \FS(\G , \L) \to B(E)$, for some unital inverse subsemigroup $S \subseteq \Bis(\G)$ which covers $\G$, such that 
\begin{enumerate}[label={(R\arabic*)}]
    \item\label{enu:full_Lp1} $E = L^p(\mu)$ for a measure $\mu$; 
    \item\label{enu:full_Lp2} operators in $\psi(C_c(X)^+)$ preserve the cone of positive functions in $E$. 
\end{enumerate}
The following comments hold by \cite[Theorem 5.13]{BKM}. 
The above definition does not depend on the choice of $S$, and so $F^p(\G , \L)$ is $\Bis(\G)$-graded. 
In fact, denoting by $\| \cdot \|_{L^p}$ the norm in $F^p(\G , \L)$, for $f\in \mathfrak{C}_c(\G,\LL)$ and $q\in [1,\infty]$ with $1/p+1/q=1$ (we put $q=1$ if $p=\infty$ and vice versa) we get 
\begin{equation}\label{eq:L_p_norm_estimates}
    \| f \|_{L^p} \leq \| f \|_{*d}^{1/p} \, \| f \|_{*r}^{1/q} \leq \|f\|_{I}.
\end{equation}
In the case $\F=\C$, condition \ref{enu:full_Lp2} can be dropped (we get the same norm without it). 
Thus
\[
    \| f \|_{L^p} = \sup\{ \| \psi(f) \| : \text{$\psi$ is a representation of $F_{I}(\G,\LL)$ into an $L^p$-operator algebra} \} 
\]
in the complex case.  
In the real case trouble may come from representations as in \cite[Example 2.14]{BKM}. 
For $\F=\R,\C$ and $p=2$, $\| \cdot \|_{L^2}$ is the maximal submultiplicative norm on $\mathfrak{C}_c(\G,\LL)$ satisfying the $C^*$-equality. 
Thus $C^*(\G,\LL) := F^2(\G , \L)$ is a (real or complex) $C^*$-algebra, which is universal in the sense that $*$-homomorphisms from $\mathfrak{C}_c(\G,\LL)$ into a $C^*$-algebra $B$ extend to representations of $C^*(\G,\LL)$ in $B$, and every representation of $C^*(\G,\LL)$ in $B$ is automatically $*$-preserving. 
The norm  $\| \cdot \|_{L^{\infty}}$ does not change, if  we replace the $L^{\infty}$-space in \ref{enu:full_Lp1} by $E=C_0(\Omega)$ for some locally compact Hausdorff space. 
Thus $F^{\infty}(\G , \L)$ could also be viewed as the \emph{full $C_0$-operator algebra} of $(\G , \LL)$. 
Moreover, if $q\in [1,\infty]$ satisfies $1/p+1/q=1$, then the involution in $\mathfrak{C}_c(\G,\LL)$ extends to an isometric anti-isomorphism $F^p(\G , \L) \anti F^q(\G , \L)$. 
\end{ex}

\begin{ex}[Reduced groupoid $L^p$-operator algebras] \label{ex:reduced_L_p_groupoid_algebra}
Let $p \in [1 , \infty]$. 
For any section $f$ of $\LL$ we put $\| f \|_{p} = (\sum_{\gamma\in \G} | f(\gamma) |^{p})^{1/p}$ when $p<\infty$ and $\|f\|_{\infty} = \sup_{\gamma\in \G}|f(\gamma)|$ when $p=\infty$. 
We denote by $\ell^p(\G,\LL)$ the Banach space of all sections of $\LL$ for which the norm $\| f \|_{p}$ is finite. 
The formula
\begin{equation}\label{eq:regular_for_p}
    \Lambda_p(f) \xi (\gamma) := (f * \xi)(\gamma) = \sum_{r(\eta) = r(\gamma)} f(\eta)\cdot \xi(\eta^{-1} \gamma), \qquad f \in \mathfrak{C}_c(\G,\LL),\ \xi \in \ell^{p}(\G,\LL) ,
\end{equation}
defines a representation $\Lambda_p : F^p(\G , \LL) \to B(\ell^{p}(\G,\LL))$ which is injective on $\mathfrak{C}_c(\G,\LL)$, see \cite[Proposition 5.1]{BKM}. 
We call $F_{\red}^p(\G , \LL):= F_{\red}^{ \{ \Lambda_p\} }(\G , \LL)$ the \emph{reduced $L^p$-operator algebra of $(\G,\LL)$} and we denote the norm by $\|\cdot\|_{L^p , \red}$. 
By \cite[Theorem 5.13]{BKM} we have
\begin{equation}\label{eq:L_1_L_infty_are_reduced}
    F^1(\G , \LL) = F_{\red}^1(\G , \LL) = \Falg{\G , \LL}{}{*d}{} \quad\text{ and } \quad F^{\infty}(\G , \LL) = F^{\infty}_{\red}(\G , \LL) = F_{*r}(\G , \LL). 
\end{equation}
By \cite[Theorem 6.19]{Gardella_Lupini17}, if $\G$ is second countable and amenable, then $F^p(\G , \LL) = F_{\red}^p(\G , \LL)$ for every $p\in[1,\infty]$.
\end{ex}

\begin{rem} 
When $\F = \C$, the reduced algebra $F_{\red}^p(\G , \LL)$ was studied in \cite{Hetland_Ortega}, and in the untwisted case in \cite{cgt}. 
The full $L^p$-operator algebra $\Falg{\G}{}{p}{}$ was introduced in a different but equivalent way by Gardella--Lupini~\cite{Gardella_Lupini17} in the complex, separable, finite $p$ case.
\end{rem}


%
%
\section{Reduced groupoid Banach algebras}
\label{sec:ReducedBanachAlgs}

We fix a twisted groupoid $(\G,\LL)$. 
In this section, we will consider completions of $\mathfrak{C}_c(\G , \L)$ that are equipped with `generalized conditional expectations' that extend the map $E_X : \mathfrak{C}_c(\G , \L)\to \B_c(X)$, given by restriction $f\mapsto f|_X$. 
When $\G$ is Hausdorff, then $E_X : C_c(\G , \L) \to C_c(X)\subseteq C_c(\G , \L)$ is a projection. 

\begin{defn}\label{de:GroupoidBanachAlgebrasExpectations}
A \emph{groupoid Banach algebra} of $(\G , \L)$ is a  Banach algebra completion $\FR(\G , \LL)$ of $\mathfrak{C}_c(\G , \L)$  
such that $C_0(X)$ canonically embeds into $\FR(\G , \LL)$ and there is a contractive map $E_{X}^{\RR} : \FR(\G , \LL) \to \B_0(X)$ that extends $E_X$. 
We say that  $\FR(\G , \LL)$ is \emph{reduced} if $E_{X}^{\RR}$ is faithful in the sense that the only ideal contained in $\ker E_{X}^{\RR}$ is $\{0\}$.
\end{defn}

\begin{rem} 
When $\G$ is Hausdorff a groupoid Banach algebra  is a completion of $\mathfrak{C}_c(\G , \L)$ in a submultiplicative seminorm $\|\cdot\|_{\RR}$ such that  $\| f|_{X} \|_{\infty} = \| f|_{X} \|_{\RR} \leq \|f\|_{\RR}$ for all $f\in C_c(\G , \L)$. 
\end{rem}

\begin{rem} 
For any unital inverse subsemigroup $S\subseteq \Bis(\G)$ covering $\G$, the algebra $\FS(\G , \LL)$ is a groupoid Banach algebra. 
For any $p \in [1 , \infty]$, the $L^p$-operator algebras $F^p(\G , \LL)$ and $F^p_{\red}(\G , \LL)$ are  groupoid Banach algebras and $F^p_{\red}(\G , \LL)$ is reduced, cf.\ Proposition~\ref{prop:regular_disintegrated} below. There is a unique $C^*$-norm $\|\cdot\|_{C^*_{\red}}$ on $\mathfrak{C}_c(\G,\LL)$ yielding a reduced Banach algebra, and then this algebra coincides with the standard reduced $C^*$-algebra $C^*_{\red}(\G , \L)$ of $(\G , \L)$ (or its real counterpart).
The full $C^*$-algebra $C^*(\G , \L)$, which is a completion of $\mathfrak{C}_c(\G,\LL)$ in the maximal $C^*$-norm $\|\cdot\|_{C^*\max}$, is a groupoid Banach algebra of $(\G , \L)$. 
Thus a  $C^*$-norm $\|\cdot\|_{\RR}$ on $\mathfrak{C}_c(\G,\LL)$ yields a groupoid Banach algebra if and only if it satisfies $\|\cdot\|_{C^*_{\red}} \leq \|\cdot\|_{\RR}\leq \|\cdot\|_{C^*\max}$. 
\end{rem}
 
We make some preparations to reveal more structure in groupoid Banach algebras. 
For any $U\in \Bis(\G)$, we view the space of bounded Borel sections $\B(U,\LL)$ as a Banach $C(X)$-bimodule:
\[
    (a  * f ) (\gamma) := a(r(\gamma)) \cdot f(\gamma), \qquad  (f  * b) (\gamma) := f(\gamma) \cdot b(d(\gamma)), \qquad a,b \in C(X),\ f\in \B(U,\LL).
\]
It contains $\B_c(U,\LL)$ and $\B_0(U,\LL)$ as non-degenerate $C_0(X)$-subbimodules.
The restriction 
\[
    E_{U}(f) := f|_{U}, \qquad f\in \mathfrak{C}_c(\G , \L),
\] 
gives a $C_0(X)$-bilinear map $E_{U} : \mathfrak{C}_c(\G , \L)\to \B(U,\LL)$. 
The map $E_U$ on $\mathfrak{C}_c(r(U)\G , \L)=C_c(r(U)) * \mathfrak{C}_c(\G , \L)$, takes values in compactly supported Borel sections $\B_c(U,\LL) = C_c(r(U))* \B(U,\LL)$. 
When $\G$ is Hausdorff $E_{U} : \mathfrak{C}_c(r(U)\G , \L) \to C_c(U,\LL)$ takes values in continuous functions. 
Indeed, $\mathfrak{C}_c(r(U)\G , \L)$ is spanned by elements $f\in C_c(V,\LL)$ for some $V\in \Bis(\G)$ with $r(V)\subseteq r(U)$, and if $\G$ is Hausdorff then $V\cap U$ is closed (and open) in $V$, and hence $f|_{U}$ is continuous for every $f\in C_c(V,\LL)$. 

We denote by $\mathfrak{C}_0(\G , \L)$ the closure of $\mathfrak{C}_c(\G , \L)$ in the supremum norm $\| \cdot \|_{\infty}$ in the space $\B(\G , \L)$ of bounded Borel sections from $\G$ to $\L$. 
Thus 
\[
    \mathfrak{C}_0(\G , \L) = \overline{\mathfrak{C}_c(\G , \L)}^{\|\cdot\|_{\infty}}
\] 
is a Banach space, which coincides with $C_0(\G , \L)$ when $\G$ is Hausdorff. 
The involution on $\mathfrak{C}_c(\G , \L)$ gives a $\|\cdot\|_{\infty}$-isometric involution on $\mathfrak{C}_0(\G , \L)$, but the convolution product might not be defined on the whole of $\mathfrak{C}_0(\G , \L)$. 
Nevertheless we have:

\begin{lem}\label{lem:bimodule_borel}
Convolution defines a $\mathfrak{C}_c(\G , \L)$-bimodule structure on $\B(\G , \L)$  which is continuous in the sense that if a sequence $( f_n )_{n=1}^\infty$ in $\B(\G , \L)$ converges to $f\in \B(\G , \L)$ in $\|\cdot\|_{\infty}$, then, for any $g\in \mathfrak{C}_c(\G , \L)$, we have $g * f_n \to g*f$ and $f_n*g\to f*g$ in $\|\cdot\|_{\infty}$. 
It restricts to a continuous   $\mathfrak{C}_c(\G , \L)$-bimodule structure on $\mathfrak{C}_0(\G , \L)$.
\end{lem}
\begin{proof}
Fix $g \in C_c(U^* , \L)$, where $U\in \Bis(\G)$. 
For any $f \in \B(\G , \L)$, we have
\begin{equation}\label{equ:bisection_left_module_formula}
   g * f(\gamma) = \begin{cases}
                       g \big( \gamma \tilde{h}_{U}(\gamma)^{-1})\cdot  f (\tilde{h}_{U}(\gamma) \big),  & \gamma \in r^{-1}(d(U)), \\
                      0,  & \gamma \not\in r^{-1}(d(U)),
                   \end{cases}
\end{equation}
where $\tilde{h}_U: r^{-1}(d(U))\to r^{-1}(r(U))$ is the homeomorphism given by $\tilde{h}_U(\gamma) := d|_{U}^{-1}(r(\gamma))\gamma$ (see \cite[Example 2.19]{BKM}). 
This implies that $g * f\in \B(\G , \L)$ and  $\|g * f\|_{\infty}\leq \|g\|_{\infty}\cdot \|f\|_{\infty}$. 
Thus $\|g*f_n- g*f\|_{\infty}\to 0$ whenever $\|f_n- f\|_{\infty}\to 0$. 
Since elements in $C_c(U , \L)$, $U\in \Bis(\G)$, span $\mathfrak{C}_c(\G , \L)$, this proves that $\B(\G , \L)$ is a left $\mathfrak{C}_c(\G , \L)$-module continuous in the desired sense. 
Similarly, one checks that $\B(\G , \L)$ is a continuous right $\mathfrak{C}_c(\G , \L)$-module, and so a bimodule by associativity of convolution. 
Since $\mathfrak{C}_c(\G , \L)$ is dense in $\mathfrak{C}_0(\G , \L)$ and closed under convolution, we conclude that $\mathfrak{C}_c(\G , \L)*\mathfrak{C}_0(\G , \L)$, $\mathfrak{C}_0(\G , \L)* \mathfrak{C}_c(\G , \L)\subseteq \mathfrak{C}_0(\G , \L)$. 
\end{proof}

\begin{lem}\label{lem:approx_conditional_exp_twisted}
For any $U\in S(\LL)\subseteq \Bis(\G)$ and $f\in C_c(r(U))*\mathfrak{C}_0(\G , \L)$, there is a norm one section $b\in C_c(U^{*} , \L)$ such that $|(b * f)|_X| = |f\circ s|_{U}^{-1}|$ as functions on $X$ (the modulus can be dropped in the untwisted case).
\end{lem}
\begin{proof}
The range  $r(K)$ of the closed support $K$ of $f\in C_c(r(U))*\mathfrak{C}_0(\G , \L)$ is a compact subset of $r(U)$. 
Thus there is  a positive norm one function $a\in C_c(r(U))$ which is $1$ on $r(K)$. Taking any unitary section $c\in C_u(U^{*},\LL)$ gives the desired $b\in C_c(U^* , \L)$ by putting $b(\gamma) := a( s(\gamma) )c(\gamma)$ for $\gamma \in U^{*}$.  
\end{proof}

The essence of the following proposition is that every groupoid Banach algebra $\FR(\G , \LL)$ is equipped with a natural contractive map into $\mathfrak{C}_0(\G , \L)$ and this map is injective 
if and only if  $\FR(\G , \LL)$ is reduced. Also we could formally weaken the definition of  groupoid Banach algebras using   Hausdorff completions of $\mathfrak{C}_c(\G , \L)$, rather then completions.

\begin{prop}\label{prop:expectations_and_j_map}
Let $\FR(\G , \LL)$ be a Banach algebra Hausdorff completion of $\mathfrak{C}_c(\G , \L)$ and let $i_{\RR}: \mathfrak{C}_c(\G , \L)\to \FR(\G , \LL)$ be the canonical homomorphism. 
Assume $i_{\RR}$ is isometric on $C_c(X)$, so that $C_0(X)$ embeds isometrically into $\FR(\G , \LL)$. 
The following conditions are equivalent: 
\begin{enumerate}
    \item\label{enu:expectations_and_j_map0} $\| f|_{X} \|_{\infty} \leq \|i_{\RR}(f)\|_{\RR}$ for all $f\in \mathfrak{C}_c(\G , \L)$; 
    \item\label{enu:expectations_and_j_map1} $\FR(\G , \LL)$ is a groupoid Banach algebra of $(\G,\LL)$ (in particular $i_{\RR}$ is injective); 
    \item\label{enu:expectations_and_j_map2} $\FR(\G , \LL)$ is a completion of $\mathfrak{C}_c(\G , \L)$ and for any $U\in \Bis(\G)$ there is a contractive linear map $E^\RR_U : \FR(\G , \LL)\to \B(U,\L)$ that extends $E_U : \mathfrak{C}_c(\G , \L) \to \B(U,\L)$; 
    \item\label{enu:expectations_and_j_map3} $\FR(\G , \LL)$ is a completion of $\mathfrak{C}_c(\G , \L)$ such that the inclusion $\mathfrak{C}_c(\G,\LL) \subseteq \mathfrak{C}_0(\G,\LL)$ extends to a contractive map $j_{\RR} :\FR(\G , \LL) \to \mathfrak{C}_0(\G,\L)$. 
\end{enumerate}
Assume the above equivalent conditions.  
For each $U\in \Bis(\G)$, $E^\RR_{U}$ is a $C_0(X)$-module map given by $E^\RR_{U}(f) = j_{\RR}(f)|_{U}$, $f\in \FR(\G , \LL)$,  $j_{\RR}$ is a $\mathfrak{C}_c(\G , \L)$-bimodule map and for any family $S\subseteq \Bis(\G)$ that covers $\G$ we have
\[
    \ker j_{\RR}=\bigcap_{U\in S} \ker E_{U}^\RR = \{ f\in \FR(\G , \LL): \text{$E^\RR_U(af) = 0$  for all $a\in C_c(r(U)),\ U\in S$} \}.
\]
Moreover,  $\ker j_{\RR}$  is the largest  ideal in $\FR(\G , \LL)$ contained in $\ker E_{X}^\RR$. 
\end{prop}
\begin{proof} 
\ref{enu:expectations_and_j_map3} implies \ref{enu:expectations_and_j_map2} by putting $E^\RR_{U}(f) := j_{\RR}(f)|_{U}$ for $f\in \FR(\G , \LL)$ and $U\in \Bis(\G)$. 
\ref{enu:expectations_and_j_map2} implies \ref{enu:expectations_and_j_map1} because $X\in \Bis(\G)$.
\ref{enu:expectations_and_j_map1}$\Rightarrow$\ref{enu:expectations_and_j_map0} is trivial. 
Assume \ref{enu:expectations_and_j_map0}. 
We need to show \ref{enu:expectations_and_j_map3}.
For any $U\in S(\LL)$ and $f\in \mathfrak{C}_c(r(U)\G , \L)$, taking $b$ as in Lemma~\ref{lem:approx_conditional_exp_twisted} we get
\[
   \|f|_{U}\|_{\infty} = \| f\circ s|_{U}^{-1} \|_{\infty} = \| (b * f)|_X\|_{\infty} = \| E_{X}^{\RR} (b *f) \|_{\infty} \leq \|i_{\RR}(b *f) \|_{\RR} \leq \|i_{\RR}(f) \|_{\RR} .
\]
Thus there is a linear contractive map $E^\RR_U : C_0(r(U)) \FR(\G , \LL) \to \B_0(U,\L)$ with $E^\RR_U (i_{\RR}(f)) = f|_{U}$ for $f\in \mathfrak{C}_c(r(U)\G , \L)$. 
Let $f\in \mathfrak{C}_c(\G,\LL)$. 
For any $\varepsilon >0$ there is $\gamma\in \G$ such that $\|f\|_{\infty}-\varepsilon < | f(\gamma) |$. 
Take $U\in S(\LL)$ with $\gamma\in U$ and norm one $a\in C_c(r(U))$ with $a( r(\gamma) ) = 1$. 
Then 
\[
    | f(\gamma) | = |(a* f)(\gamma)| = | E^\RR_{U}(i_{\RR}(a* f))(\gamma)| \leq \|E^\RR_{U}(i_{\RR}(a* f))\|_{\infty}\leq  \|i_{\RR}(a* f)\|_{\RR}\leq \| i_{\RR}(f)\|_{\RR}.
\]
This implies that $\|f\|_{\infty}\leq \| i_{\RR}(f)\|_{\RR}$, which in turn implies \ref{enu:expectations_and_j_map3}. 
Thus \ref{enu:expectations_and_j_map0}--\ref{enu:expectations_and_j_map3} are equivalent. 

Now assume that \ref{enu:expectations_and_j_map0}--\ref{enu:expectations_and_j_map3} hold.
To show that $j_{\RR}$ is a $\mathfrak{C}_c(\G , \L)$-bimodule map fix $f \in \FR(\G , \LL)$ and $g \in \mathfrak{C}_c(\G , \L)$. 
Choose a sequence $( f_n )_{n=1}^\infty$ in $\mathfrak{C}_c(\G , \L)$ converging to $f$ in $\| \cdot \|_{\RR}$. 
Then $( f_n )_{n=1}^\infty$ converges to $j_{\RR}(f)$ in $\| \cdot \|_{\infty}$ by continuity of $j_{\RR}$. 
Hence also $g * f_n \to g * j_{\RR}(f)$ in $\|\cdot\|_{\infty}$ by Lemma~\ref{lem:bimodule_borel}.
Thus $g * j_{\RR}(f) = \lim_{n\to \infty} g * f_n = \lim_{n\to \infty} j_{\RR} ( g f_n) = j_{\RR}(g f)$. 
Symmetric considerations give $j_{\RR}(f g) = j_{\RR}(f) * g$.
This in particular implies that $\ker j_{\RR}$ is an ideal in  $\FR(\G , \LL)$  and that $E^\RR_{U}$ is a $C_0(X)$-module map, for any $U\in \Bis(\G)$.
For any $S \subseteq \Bis(\G)$ covering $\G$, the equality $\ker j_{\RR}=\bigcap_{U\in S} \ker E^\RR_{U}$ is clear as $E^\RR_{U}$ is the restriction of $j_{\RR}$ to $U$.
Since  $E^\RR_U$ is a continuous $C_0(X)$-module map, we also get  $\ker E^\RR_{U} = \{f\in \FR(\G , \LL) : \text{$E^\RR_U(af)=0$ for all $a\in C_c(r(U)), \ U\in S$} \}$. This proves the displayed equalities.

Finally, assume that $I$ is an ideal in $\FR(\G , \LL)$ contained in $\ker E_{X}^\RR$, and let $f\in I$. 
For every $a\in C_c(r(U))$ with $U\in S(\LL)$, Lemma~\ref{lem:approx_conditional_exp_twisted} applied to $a*j_{\RR}(f)= j_{\RR}(af)\in C_c(r(U))*\mathfrak{C}_0(\G , \L)$ gives $b\in C_c(U^* , \L)$ such that $\|E^\RR_{X}(baf)\|_{\infty}=\|E^\RR_{U}(af)\|_{\infty}$. 
Since $baf\in I$, we have $E^\RR_{U}(af)=0$ and therefore $f\in \ker(j_{\RR})$ by the above description of $\ker(j_{\RR})$. Hence $I\subseteq \ker(j_{\RR})$.
\end{proof}

\begin{rem} 
The contractive map $j_{\RR} : \FR(\G , \LL) \to \mathfrak{C}_0(\G,\L)$, as in Proposition~\ref{prop:expectations_and_j_map}\ref{enu:expectations_and_j_map3}, in the context of $C^*$-algebras is often called \emph{Renault's $j$-map}, and it is well known that it is injective on the reduced groupoid $C^*$-algebra, cf.\ \cite[Proposition II.4.2]{Renault_book}, \cite[Proposition 2.8]{BFPR}, \cite[Proposition 7.10]{Kwa-Meyer} and \cite{DWZ}. In our more general context we will refer to it as the \emph{$j$-map} for $\FR(\G , \LL)$. 
\end{rem}

\begin{rem}\label{rem:Fourier_decomposition} 
In view of Proposition~\ref{prop:expectations_and_j_map}, we see that a reduced groupoid Banach algebra of $(\G , \LL)$ is a Banach algebra $\FR(\G , \LL)$ completion of $\mathfrak{C}_c(\G,\L)$ such that the inclusion $\mathfrak{C}_c(\G,\LL) \subseteq \mathfrak{C}_0(\G,\LL)$ extends to an \emph{injective contraction} $j_{\RR} :\FR(\G , \LL) \to \mathfrak{C}_0(\G,\LL)$ which is isometric on $C_0(X)$.
Equivalently, $\FR(\G , \LL)$ is a groupoid Banach algebra such that, fixing any $S\subseteq \Bis(\G)$ covering $\G$,  every $f\in \FR(\G , \LL)$ is uniquely determined by elements $\{E^\RR_U(f)\}_{U\in S}$ where $E^\RR_U : \FR(\G , \LL) \to \B(U,\L)$ is a contraction that extends $E_U$, $U\in S$. 
In fact, by the displayed equality in Proposition~\ref{prop:expectations_and_j_map}, elements in $\FR(\G , \LL)$ are determined by the restricted maps $E^\RR_U : C_0(r(U))\FR(\G , \LL) \to \B_0(U,\L)$, $U\in S$. 
These restricted maps take values in $C_0(U,\LL)$ when $\G$ is Hausdorff. 
In the context of Fell bundles over inverse semigroups, such maps were considered in \cite[Proposition 2.18]{KwaMeyer}. 
Assume that $S\subseteq S(\LL)$ and for each  $U\in S$  fix a unitary section that trivializes the bundle $\LL|_{U}$.  Then using the homeomorphism $r : U \to r(U)$, we get an isomorphism $\Phi_{U}:\B_0(U,\L)\stackrel{\cong}{\longrightarrow} \B_0(r(U))$ that restricts to $C_0(U,\L)\cong C_0(r(U))$. 
Thus if in addition $\G$ is Hausdorff, we get maps 
\[
    \Phi_{U}\circ E^\RR_U : C_0(r(U))\FR(\G , \LL) \to C_0(r(U))\subseteq C_0(X)
\]
that take values in continuous functions on $X$ and that distinguish elements in $\FR(\G , \LL)$. 
In this way, one recovers the \emph{Fourier decomposition maps} considered in \cite[Definition 4.4]{BK} for Banach algebra crossed products. 
In particular, our notion of a reduced algebra is consistent with the definition of a \emph{reduced Banach algebra crossed product} from \cite{BK}, and a \emph{reduced group algebra} from \cite{Phillips19}. 
\end{rem}


\begin{cor}\label{co:QuotientIsReduced} 
For any groupoid Banach algebra $\FR(\G , \LL)$ the quotient Banach algebra $F_{\RR,\red}(\G , \LL) := \FR(\G , \LL)/\ker j_{\RR}$ is naturally a reduced groupoid Banach algebra of $(\G , \L)$. 
\end{cor}
\begin{proof} 
The contraction $j_{\RR} :\FR(\G , \LL) \to \mathfrak{C}_0(\G,\LL)$ factors to an injective contraction $j_{\RR}^{\red} :F_{\RR,\red}(\G , \LL) \to \mathfrak{C}_0(\G,\LL)$.
Since $\ker j_{\RR}\cap \mathfrak{C}_c(\G , \L) = \{ 0 \}$ we may view $\FR(\G , \LL)$ as a completion of $\mathfrak{C}_c(\G , \L)$. 
Since the quotient map $\FR(\G , \LL)\to F_{\RR,\red}(\G , \LL)$ is injective and contractive on $C_0(X)$, it is isometric on $C_0(X)$, by minimality of the $C^*$-norm.
\end{proof}

\begin{rem} 
If $\FR(\G , \LL)$ is a \emph{groupoid Banach $*$-algebra}, meaning that the involution on $\mathfrak{C}_c(\G,\L)$ is isometric in $\|\cdot\|_{\RR}$, then the map $j_{\RR}$ is $*$-preserving and both the ideal $\ker j_{\RR}$ and the quotient reduced algebra $F_{\RR,\red}(\G , \LL)=\FR(\G , \LL) / \ker j_{\RR}$ are Banach $*$-algebras. 
\end{rem}


\begin{cor}\label{cor:kernel_of_reps_with_generalized_expectation} 
Let  $\FR(\G , \LL)$ be a groupoid Banach algebra and let  $\psi :\FR(\G , \LL)\to B$  be a representation such that $\|f|_X\|_{\infty} \leq \|\psi(f)\|$ for $f\in  \mathfrak{C}_c(\G,\LL)$. Then $\ker \psi \subseteq \ker j_{\RR}$ and  so $\psi$ is injective on $\mathfrak{C}_c(\G,\LL)$ and if   $\FR(\G , \LL)$ is reduced then $\psi$ is injective on $\FR(\G , \LL)$. 
\end{cor}
\begin{proof}
By the assumed inequality we may treat $F^{\psi}(\G , \LL) := \overline{\psi(\FR(\G , \LL))}$ as a groupoid Banach algebra. 
Denoting by $j_{\psi}$ the corresponding map from Proposition~\ref{prop:expectations_and_j_map}\ref{enu:expectations_and_j_map3} we get $j_{\psi} \circ \psi = j_{\RR}$, which implies the assertion. 
\end{proof}

The following lemma generalizes \cite[Lemma 4.10]{BK}. 
It allows us to produce (reduced) groupoid Banach algebras and Banach-$*$ algebras from other (reduced) groupoid Banach algebras.

\begin{lem}\label{lem:reduced_from_other_reduced}
Let $\|\cdot \|_{\RR}$ and $\|\cdot\|_{\RR_{i}}$, $i\in I$, be norms on $\mathfrak{C}_c(\G , \L)$ that define groupoid Banach algebras $\FR(\G , \LL)$ and $F_{\RR_i}(\G , \LL)$, $i\in I$, for $(\G , \L)$. 
Then the formulas 
\begin{equation}\label{eq:R_norms_from_other_norms}
    \| f \|_{\RR^*} := \|f^*\|_{\RR}, \qquad \|f\|_{\{ \RR_i\}_{i}} := \sup_{i\in I}\|f\|_{\RR_i}, \qquad f\in \mathfrak{C}_c(\G,\LL),
\end{equation}
 define norms that yield groupoid Banach algebras $F_{\RR^*}(\G , \LL)$, $F_{ \{ \RR_i \}_{i}}(\G , \LL)$ for $(\G , \L)$, provided $\|\cdot\|_{\{ \RR_i\}_{i}}$ is finite.  
In particular, the norm $\|f\|_{\RR,*} := \max \{ \|f\|_{\RR} , \|f^*\|_{\RR} \}$ defines a groupoid algebra $F_{\RR , *}(\G , \LL)$ which is a Banach $*$-algebra.
Moreover, $F_{\RR^*}(\G , \LL)$ and $F_{\RR,*}(\G , \LL)$ are reduced when $\FR(\G , \LL)$ is; and $F_{ \{ \RR_i \}_{i}}(\G , \LL)$ is reduced if and only if all $F_{\RR_i}(\G , \LL)$, $i\in I$, are reduced.
\end{lem}
\begin{proof}
The involution on $\mathfrak{C}_c(\G , \L)$ extends to an antimultiplicative antilinear isometry from $F_{\RR^*}(\G , \LL)$ onto $\FR(\G , \LL)$. 
In particular, $\|f|_{X}\|_{\infty} = \| f^*|_{X} \|_{\infty} \leq \|f^*\|_{\RR} = \|f\|_{\RR^*}$ for $f\in \mathfrak{C}_c(\G , \L)$. 
Hence $F_{\RR^*}(\G , \LL)$ is a groupoid Banach algebra. 
Moreover, $j_{\RR^*}(f) = j_{\RR}(f^*)^*$ for $f\in \Falg{\G , \LL}{}{\RR^*}{}$.
Thus $j_{\RR^*}$ is injective if and only if $j_{\RR}$ is, equivalently $\Falg{\G , \LL}{}{\RR^*}{}$ is reduced if and only $\Falg{\G , \LL}{}{\RR}{}$ is. 

Assuming it is finite,  $\|\cdot \|_{\{ \RR_i\}_{i}}$ is clearly a submultiplicative norm on $\mathfrak{C}_c(\G , \L)$ which coincides with $\|\cdot\|_{\infty}$ on $C_c(X)$ and $\|f\|_{\infty}\leq \|f\|_{\{ \RR_i\}_{i}}$ for  $f\in \mathfrak{C}_c(\G , \L)$, as this holds for every $\|\cdot\|_{\RR_i}$, $i\in I$.
Thus $\Falg{\G , \LL}{}{\{ \RR_i \}_i }{}$ is a groupoid Banach algebra, and we have a canonical contraction $j_{\{ \RR_i\}_i} : \Falg{\G , \LL}{}{ \{ \RR_i \}_i }{} \to \mathfrak{C}_0(\G , \L)$. 
Note that $\pi(a) := \prod_{i\in I} a$, for $a\in \mathfrak{C}_c(\G , \L)$, determines an isometric embedding of $\Falg{\G , \LL}{}{\{ \RR_i \}_i}{}$ into the direct product $\prod_{i\in I} \Falg{\G , \LL}{}{\RR_i}{}$. 
Moreover, $\prod_{i\in I} j_{\{ \RR_i \}_i }= \prod_{i\in I}j_{\RR_i}\circ \pi$, as maps from $\Falg{\G , \LL}{}{\{ \RR_i \}_i}{}$ to $\prod_{i\in I} \mathfrak{C}_0(\G,\L)$. 
Hence $j_{\{ \RR_i\}_i}$ is injective if all $j_{\RR_i}$, $i\in I$, are injective. 
That is, $\Falg{\G , \LL}{}{\{ \RR_i \}_i}{}$ is reduced if all $\Falg{\G , \LL}{}{\RR_i}{}$, $i\in I$, are reduced. 
\end{proof}

\begin{rem} 
Let $\FR(\G , \L)$ be a groupoid Banach algebra which in addition is $S$-graded  by some unital inverse subsemigroup $S \subseteq \Bis(\G)$ which covers $\G$,
see Definition~\ref{defn:S_graded_algebras}. 
By Proposition~\ref{prop:expectations_and_j_map}\ref{enu:expectations_and_j_map2},
 the spaces $C_0(U,\LL)$, $U\in S$, forming the grading, embed isometrically into 
$\FR(\G , \LL)$. 
Also then the algebras $F_{\RR^*}(\G , \LL)$ and $F_{\RR,*}(\G , \LL)$ are $S$-graded, and if $F_{\RR_i}(\G , \LL)$, $i\in I$, are all $S$-graded groupoid Banach algebras, then also $F_{\RR^*}(\G , \LL)$ is $S$-graded. 
\end{rem}

Inspired by \cite[Definition 4.12]{BK}, we now generalize $L^p$-groupoid operator algebras to $L^P$-algebras where $P\subseteq [1,\infty]$ is a set of H\"older exponents.

\begin{defn}\label{def:F^P_crossed_products}
For any non-empty $P\subseteq [1,\infty]$ we put
\[
    F^P(\G,\LL) := \overline{\mathfrak{C}_c(\G,\LL)}^{\|\cdot\|_{L^P}}\quad \textrm{and} \qquad F^P_{\red}(\G,\LL) := \overline{\mathfrak{C}_c(\G,\LL)}^{\|\cdot\|_{L^P,\red}}, 
\]
where $\| f\|_{L^P} := \sup_{p\in P}\|f\|_{L^p}$ and  $\|f\|_{L^P,\red} := \sup_{p\in P}\|f\|_{L^p,\red}$, $f\in \mathfrak{C}_c(\G,\LL)$. 
We denote by $\Lambda_{P} : F^P(\G,\LL) \to F^P_{\red}(\G,\LL)$ the canonical representation (which is the identity on $\mathfrak{C}_c(\G , \L)$). 
\end{defn}

\begin{prop}\label{prop:regular_disintegrated}
For any non-empty $P\subseteq [1,\infty]$, $F^P_{\red}(\G,\LL)$ is a reduced groupoid Banach algebra which is $\Bis(\G)$-graded, and the corresponding $j$-map $j_P^{\red} : F^P_{\red}(\G,\LL) \to \mathfrak{C}_0(\G , \L)$ turns the product in $F^P_{\red}(\G,\LL)$ into the convolution: 
\begin{equation}\label{equ:j_map1} 
    j_P^{\red} ( f\cdot g)(\gamma) = \sum_{r(\eta) = r(\gamma)} j_P^{\red} (f)(\eta)\cdot j_P^{\red}(g)(\eta^{-1} \gamma), \qquad f,g\in F^P_{\red}(\G,\LL),
\end{equation}
and it preserves the involution in the sense that $j_P^{\red} (f)^* = j_{P^*}^{\red} (f^*)$, $f\in F^P_{\red}(\G,\LL)$, where $P^* := \{ q : 1/p + 1/q = 1, p\in P \}$. 
In particular, $F^P(\G,\LL)$ is a groupoid Banach algebra with the $j$-map $j_P := j_{P}^{\red}\circ \Lambda_{P}$ and so $\ker j_{P} = \ker \Lambda_{P}$. 
\end{prop}
\begin{proof} 
By Lemma~\ref{lem:reduced_from_other_reduced}, it suffices to consider the case when $P = \{ p \}$ is a singleton (in particular, $j_{P}^{\red} = \prod_{p\in P}j_{p}^{\red} \circ \pi$, where $\pi : F^P_{\red}(\G,\LL) \to \prod_{p\in P} F^p_{\red}(\G,\LL)$ is given by  $\pi(a) := \prod_{p\in P} a$, for $a\in \mathfrak{C}_c(\G , \L)$). 
Moreover, since involution on $\mathfrak{C}_c(\G , \L)$ gives $\Falg{\G , \LL}{}{\infty}{\red} \stackrel{\text{anti}}{\cong} \Falg{\G , \LL}{}{1}{\red}$ we may assume that $p < \infty$. 
Then the claim, using a different picture of the twist, in essence follows from \cite[Propositions 6.21, 6.24]{Hetland_Ortega}. 
For the sake of completeness and further reference, we give a complete proof based on the argument behind \cite[Proposition 2.8]{BFPR}.
Namely, we identify $F^p_{\red}(\G,\LL)$ with the corresponding subalgebra of $B(\ell^{p}(\G,\LL))$ using the extension of the map $\mathfrak{C}_c(\G,\LL) \to B(\ell^{p}(\G,\LL))$ given by \eqref{eq:regular_for_p}. 
For each $\gamma \in \G$, choose a norm one element $1_{\gamma}\in L_\gamma$ and treat it as a section of $\LL$ which is zero at $\eta \neq \gamma$. 
Then $\{1_\gamma\}_{\gamma\in \G}$ is a Schauder basis for $\ell^{p}(\G,\LL)$ and for any $f\in \mathfrak{C}_0(\G,\LL)$ we have $\| f \|_{\infty} = \sup_{\gamma\in \G} | ( f 1_{d(\gamma)} )(\gamma) | \leq \sup_{\gamma\in \G} \| f 1_{d(\gamma)} \|_p \leq \|f\|_{L^p,\red}$. 
Thus we have a contractive linear map $j_p^{\red} : F^p_{\red}(\G,\LL) \to \mathfrak{C}_0(\G,\LL)$ extending the inclusion $\mathfrak{C}_c(\G , \L)\subseteq \mathfrak{C}_0(\G,\LL)$. 
Clearly, for every $\gamma \in \G$ and any $f\in \mathfrak{C}_c(\G , \L)$,  we have
\begin{equation}\label{eq:j_map}
    |j_p^{\red}(f)(\gamma)| := |\big( f 1_{d(\gamma)} \big) (\gamma)|  
    \qquad\text{ and } \qquad
    \| f 1_{\gamma} \|_p = 
        \left(\sum_{d(\eta) = d(\gamma)} |j_p^{\red}(f)(\eta \gamma^{-1})|^p\right)^{1/p}.
\end{equation}
By continuity, these relations hold for any $f\in F^p_{\red}(\G,\LL)$. 
The second formula in \ref{eq:j_map} implies that $j_p^{\red}$ is injective on $F^p_{\red}(\G,\LL)$ and hence $F^p_{\red}(\G,\LL)$ is a reduced groupoid Banach algebra.

Since $^*$ yields the isometry $F^p_{\red}(\G,\LL) \stackrel{\text{anti}}{\cong} F^q_{\red}(\G , \LL)$, where $1/p + 1/q=1$, we get $j_p^{\red} (f)^* = j_{q}^{\red} (f^*)$, $f\in \Falg{\G , \LL}{}{p}{r}$. 
Now we prove \eqref{equ:j_map1}, with $P=\{p\}$. 
Fix $f,g\in F^p_{\red}(\G,\LL)$ and take nets $(f_i), (g_j)$ in $\mathfrak{C}_c(\G , \L)$ converging to $f,g$ in the norm of $F^p_{\red}(\G,\LL)$. 
Using H\"olders inequality, \eqref{eq:j_map}, and the isometry $F^p_{\red}(\G,\LL) \stackrel{\text{anti}}{\cong} F^q_{\red}(\G , \LL)$, for any $\gamma\in \G$, we get
\[
\begin{split}
     \sum_{r(\eta) = r(\gamma)} \left|j_p^{\red} (f_i - f)(\eta) \cdot j_p^{\red}(g_j)(\eta^{-1} \gamma) \right| 
	   \leq \| ( f_i - f )^* 1_{r(\gamma)} \|_{q} \| g_j 1_{\gamma} \|_p 
	   &\leq \| ( f_i - f )^* \|_{_{L^q,\red}} \| g_j \|_{_{L^p,\red}} \\
            &= \| f_i - f \|_{L^p,\red} \| g_j \|_{L^p,\red}.
\end{split}
\]
Similarly, we get $\sum_{r(\eta) = r(\gamma)}| j_p^{\red}(f) (\eta) \cdot j_p^{\red}( g_j - g ) (\eta^{-1} \gamma) | \leq \| f \|_{L^p,\red} \| g_j - g \|_{L^p,\red}$. 
Altogether, since $j_{p}^{\red}(f_i \cdot g_j)= f_i * g_j$, this implies that
\[
\begin{split}
    \left|j_{p}^{\red}(f_i \cdot g_j)(\gamma) - \sum_{r(\eta) = r(\gamma)} j_p(f)(\eta) \cdot j_p(g)(\eta^{-1} \gamma) \right| \leq \| f_i - f \|_{L^p,\red} \| g_j \|_{L^p,\red} +\| f \|_{L^p,\red} \| g_j - g \|_{L^p,\red} 
\end{split}
\]
tends to zero, which proves \eqref{equ:j_map1}, as $j_{p}^{\red}(f_i \cdot g_j)(\gamma)$ tends to $j_{p}^{\red}(f \cdot g)(\gamma)$. 
\end{proof}

\begin{rem}\label{rem:L_P_definition_discussion}
If $P = P^*$, then $F^P(\G,\LL)$ and $F^P_{\red}(\G,\LL)$ are Banach $*$-algebras and $\Lambda_{P}$ is a $*$-homomorphism, cf.\ \cite[Theorem 5.13(7)]{BKM}. 
For $P=\{p,q\}$ with $1/p+1/q=1$, the Banach $*$-algebras $F^P_{\red}(\G,\LL)$ are sometimes called \emph{symmetrized pseudo-function algebras}. 
They  were studied in \cite{Austad_Ortega} and in the group case in \cite{Elkiaer}, \cite{LiYu}, \cite{Phillips19}. 
When $P\subseteq [1,\infty]$ contains $\{1,\infty\}$, then $F^P(\G,\LL) = F^P_{\red}(\G,\LL) = F_{I}(\G,\LL)$, by \eqref{eq:L_p_norm_estimates} and \eqref{eq:L_1_L_infty_are_reduced}. 
In particular, $F_{I}(\G,\LL)$ is a reduced groupoid Banach algebra. 
In fact, we always have $F^P(\G,\LL) = F^P_{\red}(\G,\LL)$ for any $P\subseteq \{1,\infty\}$, and if $\G$ is second countable and amenable, then $F^P(\G,\LL) = F^P_{\red}(\G,\LL)$ for every $P\subseteq [1,\infty]$, see the last part of Example~\ref{ex:reduced_L_p_groupoid_algebra}. 
\end{rem}


%
%
\section{Essential and exotic groupoid Banach algebras}
\label{sec:Essential_and_exotic}

The reader interested only in Hausdorff groupoids may skip this section. 
When the groupoid $\G$ is non-Hausdorff we need to extend the previous definition of groupoid Banach algebras  to allow constructions that `disregard' the points of discontinuity for quasi-continuous functions in $\mathfrak{C}_0(\G , \L)$.  
In the $C^*$-algebraic context, this quotient was introduced in \cite{Kwa-Meyer} and examples go back (at least) to the work of Nekrashevych~\cite{Nekrashevych}, \cite{Nekrashevych2}. 
The crucial observation is that elements of $\mathfrak{C}_0(\G , \L)$ are continuous modulo \emph{meager sets} (that is, countable unions of nowhere dense sets).

\begin{lem}[{\cite[Lemma 7.13]{Kwa-Meyer}}]\label{lem:comeager}
For every $f\in \mathfrak{C}_0(\G , \L)$, there is a meager set $M\subseteq \G$ such that $f$ is continuous at every point $\gamma \in \G \setminus M$.
\end{lem}
\begin{proof}
When $f\in \mathfrak{C}_0(\G , \L)$ is supported on $U\in \Bis(\G)$, then $f$ can be discontinuous only on the nowhere dense set $\partial U \subseteq \G$.
Thus $f\in \mathfrak{C}_c(\G , \L)$ can be discontinuous only on a finite union of nowhere dense sets. 
Since every $f\in \mathfrak{C}_0(\G , \L)$ is a limit of a sequence of functions in $\mathfrak{C}_c(\G , \L)$, the set of discontinuity points of $f$ is contained in a countable union of nowhere dense sets. 
\end{proof}

\subsection{Quasi-continous functions with meager support} 
\label{ssec:QuasiContinuousMeagerSupport}

We will first discuss the space
\[
    \mathfrak{M}_0(\G , \L) := \{ f \in \mathfrak{C}_0(\G , \L): \text{$\supp(f)$ is meager} \} , 
\]
where by $\supp$ we mean the strict support, given by 
\[
    \supp(f) := \{ \gamma \in \G : f(\gamma)\neq 0 \}.
\]

We show that the elements in $\mathfrak{M}_0(\G , \L)$ have to be supported on a subgroupoid of $\G$ where Hausdorffness fails.
To this end, we adopt the nomenclature of Dixmier \cite{Dixmier}. 

\begin{defn}[Dixmier]\label{defn:dangerous} 
Let $Y$ be a topological space. We call $x\in Y$ a \emph{Hausdorff point}  in $Y$ if 
for any other point $y\in X$ there are disjoint open neighbourhoods of $x$ and $y$. 
We denote by $Y_{\Hau}$ the set of all Hausdorff points in $Y$.
\end{defn}

\begin{rem}
We  apply the above definition to the groupoid  $\G$. 
In particular, $\gamma\in \G\setminus \G_{\Hau}$ if there is $\eta \in \G \setminus \{ \gamma \}$ such that $\gamma$ and $\eta$ cannot be separated by disjoint open sets (equivalently, there is a net $(\gamma_i)_i$ in $\G$ converging to $\gamma$ and $\eta$). 
The authors of \cite{Buss_Martinez} call arrows in $ \G\setminus \G_{\Hau}$ \emph{dangerous}, which was inspired by \cite[Definition 7.14]{Kwa-Meyer} where the points in $X\setminus \G_H$  were given this name.
\end{rem}

Statements similar to those in the following lemma can also be found in \cite{Buss_Martinez}, which appeared on arXiv  after the first preprint of this article, and were proved independently. 

\begin{lem}\label{lem:singular_groupoid}
Both $\G_{\Hau}$  and  $\G\setminus \G_{\Hau}$ are full subgroupoids of $\G$ and so $\G$ decomposes into the groupoid disjoint union $\G=\G_{\Hau}\sqcup \G\setminus \G_{\Hau}$.
In particular, $d(\G_{\Hau})=r(\G_{\Hau})=X\cap \G_{\Hau}$ and  $d(\G\setminus \G_{\Hau})=r(\G\setminus\G_{\Hau})=X\setminus \G_{\Hau}$. 
Moreover,  
\begin{enumerate}
	\item \label{en:singular_groupoid1} for any  $x\in  X\setminus \G_{\Hau}$	the isotropy group $\G(x) := \{\gamma\in \G: d(\gamma)=r(\gamma)=x\}$   is non-trivial; 
	\item \label{en:singular_groupoid2} if $\G$ has a countable open cover consisting of bisections, then  $\G_{\Hau}$ is dense in $\G$
	(in fact $\G_{\Hau}$ is comeager and $\G$ is Baire);
	\item \label{en:singular_groupoid3}	every $f\in \mathfrak{C}_0(\G , \L)$ is continuous at every $\gamma\in \G_{\Hau}$. 
\end{enumerate}
\end{lem}
\begin{proof} 
Assume that a net  $(\gamma_i)_i$ in $\G$ converges towards two different elements $\gamma, \gamma'\in \G$. 
Then $(\gamma_i^{-1})_i$ converges towards two different elements $\gamma^{-1}, \gamma'^{-1}\in \G$, so $(\G\setminus\G_{\Hau})^{-1} = \G\setminus\G_{\Hau}$. 
Also $r(\gamma)=r(\gamma')$ and $d(\gamma)=d(\gamma')$, because $r,d$ are continuous and $X$ is Hausdorff. 
This in particular implies \ref{en:singular_groupoid1}, cf.\ \cite[Lemma 7.15]{Kwa-Meyer}.
Thus, if $\eta\in \G$ is composable with $\gamma$, it is also composable with $\gamma'$ and we may construct a net $(\eta_i)_i$ convergent to $\eta$ such that $\eta_i$ and $\gamma_i$ are composable, so the net $(\eta_i\gamma_i)_i$ converges to two different elements $\eta\gamma$ and $\eta\gamma'$. 
Indeed, fixing two bisections $U,V\in \Bis(\G)$ where $\eta\in V$, $\gamma\in U$ and $d(V)=r(U)$, we may assume that $(\gamma_i)_i$ is a net in $U$, and then the elements $\eta_i := d|_V^{-1}(r(\gamma_i))$ form the desired net. 
This shows that $\G  (\G\setminus\G_{\Hau}) \subseteq \G\setminus\G_{\Hau}$. 
By taking inverses, and using $(\G\setminus\G_{\Hau})^{-1} = (\G\setminus\G_{\Hau})$, we also get $(\G\setminus\G_{\Hau}) \G\subseteq (\G\setminus\G_{\Hau})$. 
Since arrows in the groupoid are invertible, these relations for $\G\setminus\G_{\Hau}$ are equivalent to saying that $\G\setminus\G_{\Hau}$  and  $\G_{\Hau}$ are full subgroupoids of $\G$.
 If $\G$ has a countable open cover consisting of bisections, then since open bisections are Baire spaces, $\G$ is a Baire space.
Also by \cite[Lemma 7.15]{Kwa-Meyer}, $\G\setminus \G_{\Hau}$ is meager, and therefore it has empty interior. This proves item \ref{en:singular_groupoid2}. Item \ref{en:singular_groupoid3} follows from (the proof of) \cite[Lemma 7.13]{Kwa-Meyer}. 
\end{proof}

\begin{rem}\label{rem:meagerness_of_singular} 
Recall that a groupoid $\G$ is called \emph{principal} if there are no points $x\in X$ with non-trivial isotropy, and $\G$ is called \emph{topologically principal} if the set of points with non-trivial isotropy has empty interior.
By Lemma~\ref{lem:singular_groupoid}\ref{en:singular_groupoid1}, if $\G$ is principal then $\G=\G_{\Hau}$ is Hausdorff. 
If $\G$ is topologically principal, then $ \G_{\Hau}$ is dense in $\G$. 
\end{rem}

The following proposition is a minor improvement of \cite[Proposition 7.18]{Kwa-Meyer}, cf.\ also \cite[Corollary 3.15]{Buss_Martinez}. 
Namely, in general we do not need any assumptions on $\G$ and we consider quasi-continuous functions that vanish at infinity.

\begin{prop} \label{prop:singular_functions}
For any $f\in \mathfrak{C}_0(\G , \L)$ the following conditions are equivalent:
\begin{enumerate}
    \item \label{en:singular_functions1} $\supp(f)$ is meager in $\G$, that is, $f\in \mathfrak{M}_0(\G , \L)$; 
    \item \label{en:singular_functions2} $\supp(f)$ has empty interior in $\G$;
    \item \label{en:singular_functions3} $\{ \gamma \in \G : | f(\gamma) | > \varepsilon \}$ has empty interior in $\G$ for all $\varepsilon>0$;
    \item \label{en:singular_functions4} $\supp(f^* * f)\cap X$ is meager in $X$;
    \item \label{en:singular_functions5} $\supp(f^* * f)\cap X$ has empty interior in $X$;
    \item \label{en:singular_functions6} $\{ x \in X: |f^* *f (x)|> \varepsilon\}$ has empty interior in $X$ for all $\varepsilon > 0$.
   \end{enumerate}	
Moreover, the above equivalent conditions imply
\begin{enumerate} \setcounter{enumi}{6}
    \item \label{en:singular_functions7} $\supp(f)\cap \G_{\Hau}=\emptyset$,	
\end{enumerate}	
and if $\G_{\Hau}$ is comeager or dense in $\G$ (which holds when $\G$ can be covered by a countable family of bisections or when $\G$ is topologically principal), then all conditions \ref{en:singular_functions1}--\ref{en:singular_functions7} are equivalent.
\end{prop}
\begin{proof}
Implication \ref{en:singular_functions1}$\Rightarrow$\ref{en:singular_functions2} holds because we can cover $\supp(f)$ with a countable family of bisections (cf. the proof of Lemma \ref{lem:comeager}), and so its union is a Baire space. 
Implication \ref{en:singular_functions2}$\Rightarrow$\ref{en:singular_functions3} is obvious. 
And \ref{en:singular_functions3} implies \ref{en:singular_functions1} because if $\supp(f)$ is not meager in $\G$, then by Lemma~\ref{lem:comeager} there is a point $\gamma_0 \in \supp(f)$ at which $f$ is continuous. 
Thus, for $\varepsilon = |f(\gamma_{0})|/2$, the set $\{\gamma\in \G: |f(\gamma)|> \varepsilon\}$ contains a neighbourhood of $\gamma_0$ and so has non-empty interior. 
Hence \ref{en:singular_functions1}--\ref{en:singular_functions3} are equivalent.  

Implications \ref{en:singular_functions4}$\Rightarrow$\ref{en:singular_functions5}$\Rightarrow$\ref{en:singular_functions6} are clear. 
Since $(f^* * f)(x) = \sum_{ \gamma \in d^{-1}(x) } |f(\gamma)|^2$ for $x\in X$,  \ref{en:singular_functions6} implies \ref{en:singular_functions3}. 
We close the cycle of implications by showing that \ref{en:singular_functions1} implies \ref{en:singular_functions6}. 
To this end, note that $\supp(f)$ can be covered by a countable family $\{ U_n \}_{n=1}^{\infty} \subseteq \Bis(\G)$, because $f$ is a uniform limit of a sequence of functions in $\mathfrak{C}_0(\G , \L)$, which are supported on a finite union of bisections. 
So $\supp(f^* * f) \cap X = \bigcup_{n=1}^\infty d(\supp(f)\cap U_n)$ is meager if each of the sets $\supp(f)\cap U_n$ is meager. 
This proves that \ref{en:singular_functions1}--\ref{en:singular_functions6} are equivalent.

Lemma~\ref{lem:singular_groupoid}\ref{en:singular_groupoid3} implies that if $\supp(f)\cap \G_{\Hau}\neq \emptyset$, then $\supp(f)$ has non-empty interior. 
Thus \ref{en:singular_functions2} implies \ref{en:singular_functions7}. 
The converse implication is trivial if we assume that $\G_{\Hau}$ is dense (which is guaranteed if there is a countable cover of $\G$ by bisections, by Lemma~\ref{lem:singular_groupoid}\ref{en:singular_groupoid2}, or when $\G$ is topologically principal, see  Remark~\ref{rem:meagerness_of_singular}). 
\end{proof}

\begin{rem}
In some pathological situations, see \cite[Example 7.16]{Kwa-Meyer}, it may happen that $\G_{\Hau} = \emptyset$ and so in general condition \ref{en:singular_functions7} above does not imply the others. 
\end{rem}

When $\G=\G_{\Hau}$ is Hausdorff, then $\mathfrak{C}_0(\G , \L) = C_0(\G , \L)$ and $\mathfrak{M}_0(\G , \L) = \{ 0 \}$.
Examples of non-Hausdorff groupoids with $\mathfrak{M}_0(\G , \L) = \{ 0 \}$ can be constructed using the following criteria coming from \cite[Lemma 3.1]{CEPSS}.
Recall that a topological groupoid is \emph{ample} if it has a basis consisting of compact open bisections. 
Equivalentely, it is an \'etale groupoid with a totally disconnected unit space. 

\begin{lem} \label{lem:ample_essential_reduced_coincide}
If $\G$ is ample and every compact open set in $\G$ is regular open, that is, it is equal to the interior of its closure, then $\mathfrak{M}_0(\G , \L) = \{ 0 \}$.
\end{lem}
\begin{proof} 
Let $K(\G , \L)\subseteq \mathfrak{C}_c(\G , \L)$ be the linear span of sections $g_V \in C(V , \L)$ where $V$ is a compact open bisection and $|g_V| = \textit{const}$ on $V$. 
Since $\G$ is ample, $K(\G,\L)$ is dense in $\mathfrak{C}_0(\G , \L)$. 
Indeed, the linear span of sections in $C_0(V,\LL)$, $V\in S(\LL)$, is dense in $\mathfrak{C}_0(\G , \L)$, and for every $V\in S(\LL)$ we have $C_0(V,\LL)\cong C_0(V)$, and it is well known that the linear span of indicator functions of compact open sets is dense in $C_0(V)$ when $V$ is totally disconnected. 

We claim that for any $g\in K(\G , \L)$ and $c>0$ the set $|g|^{-1}(c)=\{\gamma\in \G: |g(\gamma)|=c\}$ is either empty or has non-empty interior. 
Indeed, we have $g = \sum_{V \in F} g_V$, where $F$ is a finite collection of compact open bisections and $g_V \in C(V , \L)$ with $|g_V| = \textit{const}$ on  $V\in F$. 
By  the inclusion-exclusion principle, putting for any non-empty $F_0\subseteq F$
\[
   V_{F_0}^F := \left( \bigcap_{V \in F_0} V \right) \setminus \left( \bigcup_{W\in F\setminus F_0} W \right) ,
\]
we get a family $\{ V_{F_{0}}^{F} \}_{\emptyset \neq F_0\subseteq F}$ of pairwise-disjoint sets such that $V = \bigsqcup_{V\in F_0\subseteq F}V_{F_0}^F$, for every $V \in F$ (cf.\ \cite[Remark 2.5]{CFST} or \cite[Lemma 6.2]{BKM}). 
Thus $|g|$ is constant on each of the sets $V_{F_{0}}^{F}$. 
Hence if $|g|^{-1}(c)\neq \emptyset$, then $|g|^{-1}(c)$ contains a non-empty set $V_{F_{0}}^{F}$. 
But the set $V_{F_{0}}^{F}$ is the difference of two compact open sets $\bigcap_{V \in F_0} V $ and $\bigcup_{W\in F\setminus F_0} W$, which are regular open by assumption. 
Hence $V_{F_{0}}^{F}$ has a non-empty interior by \cite[Lemma 2.1]{CEPSS}. 

Now let $f\in \mathfrak{C}_0(\G , \L)$ be non-zero. 
Find $g\in K(\G , \L)$ with $\|f-g\|_{\infty}<\|f\|_{\infty}/3$, and so $\|g\|_{\infty}> 2\|f\|_{\infty}/3$. 
Thus there is $\gamma\in \G$ with $c:=|g(\gamma)|> 2\|f\|_{\infty}/3$. 
By the above claim $|g|^{-1}(c)$ has a non-empty interior. 
Moreover, $|g|^{-1}(c) \subseteq \supp(f)$ because, for each $\gamma \in |g|^{-1}(c)$, we have $| f(\gamma) | \geq c -\| f \|_{\infty} / 3 > \| f \|_{\infty} / 3 >0$. 
Hence $\supp(f)$ has nonempty interior and therefore $f\notin \mathfrak{M}_0(\G , \L)$.
\end{proof}

\subsection{Essential and exotic groupoid Banach algebras}

For any open $U \subseteq \G$, we denote by $\mathfrak{M}(U , \L)\subseteq \B(U , \L)$ the space of bounded Borel sections of meager support, and we consider
\[
    \D(U , \L) := \B(U , \L)/ \mathfrak{M}(U , \L)  
\]
as the quotient Banach space. 
Let $q_U : \B(U , \L) \to  \D(U , \L)$ be the quotient map.
If $U,V\in \Bis(\G)$, then the convolution and involution induce mappings $\D(U, \L) \times \D(V,\L) \to \D(UV, \L)$ and $\D(U, \L)\to \D(U^*, \L)$ that we will still denote by $*$ and $^*$. 
Also we may and we will treat $\D(U,\LL)$ as a subspace of $\D(\G,\LL)$. 
When the twist is trivial we will omit writing it. 
In particular, 
\[
    \D(X) = \B(X)/ \mathfrak{M}(X)
\]
is a $C^*$-algebra, often called the \emph{Dixmier algebra of $X$}, as it was studied first in \cite{Dixmier}. 
It can be identified both with the \emph{injective hull} of $C_0(X)$ and with the \emph{local multiplier algebra} of $C_0(X)$, see \cite[subsection 4.4]{Kwa-Meyer} and references therein. 
The Dixmier algebra $\D(X)$ contains $C(X)$, in the sense that the quotient map $q_X : \B(X) \to \D(X)$ is isometric on $C(X)$, so we will write $C(X) \subseteq \D(X)$. 
Similarly, we have $C(U,\LL) \subseteq \D(U,\LL)$ for any $U\in \Bis(\G)$, and for $f \in \B(U,\LL)$ we have
\begin{equation}\label{eq:essential_norm}
    \| q_{U}(f) \| =
        \inf_{ \stackrel{M \subseteq U}{\text{meager}} } \sup_{\gamma\in U\setminus M} | f(\gamma) | = \min_{ \stackrel{M \subseteq U}{\text{meager}} } \sup_{\gamma\in U\setminus M}|f(\gamma)| , 
\end{equation}
where the minimum exists because countable intersections of comeager sets are comeager. 
The Banach $C_0(X)$-bimodule structure on $\B(U,\LL)$ factors to a Banach $C_0(X)$-bimodule structure on $\D(U,\LL)$, and we have a $C_0(X)$-bimodule map $\E_{U}:=q_{U}\circ E_{U}: \mathfrak{C}_c(\G , \L) \to \D(U,\LL)$. 
Now we proceed as in Section~\ref{sec:ReducedBanachAlgs}, but with $E_U$ replaced by $\E_U$. 
In the $C^*$-algebraic context, $\E_X$ was called a local expectation in \cite{Kwa-Meyer}.

\begin{defn}\label{de:ExoticGroupoidBanachAlgebra}
An \emph{exotic groupoid Banach algebra of $(\G , \L)$} is a Hausdorff completion $\Falg{\G , \LL}{}{\RR}{}$ of $\mathfrak{C}_c(\G , \L)$ such that  $C_0(X)$ embeds isometrically into  $\FR(\G , \LL)$ and  $\E_{X}:\mathfrak{C}_c(\G , \L)\to \D(X)$ induces a contractive map $\E_{X}^{\RR} : \Falg{\G , \LL}{}{\RR}{} \to \D(X)$, that is, the corresponding seminorm satisfies $\| q_{X}(f|_X) \|\leq \|f\|_{\RR}$ for  $f\in \mathfrak{C}_c(\G , \L)$ and $\|f\|_{\infty}=\|f\|_{\RR}$ for $f\in C_c(X)$. 
If in addition $\E_{X}^{\RR}$ is faithful, we call $\FR(\G , \LL)$ an \emph{essential groupoid Banach algebra} of $(\G , \L)$. 
\end{defn}

\begin{rem}\label{rem:about_exotic}
A $C^*$-algebra is an exotic groupoid Banach algebra of $(\G , \L)$ if and only if it is an exotic $C^*$-algebra of $(\G , \L)$ in the sense of \cite{Kwa-Meyer}. 
There is a unique $C^*$-algebra which is an essential groupoid Banach algebra of $(\G , \L)$, and  this is the algebra introduced in \cite[Definition 7.12]{Kwa-Meyer} (for $\F=\C$).  Every groupoid Banach algebra is an exotic groupoid Banach algebra and the converse holds when $\G$ is Hausdorff because then, for every $f\in \mathfrak{C}_c(\G , \L) = C_c(\G , \L)$, we have $f|_X\in C_c(X)$ and therefore $\| q_{X}(f|_X)\| = \| f|_X \|_{\infty}$.
In particular, if $\G$ is Hausdorff then the classes of essential and reduced groupoid Banach algebras coincide.

\end{rem}

\begin{lem}\label{lem:bimodule_Dixmier}
The $\mathfrak{C}_c(\G,\LL)$-bimodule structure on $\B(\G , \L)$ induced by convolution, defined in Lemma~\ref{lem:bimodule_borel}, factors to a $\mathfrak{C}_c(\G,\LL)$-bimodule structure on $\D(\G , \L)$,  which is continuous
in the sense that for any sequence  $(f_n)_{n=1}^\infty$ in $\B(\G , \L)$, $f \in \B(\G , \L)$ and $g\in \mathfrak{C}_c(\G , \L)$, if $q_{\G}(f_n)\to q_{\G}(f)$, then $q_{\G}(g*f_n)\to q_{\G}(g*f)$ and $q_{\G}(f_n*g)\to q_{\G}(f*g)$. 
This turns $q_{\G} : \mathfrak{C}_c(\G,\LL) \to \D(\G,\LL)$ into a $\mathfrak{C}_c(\G,\LL)$-bimodule map.
\end{lem}
\begin{proof} 
If $g \in C_c(U^* , \L)$, for $U\in \Bis(\G)$ and $f\in \B(\G,\LL)$, then \eqref{equ:bisection_left_module_formula} implies $\|q_{\G}(g*f)\|_{\D(\G , \L)}\leq \|g\|_{\infty} \|q_{\G}(f)\|_{\D(\G , \L)}$, which gives $q_{\G}(g*f_n)\to q_{\G}(g*f)$ whenever $q_{\G}(f_n)\to q_{\G}(f)$. 
Since elements in $C_c(U , \L)$, $U\in \Bis(\G)$, span $\mathfrak{C}_c(\G , \L)$, this holds for any $g\in \mathfrak{C}_c(\G , \L)$. 
Similarly (or by using the involution) one gets the same conclusion for $q_{\G}(f_n*g)\to q_{\G}(f*g)$. 
\end{proof}

\begin{prop}\label{prop:local_expectations_and_local_j_map}
Let $\Falg{\G , \LL}{}{\RR}{}$ be a Banach algebra Hausdorff completion of $\mathfrak{C}_c(\G , \L)$ in a seminorm  satisfying $\|f\|_{\RR}=\|f\|_{\infty}$ for $f\in C_c(X)$.
The following conditions are equivalent:
\begin{enumerate}
    \item\label{enu:local_expectations_and_j_map1} $\Falg{\G , \LL}{}{\RR}{}$ is an exotic groupoid Banach algebra of $(\G,\LL)$; 
    \item\label{enu:local_expectations_and_j_map2} for any $U\in \Bis(\G)$, $\E_{U}: \mathfrak{C}_c(\G , \L) \to \D(U,\LL)$ factors through a contractive linear map $\E^\RR_U : \Falg{\G , \LL}{}{\RR}{} \to \D(U,\L)$;
    \item\label{enu:local_expectations_and_j_map3} the quotient map $q_{\G}:\mathfrak{C}_c(\G,\LL) \to \D(\G,\LL)$ factors through a contractive linear map $j_{\RR}^{\ess} : \Falg{\G , \LL}{}{\RR}{} \to \D(\G,\LL)$. 
\end{enumerate}
If the above equivalent conditions hold, then 
for each $U\in \Bis(\G)$, $\E^\RR_{U}$ is a $C_0(X)$-module map given by $\E^\RR_{U}(f) = j_{\RR}^{\ess}(f)|_{U}$ for $f\in \Falg{\G , \LL}{}{\RR}{}$, and so $C_c(U,\LL)$ embeds canonically into  $\Falg{\G , \LL}{}{\RR}{}$. 
Moreover,  for any $S \subseteq \Bis(\G)$ that covers $\G$  we have
\[
    \ker j_{\RR}^{\ess} = \bigcap_{U\in S} \ker \E^\RR_{U} = \{ f\in \Falg{\G , \LL}{}{\RR}{} : \text{$\E^\RR_U(af)=0$ for all $a\in C_0(r(U)),\  U\in S$} \},
\]
and $\ker j_{\RR}^{\ess}$ 
is the largest ideal in $\Falg{\G , \LL}{}{\RR}{}$ contained in $\ker \E^\RR_{X}$.
\end{prop}
\begin{proof} 
We adapt the proof of Proposition~\ref{prop:expectations_and_j_map}. 
Implications \ref{enu:local_expectations_and_j_map3}$\Rightarrow$\ref{enu:local_expectations_and_j_map2}$\Rightarrow$\ref{enu:local_expectations_and_j_map1} are clear (by putting $\E^\RR_{U}(f) := j_{\RR}^{\ess}(f)|_{U}$ for $f\in \Falg{\G , \LL}{}{\RR}{}$, $U\in \Bis(\G)$). 
We assume \ref{enu:local_expectations_and_j_map1}. 
For any $U\in S(\LL)$ and $f\in \mathfrak{C}_c(r(U)\G , \L)$, taking $b$ as in Lemma~\ref{lem:approx_conditional_exp_twisted}, 
we get
\[
    \|q_U(f|_{U})\|=\|q_X(f\circ s|_{U}^{-1})\| = \| q_X(b * f|_X)\| = \| \E_{X}^{\RR} (b *f) \| \leq \|b *f \|_{\RR} \leq \| f \|_{\RR}.
\]
This proves the existence of a linear contractive map $\E^\RR_U : C_0(r(U)) \Falg{\G , \LL}{}{\RR}{} \to \D(U,\L)$ induced by $\E^\RR_U$. 
Let $f\in \mathfrak{C}_c(\G,\LL)$.  
By Lemma~\ref{lem:comeager} and \eqref{eq:essential_norm}, for any $\varepsilon >0$ there is a non-empty open set $U\subseteq \G$ and a meager set $M\subseteq U$  such that $\|q_{\G}(f)\|<|f(\gamma)|+\varepsilon$ for any $\gamma\in U\setminus M$. 
We may assume that $U\in S(\LL)$. 
Pick a  norm one function $a\in C_c(r(U))$ which is equal to $1$ on $r(V)$ for some non-empty open set $V\subseteq U$. 
Then $f(\gamma)=E^\RR_{U}(a* f)(\gamma)$ for $\gamma \in V$. 
Thus 
\[
    \| q_{\G}(f) \|-\varepsilon < \|q_{V}(f|_{V})\|=\|q_{V}\circ E^\RR_{V}(a* f)\|\leq \|\E^\RR_{U}(a* f)\| \leq \|a* f\|_{\RR}\leq \| f\|_{\RR}.
\]
This implies that $\|q_{\G}(f)\|_{\infty}\leq \| f\|_{\RR}$ and hence \ref{enu:local_expectations_and_j_map3} holds. 
Thus \ref{enu:local_expectations_and_j_map1}--\ref{enu:local_expectations_and_j_map3} are equivalent. 

Now assume that \ref{enu:local_expectations_and_j_map1}--\ref{enu:local_expectations_and_j_map3} hold.
Let $i_{\RR} : \mathfrak{C}_c(\G,\LL) \to \Falg{\G , \LL}{}{\RR}{}$ be the canonical homomorphism. 
The quotient $q_{\G} : \mathfrak{C}_c(\G,\LL) \to \mathfrak{C}_0(\G,\LL)/\mathfrak{M}(\G , \LL) \subseteq \D(\G,\LL)$ is a $\mathfrak{C}_c(\G,\LL)$-bimodule map by Lemma~\ref{lem:bimodule_Dixmier}. 
By definition of $j_\RR^{\ess}$ we have $q_{\G}|_{\mathfrak{C}_c(\G,\LL)} = j_\RR^{\ess}\circ i_{\RR}$.
Fix $f \in \Falg{\G , \LL}{}{\RR}{}$ and $g \in \mathfrak{C}_c(\G , \L)$. 
Using the bimodule structure on $\D(\G,\L)$ from Lemma~\ref{lem:bimodule_Dixmier}, we claim that 
\begin{equation}\label{eq:j_ess_bimodule_property}
    j_{\RR}^{\ess}(i_{\RR}(g)f) = g *j_{\RR}^{\ess}(f), \qquad j_{\RR}^{\ess}(f i_{\RR}(g)) = j_{\RR}^{\ess}(f) * g.
\end{equation}
Indeed, pick a sequence $( f_n )_{n=1}^\infty$ in $\mathfrak{C}_c(\G , \L)$ such that $( i_{\RR}(f_n) )_{n=1}^\infty$ converges to $f$ in $\| \cdot \|_{\RR}$. 
Then $( q_{\G}(f_n) )_{n=1}^\infty$ converges to $j_{\RR}^{\ess}(f)$ in the quotient norm of $\D(\G,\LL)$.
Hence also $g * q_{\G}(f_n) \to g* j_{\RR}^{\ess}(f)$ by Lemma~\ref{lem:bimodule_Dixmier}.
Thus
\[  
    g * j_{\RR}^{\ess}(f) = \lim_{n\to \infty} g * q_{\G}(f_n) =\lim_{n\to \infty} q_{\G}(g * f_n) = \lim_{n\to \infty} j_{\RR}^{\ess} \big( i_{\RR}(g) i_{\RR}(f_n) \big) = j_{\RR}^{\ess}( i_{\RR}(g) f).
\]
Symmetric considerations give $j_{\RR}^{\ess}(fi_{\RR}(g))=j_{\RR}^{\ess}(f) *g$, which proves \eqref{eq:j_ess_bimodule_property}.
By density of $\mathfrak{C}_c(\G,\LL)$ in $\Falg{\G , \LL}{}{\RR}{}$, the equalities \eqref{eq:j_ess_bimodule_property} imply that $\ker j_{\RR}^{\ess}$ is an ideal in $\Falg{\G , \LL}{}{\RR}{}$. Since $\E^\RR_{U}(b) = j_{\RR}^{\ess}(b)|_{U}$ for 
 $U\in \Bis(\G)$ and $b\in \Falg{\G , \LL}{}{\RR}{}$, \eqref{eq:j_ess_bimodule_property} implies that $\E^\RR_{U}$ is a $C_0(X)$-module map. 
Now the equivalent descriptions of $\ker j_{\RR}^{\ess}$ follow, and the same argument as in the proof of Proposition \ref{prop:expectations_and_j_map} 
shows that $\ker j_{\RR}^{\ess}$ is the largest ideal in $\Falg{\G , \LL}{}{\RR}{}$ contained in  $\ker \E^\RR_{X}$.
\end{proof}

\begin{rem}\label{rem:essential_vs_reduced}
In view of Proposition~\ref{prop:local_expectations_and_local_j_map}, an exotic Banach algebra of  $(\G , \L)$ 
is a Hausdorff completion $\Falg{\G , \LL}{}{\RR}{}$ of $\mathfrak{C}_c(\G , \L)$ equipped with a contractive linear map $j_{\RR}^{\ess} : \Falg{\G , \LL}{}{\RR}{} \to \D(\G,\LL)$ which is the identity on each of the spaces $C_c(U,\LL)$, $U\in \Bis(\G)$. 
This map is injective if and only if  $\Falg{\G , \LL}{}{\RR}{}$ is essential.
By the same argument as in Corollary~\ref{co:QuotientIsReduced}, for any exotic groupoid Banach algebra $\Falg{\G , \LL}{}{\RR}{}$, the quotient $\Falg{\G , \LL}{}{\RR}{} / \ker j_{\RR}^{\ess}$ is naturally an essential groupoid Banach algebra of $(\G , \L)$. 
Every groupoid Banach algebra $\Falg{\G , \LL}{}{\RR}{}$ is exotic and we have $j_{\RR}^{\ess} = q_{\G}\circ j_{\RR}$ and $\ker j_{\RR}^{\ess}=j_{\RR}^{-1}(\mathfrak{M}_0(\G , \L))$. 
In particular, if $\mathfrak{M}_0(\G , \L)=\{0\}$, then all reduced algebras are essential, as then $\ker j_{\RR}^{\ess}=\ker j_{\RR}^{\red}$. 
\end{rem}

\begin{rem}\label{rem:exotic_and_graded}  
If an exotic groupoid Banach algebra $\FR(\G , \L)$ is $S$-graded by a unital inverse subsemigroup $S \subseteq \Bis(\G)$ which covers $\G$, see Definition~\ref{defn:S_graded_algebras}, then  the spaces $C_0(U,\LL)$, $U\in S$,  embed isometrically into $\FR(\G , \LL)$ in a way that $j_{\RR}^{\ess}|_{C_0(U,\LL)}$ is the identity on $C_0(U,\LL)$, $U\in S$. 
Indeed, by $S$-grading, for $f\in C_c(U,\LL)$, $U\in S$, we have $\|f\|_{\RR}\leq \|f\|_{\infty}$ and contractivness of $\E^\RR_U$ 
gives  the opposite inequality. 
\end{rem}

The following is an analogue of Corollary~\ref{cor:kernel_of_reps_with_generalized_expectation}: 

\begin{cor}\label{cor:kernel_of_reps_with_expectation} 
Let $\Falg{\G , \LL}{}{\RR}{}$ be an exotic groupoid Banach algebra. 
For any representation $\psi : \Falg{\G , \LL}{}{\RR}{} \to B$ such that $\|q_X(f|_X)\| \leq \|\psi(f)\|$ for $f\in \mathfrak{C}_c(\G,\LL)$, we have $\ker \psi \subseteq \ker j_{\RR}^{\ess}$, and thus $\psi$ is necessarily injective on the spaces $C_c(U , \L)$, $U\in \Bis(\G)$. 
If $\Falg{\G , \LL}{}{\RR}{}$ is an essential groupoid Banach algebra, then $\psi$ is necessarily injective on $\Falg{\G , \LL}{}{\RR}{}$.
\end{cor}
\begin{proof}
By the assumed inequality we may treat $F_{\psi}(\G , \LL) := \overline{\psi(\Falg{\G , \LL}{}{\RR}{})}$ as an exotic groupoid Banach algebra. 
By Proposition~\ref{prop:local_expectations_and_local_j_map}, there are contractive maps $j_{\psi}^{\ess} : F_{\psi}(\G , \LL) \to \D(\G,\LL)$ and $j_{\RR}^{\ess} : \Falg{\G , \LL}{}{\RR}{} \to \D(\G,\LL)$ such that $j_{\psi}^{\ess} \circ \psi = j_{\RR}^{\ess}$. 
Hence $\ker \psi \subseteq \ker j_{\RR}^{\ess}$. 
If $\Falg{\G , \LL}{}{\RR}{}$ is essential, then  $j_{\RR}^{\ess}$ is injective, and so $\psi$ is injective.
\end{proof}

Consider the essential algebra $F^p_{\ess}(\G , \LL) := F^p(\G,\LL) / j_{p}^{-1}(\mathfrak{M}_0(\G , \L))$ for $p\in [1,\infty]$. 
If $p=2$, then $F^2_{\ess}(\G , \LL) = C^{*}_{\ess}(\G , \L)$ is the unique essential $C^*$-algebra. 
If $p\neq 2$, we do not know if $\Falg{\G , \LL}{\ess}{p}{}$ is an $L^p$-operator algebra, as such operator algebras in general are not closed under quotients. 
However, if $\G$ can be covered by a countable family of bisections (so for instance when $\G$ is $\sigma$-compact) or when $\G$ is topologically principal, then we have a natural injective representation of $\Falg{\G , \LL}{\ess}{p}{}$ on an $L^p$-space. 
Namely, by Lemma~\ref{lem:singular_groupoid}, $\G_{\Hau}$ is a full subgroupoid of $\G$, and so the regular representation $\Lambda_p : F^p(\G,\LL) \to B(\ell^{p}(\G , \L))$ compresses to an `essential representation' $\widetilde{\Lambda}_p : F^p(\G,\LL) \to B ( \ell^{p} (\G_{\Hau} , \L))$. 
By \eqref{eq:j_map}, we have
\[
    \ker \widetilde{\Lambda}_p = \{ f\in \mathfrak{C}_c(\G , \L) : \supp( j_p(f) ) \subseteq \G\setminus \G_{\Hau} \} .
\]
Hence if $\G_{\Hau}$ is dense in $\G$ (which holds whenever $\G$ can be covered by a countable family of bisections or when $\G$ is topologically principal), then using Proposition~\ref{prop:singular_functions} we see that $\widetilde{\Lambda}_p$ induces an injective representation $\Falg{\G , \LL}{\ess}{p}{} \to B(\ell^{p}( \G_{\Hau} , \L))$.  
The closure of its range is an essential groupoid $L^p$-operator algebra:

\begin{prop}\label{prop:essential_L^p_operator_algebra}
Let $p\in [1,\infty]$.
Assume $\G$ can be covered by a countable family of open bisections and consider the compression $\widetilde{\Lambda}_p : F^p(\G,\LL) \to B ( \ell^{p} (\G_{\Hau} , \L))$ of the regular representation $\widetilde{\Lambda}_p : F^p(\G,\LL) \to B ( \ell^{p} (\G_{\Hau} , \L))$, given by \eqref{eq:regular_for_p}.
Then  
\[
    \FFalg{\G , \LL}{\ess}{p}{} := \Falg{\G , \LL}{}{\{\widetilde{\Lambda}_p\}}{}\cong \overline{\widetilde{\Lambda}_p(F^p(\G,\LL))}
\]
is an essential groupoid $L^p$-operator algebra of $(\G , \L)$. 
\end{prop}
\begin{proof} 
As in the the proof of Proposition \ref{prop:regular_disintegrated}, we may assume that $p<\infty$. 
We identify $\FFalg{\G , \LL}{\ess}{p}{}$  with $\overline{\widetilde{\Lambda}_p(F^p(\G,\LL))}$, and we choose a Schauder basis $\{1_\gamma\}_{\gamma\in  \G_{\Hau}}$ for $\ell^{p}( \G_{\Hau},\LL)$. 
By Lemma~\ref{lem:singular_groupoid}\ref{en:singular_groupoid2}, the set $\G_{\Hau}$ is meager. 
Thus, using \eqref{eq:essential_norm} and recalling the notation from the proof of Proposition~\ref{prop:regular_disintegrated}, for any $f\in \mathfrak{C}_c(\G , \L)$ we get 
\[
    \| q_{\G}(f) \| \leq \| f|_{\G_{\Hau} } \|_{\infty} = \sup_{\gamma\in  \G_{\Hau}} | ( f 1_{d(\gamma)} )(\gamma) | \leq \sup_{\gamma\in\G_{\Hau}} \|\widetilde{\Lambda}_p(f) 1_{d(\gamma)} \|_p \leq \|\widetilde{\Lambda}_p(f)\|.
\]
Hence we get a contractive linear map $j_p^{\ess} : \FFalg{\G , \LL}{\ess}{p}{} \to \D(\G,\LL)$  consistent with  the quotient map $q_{\G}:\mathfrak{C}_c(\G,\LL) \to \D(\G,\LL)$. 
As in \eqref{eq:j_map} one sees that  for every $\gamma \in \G_{\Hau}$ and any $f\in \FFalg{\G , \LL}{\ess}{p}{}$, we have $\| f 1_{\gamma} \|_p = \left(\sum_{d(\eta) = d(\gamma)} |j_p^{\ess}(f)(\eta \gamma^{-1})|^p\right)^{1/p}$. 
Thus $j_p^{\ess}$ is injective.
\end{proof}

\begin{rem}
When $p=2$, the injective representation $\widetilde{\Lambda}_p$ is automatically isometric and so 
$\FFalg{\G , \LL}{\ess}{2}{}\cong \Falg{\G , \LL}{\ess}{2}{}\cong C^*_{\ess}(\G,\LL)$ is the unique essential $C^*$-algebra. 
This representation (for untwisted groupoids) also appears in \cite[Proposition 3.31]{Buss_Martinez}, and a similar representation was constructed in \cite[Proposition 1.12]{Neshveyev_Schwartz}.
For $p\neq 2$, there is no reason for $\widetilde{\Lambda}_p$ to be isometric, unless $\Falg{\G , \LL}{\red}{p}{}=\Falg{\G , \LL}{\ess}{p}{}$.
\end{rem}

\begin{lem}\label{lem:essential_from_other_essential}
Let $\|\cdot \|_{\RR}$ and $\|\cdot\|_{\RR_{i}}$, $i\in I$, be seminorms on $\mathfrak{C}_c(\G , \L)$ that define exotic groupoid Banach algebras 
$\Falg{\G , \LL}{}{\RR}{}$ and $\Falg{\G , \LL}{}{\RR_i}{}$, $i\in I$, for $(\G , \L)$. 
Then the formulas \eqref{eq:R_norms_from_other_norms}, provided $\|\cdot\|_{\{ \RR_i\}_{i}}$ is finite, define seminorms that yield exotic groupoid Banach algebras $\Falg{\G , \LL}{}{\RR^*}{}$, $\Falg{\G , \LL}{}{ \{ \RR_i \}_{i} }{}$ for $(\G , \L)$.  
Moreover, $\Falg{\G , \LL}{}{\RR^*}{}$ is essential when $\Falg{\G , \LL}{}{\RR}{}$ is, and $\Falg{\G , \LL}{}{ \{ \RR_i \}_i}{}$ is essential if and only if all $\Falg{\G , \LL}{}{\RR_i}{}$, $i\in I$, are essential. 
\end{lem}
\begin{proof}
We leave it to the reader to adapt the proof of Lemma \ref{lem:reduced_from_other_reduced}.
\end{proof}

\begin{cor}\label{cor:essential_L^P_operator_algebra}
Assume $\G$ can be covered by a countable family of open bisections and let  $P\subseteq [1,\infty]$ be non-empty.
Then  $\FFalg{\G , \LL}{\ess}{P}{} := \Falg{\G , \LL}{}{\{\widetilde{\Lambda}_p\}_{p\in P}}{}$ is an essential groupoid Banach algebra of $(\G , \LL)$. 
\end{cor}

\section{Intersection property and topological freeness}
\label{sec:IntersectionAperiodicTopFree}

In this section, we study conditions implying the following intersection properties for $A=C_0(X)$ and $B$ being an exotic groupoid Banach algebra, see \cite[Definition 5.6]{BK} and \cite[Definitions 5.5, 5.6]{Kwa-Meyer}. 

\subsection{Intersection property and topological freeness} 
\label{ssec:IntersectionPropTopFree}

Recall that the ideals we consider are closed and two-sided. 

\begin{defn}\label{def:visible} 
Let $A\subseteq B$ be a Banach subalgebra of a Banach algebra $B$. 
We say that \emph{$A$ detects ideals in $B$}, or that $A\subseteq B$ has the \emph{intersection property}, if for every non-zero ideal $J$ in $B$, we have $J\cap A \neq \{ 0 \}$. 
The inclusion has the \emph{generalized intersection property} if there is a largest ideal $\Null$ in $B$ with $\Null \cap A = \{ 0 \}$. 
In this case we call $\Null$ the \emph{hidden ideal} and put $B_{\ess} := B/ \Null$. We say $A\subseteq B$ is \emph{minimal} if for every non-zero  $a\in A$ we have $\overline{BaB}=B$. 
\end{defn}

\begin{lem}\label{lem:simplicity_in_general} 
If $A\subseteq B$ has the generalized intersection property and $\overline{A}$ is the closure of the range of $A$ in the quotient $B_{\ess}$, then $\overline{A}\subseteq B_{\ess}$ has the intersection property.  
An inclusion $A\subseteq B$  has the intersection property and is minimal if and only if $B$ is simple. 
\end{lem}	
\begin{proof} 
Assume $A\subseteq B$ has the generalized intersection property, let $q : B\to B_{\ess}$ be the quotient map and put $\overline{A} := \overline{q(A)}$. 
If $J$ is a non-zero ideal in $B_{\ess}$ then $q^{-1}(J)$ is an ideal in $B$ which is strictly larger than $\Null$. 
Thus $A\cap q^{-1}(J)\neq \{0\}$ and because $q$ is injective on $A$ we therefore get $\{0\}\neq q( A\cap q^{-1}(J))\subseteq \overline{A}\cap J$. 
Hence $\overline{A}\subseteq B_{\ess}$ has the intersection property. 

Let $A\subseteq B$ be any Banach algebra inclusion.
If it is minimal and has the intersection property, then for any non-zero ideal $J$ in $B$, $I:= A \cap J$ is a non-zero ideal in $B$ and therefore $B = \overline{BIB} \subseteq J$, which proves that $B$ is simple. 
The converse implication is straightforward. 
\end{proof}

\begin{rem}\label{rem:intersection_property_for_C(X)} 
If $A\subseteq B$ has the generalized intersection property and $A$ is a $C^*$-algebra, 
then the embedding $A\subseteq B_{\ess}$ is isometric (that is, we have $A = \overline{A}$ in Lemma~\ref{lem:simplicity_in_general}). 
\end{rem} 

Recall that a set $U\subseteq X$ is \emph{$\G$-invariant} if $d(\gamma)\in U$ implies $r(\gamma)\in U$ for all $\gamma \in G$.
The groupoid is \emph{minimal} if there are no non-trivial $\G$-invariant open sets in $X$.

\begin{lem}\label{lem:invariant_ideals}
Let $B$ be a Banach algebra of the form $\Falg{\G , \L}{}{\RR}{}$ and such that $C_0(X)$ embeds isometrically into $B$. 
For any ideal $I$ in $C_0(X)$, the following are equivalent: 
\begin{enumerate}
    \item\label{enu:invariant_ideals1} $I$ is restricted, that is, $I=C_0(X)\cap J$ for an ideal $J$ in $B$; 
    \item\label{enu:invariant_ideals2} $I=C_0(U)$ for an open $\G$-invariant set $U\subseteq X$; 
    \item\label{enu:invariant_ideals3} $I$ is symmetric, that is, $\overline{IB} = \overline{BI}$.
\end{enumerate} 
Moreover, the inclusion $C_0(X) \subseteq B$ is minimal if and only if $\G$ is minimal. 
\end{lem}
\begin{proof} 
We have $I = C_0(U)$ for an open set $U\subseteq X$. 
For $V\in \Bis(\G)$, we denote the image of $C_0(V,\LL)$ in $B = \Falg{\G , \L}{}{\RR}{}$ by $B_V$. 

\ref{enu:invariant_ideals1}$\Rightarrow$\ref{enu:invariant_ideals2}. Assume $I=C_0(X)\cap J$ for an ideal $J$ in $B$. 
For every $V\in \Bis(\G)$, we have $C_0(V U V^*) = \overline{B_V C_0(U)  B_{V^*}} \subseteq J \cap C_0(X) = C_0(U)$.
Hence $V U V^* \subseteq U$ for every $V \in \Bis(\G)$, which implies $U$ is $\G$-invariant. 

\ref{enu:invariant_ideals2}$\Rightarrow$\ref{enu:invariant_ideals3}. 
If \ref{enu:invariant_ideals2} holds, then $VU=UV$ for all $V\in\Bis(\G)$, which implies that $I B_V = B_V I$ for $V\in \Bis(\G)$. 
Hence $\overline{IB} = \overline{\sum_{V\in \Bis(\G)} IB_V} = \overline{\sum_{V\in \Bis(\G)} B_VI} = \overline{BI}$.

\ref{enu:invariant_ideals3}$\Rightarrow$\ref{enu:invariant_ideals1}.  If \ref{enu:invariant_ideals3} holds, then $J=\overline{IB}=\overline{BI}$ is an ideal in $B$ (generated by $I$) and an approximate unit $\{\mu_i\}_i$ in $I$ is an approximate unit in $J$. 
Using this, one gets  $I = J\cap C_0(X)$.

Now assume $\G$ is minimal and let $a \in C_0(X)$ be non-zero. 
By the equivalence \ref{enu:invariant_ideals1}$\Leftrightarrow$\ref{enu:invariant_ideals2},  $\overline{BaB} \cap C_0(X) = C_0(U)$ where $U$ is an open $\G$-invariant set. 
Since $U$ contains the open support of $a$ it is non-empty. 
Hence $U=X$ by $\G$-minimality, and therefore $C_0(X)\subseteq \overline{BaB}$.
Since $C_0(X)$ contains an approximate unit for $B$, this implies that $\overline{BaB}=B$. 
Thus $C_0(X)\subseteq B$ is minimal. 
For the converse, assume $\G$ is not minimal. 
Let $U\subseteq X$ be a non-trivial open $\G$-invariant set and put $I=C_0(U)$. 
By the equivalence \ref{enu:invariant_ideals2}$\Leftrightarrow$\ref{enu:invariant_ideals3}, $\overline{BIB} = \overline{IB}$ and so any approximate unit in $I$ is also an approximate unit in $\overline{BIB}$. 
This implies that $\overline{BIB}\cap C_0(X) = I\neq C_0(X)$. 
Hence for any non-zero $a\in I$ we get $\overline{BaB}\neq B$. 
Thus $C_0(X)\subseteq B$ is not minimal. 
\end{proof}

Topological freeness for \'etale groupoids was introduced in \cite[Definition 2.20]{Kwa-Meyer}.

\begin{defn}
An \'etale groupoid $\G$ is \emph{topologically free} if  there is no non-empty open set $V\subseteq \G\setminus X$ with $r|_V=d|_V$.   
\end{defn}

\begin{rem}\label{re:TopologicallyFreeGroupoid}
The groupoid $\G$ is topologically free if and only if for every bisection $V\subseteq \G\setminus X$, the set $\{ x\in X : \G(x) \cap V \neq \emptyset\} = \{ x\in d(V): h_{V}(x) = x \}$ has empty interior in $X$ ($h_V=r \circ d|_{V}^{-1}$ is the partial homeomorphism induced by $V$). The relationship between this and other popular non-triviality conditions is presented in the following diagram: 
		\begin{center}\small
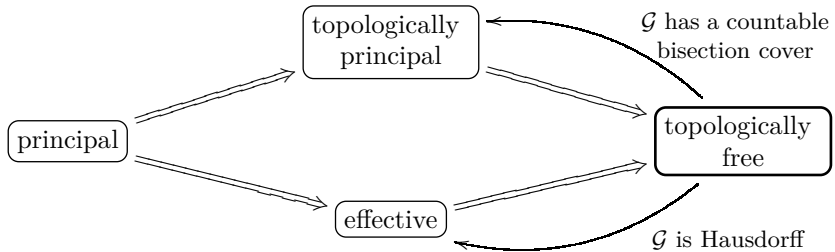
			\begin{picture}(320,104)(0,14)
\put(-8,94){
$ \xymatrixcolsep{5pc} \xymatrixrowsep{0.35pc}
	\xymatrix{
		 & 					
	\text{\tframed[line width=0.5bp,fill=white!]{
\begin{minipage}{2.1cm}topologically 
\centerline{principal}
\end{minipage} }}
 \ar@{=>}[rd]  
  	&     
	\\
	\text{\tframed[line width=0.5bp,fill=white!]{principal}} \ar@{=>}[rd] \ar@{=>}[ru]  &    &  
\text{\tframed[line width=1bp,fill=white!50]{
\begin{minipage}{2.1cm}topologically 
\centerline{free}
\end{minipage} }}  
  \ar@/_2pc/[lu]
	 \ar@/^2pc/[ld]
	\\
			 & \text{\tframed[line width=0.5bp,fill=white!]{effective}}\ar@{=>}[ru]  & 
			}$}
			\put(242,19){\footnotesize $\G$ is Hausdorff}
			\put(238,97){\footnotesize \begin{minipage}{2.5cm}  $\G$ has a countable
				\\
				\centerline{bisection cover}
				\end{minipage}} 
			
\end{picture}
\end{center}
\emph{Effectiveness} of $\G$ means that for any open $V\subseteq \G$ with $r|_V = d|_V$ we have $V\subseteq X$. It  implies topological freeness and the converse is true if $\G$ is Hausdorff ($X$ is closed in $\G$). If $\G$ is topologically principal, that is  if $\{ x\in X : \G(x) \neq \{x\} \}$ has empty interior,
then $\G$ is topologically free and the converse is true if $\G$ is covered by a countable family of bisections (so, for example, when $\G$ is second countable or more generally $\sigma$-compact). 
In general, all these conditions are  different  and topological freeness is the weakest. 
See \cite[Section 2.4]{Kwa-Meyer} for more details.  
\end{rem}

The following lemma is the technical heart of our simplicity and pure infiniteness criteria. 
For these results, we will need an $S$-grading, that we have been able to ignore until now. 
In what follows, we suppress writing the convolution sign `$*$' when multiplying elements of $\mathfrak{C}_c(\G , \L)$. 

\begin{defn}
We call $a\in C_c(X)$ a \emph{bump function} if it is positive norm-one and equal to $1$ on some non-empty open set. 
\end{defn}

\begin{lem}[Pinching property]\label{lem:pinching_property}
Assume $\G$ is topologically free and let $S\subseteq \Bis(\G)$ be a unital inverse semigroup covering $\G$. 
For any $f\in \mathfrak{C}_c(\G , \L)$ and any $\varepsilon >0$, there is a contractive element $b\in C_c(X)$ and a bump function $a \in C_c(X)$ such that $b  f|_X:=(bf)|_{X}=b (f|_{X})$ is in $C_0(X)^+$ and 
\[
    \big\| a \cdot \|q_{X} (f|_{X})\|  - a  (b  f|_X) \big\|_{\infty} \leq \varepsilon  \quad \text{ and } \quad \| a b f a - a ( b  f|_X ) a \|^S_{\max} \leq \varepsilon. 
\]
In particular, the first inequality implies that $\| q_{X} (f|_{X}) \| \leq \| a (b f |_X) a \|_{\infty} + \varepsilon$.
\end{lem}
\begin{proof}
Write $f = \sum_{V\in F} f_V $ where $F\subseteq S$ is finite and $f_V \in C_c(V, \L)$ for $V\in F$. 
The section $f$ is continuous outside of the closed meager set $\bigcup_{V\in F} \partial V$. 
Hence all the more it is continuous on the complement of the closed meager set $M = \bigcup_{V\in F} \partial V \cup r\left(\partial X \cap \overline{\supp(f_V)}\right)$.
Hence $\| q_{X} (f|_{X}) \| = \sup_{x\in X\setminus M} | f (x) |$ and for any $\varepsilon>0$ there is a nonempty open set $W\subseteq X\setminus M$ such that
\begin{equation}\label{eq:auxiliary_approx0}
    |f(x)| \leq \|q_{X} (f|_{X})\|< |f(x)| +\varepsilon , \qquad x\in W . 
\end{equation}
We may assume that $\varepsilon < \|q_{X} (f|_{X})\|$, so that $f\neq 0$ on $W$. 
Let $V\in F$ and consider the open bisection $V\setminus \overline{X}\subseteq \G\setminus X$. 
By topological freeness and \cite[Proposition 2.24]{Kwa-Meyer}, see also \cite[Lemma 2.2]{ELQ}, the finite union  
\[
    R:=\bigcup_{V\in F}\{x\in d(V\setminus \overline{X}): h_{V\setminus \overline{X}}(x)=x\}=\bigcup_{V\in F}\{ x\in X : \G(x) \cap V\setminus \overline{X} \neq \emptyset\}
\]
has empty interior. 
Thus there is $x_0\in W\setminus R$. 
For each $V \in F$, we claim there is a bump function $a_V\in C_c(W)$ equal to $1$ in a neighbourhood of $x_0$ and
\begin{equation}\label{eq:auxiliary_approx}
    | a_V(r(\gamma)) \cdot f_V(\gamma) \cdot a_V(d(\gamma)) | \leq \varepsilon/|F| , \qquad \gamma\in V\setminus \overline{X}.
\end{equation}
Recall that $x_0\notin r\left(\partial X \cap \overline{\supp(f_V)}\right)$ and so only the following three cases are possible:

(1) $x_0 \in X\setminus r\left(\overline{\supp(f_V)}\right)$. 
Then $f_V( r|_{V}^{-1}(x_0))=0$ so there is a neighbourhood $U\subseteq W$ of $x_0$ where $|f_V\circ  (r|_{V})^{-1}| \leq \varepsilon/|F|$. 
Taking any norm one  $a_V\in C_0(U)$, 
we get  $| a_V(r(\gamma)) \cdot f_V(\gamma) | \leq \varepsilon/|F|$ for $\gamma\in V\setminus \overline{X}$, and so \eqref{eq:auxiliary_approx} is satisfied.

(2) $x_0\in r\left(\overline{\supp(f_V)}\cap X\right)\subseteq r(V\cap X)=V\cap X$.  
Thus $x_0\in V\cap W$ and, for any $a_V\in C_c( V\cap W)$, we get $a_V(r(\gamma)) \cdot f_V(\gamma)=0$ for $\gamma \in V\setminus \overline{X}$ because $d(V\setminus \overline{X})\cap V=\emptyset$.

(3) $x_0 \in r\left(\overline{\supp(f_V)}\setminus \overline{X}\right)\subseteq r(V\setminus \overline{X})$. 
Since $x_0\not\in R$, there is a neighborhood $U\subseteq r(V\setminus \overline{X})\cap W$ of $x_0$ such that $h_{V}^{-1}(U) \cap U = \emptyset$. 
Then, for any $a_V\in C_0(U)$, we have $a_V(r(\gamma))a_V(d(\gamma))=0$ for $\gamma \in V\setminus \overline{X}$, and so \eqref{eq:auxiliary_approx} holds.

This proves our claim. 
Since $W\cap r\left(\partial X \cap \overline{\supp(f_V)}\right) = \emptyset$, we have $(a_Vf_V )|_{\partial X} =0$. 
Thus $(a_Vf_V )|_{X}=a_V  (f_V|_{X})\in C_0(X)$, $(a_Vf_V )|_{V\setminus \overline{X}}\in C_0(V\setminus \overline{X},\LL)\subseteq \mathfrak{C}_c(\G , \L)$ and $a_V f_V= (a_Vf_V )|_{X} + (a_Vf_V )|_{V\setminus \overline{X}}$. 
Therefore putting $a := \prod_{V\in F} a_V\in C_c(W)$, we get a bump function equal to $1$ in a neighborhood of $x_0$
such that 
\[
    a f = af|_X  + \sum_{V\in F} af_V|_{V\setminus \overline{X}} 
\]
where $a f|_X := (af)|_X = a\cdot (f|_X) \in C_0(X)$ and $af_V|_{V\setminus \overline{X}}:= (af_V)|_{V\setminus \overline{X}}\in C_0(V\setminus \overline{X},\LL)$, for $V\in F$. 
Take any bump function $c\in C_c(W)$ equal to $1$ on the support of $a$. 
Then $b(x):= c(x)\cdot\frac{\overline{f}(x)}{| f(x) | }$ defines a norm one continuous function on $X$ and $b \cdot (f|_X) = c |(f|_X)|\in C_c(W)^+$.
Moreover, $a b \cdot (f|_X) = a | (f|_X) | \in C_c(W)^+$. 
Thus using \eqref{eq:auxiliary_approx0} we get the first inequality in the assertion.
By the triangle inequality, the definition of $\|\cdot\|^S_{\max}$ and \eqref{eq:auxiliary_approx} we get
\[
    \| \sum_{V\in F} a f_V|_{V\setminus \overline{X}} a \|^S_{\max} \leq \sum_{V\in F} \| af_V|_{V\setminus \overline{X}} a \|^S_{\max} = \sum_{V\in F} \| (a f_V a)|_{V\setminus \overline{X}} \|_{\infty} \leq \varepsilon.
\]
Thus $\| a b f a - a (b f|_X) a \|^S_{\max} \leq \| a f a - a f|_X a \|^S_{\max} = \| \sum_{V\in F} af_V|_{V\setminus \overline{X}} a \|^S_{\max} \leq \varepsilon$. 
\end{proof}

\begin{cor}\label{cor:top_free_implies_expectation}
Assume $\G$ is topologically free and let $S\subseteq \Bis(\G)$ be a unital inverse semigroup covering $\G$. 
If $ \psi : \Falg{\G , \L}{S}{}{} \to B$ is a representation which is injective on $C_0(X)$, then $\overline{\psi(\Falg{\G , \L}{S}{}{})}$ is naturally an exotic groupoid Banach algebra. 
\end{cor} 
\begin{proof}
Since $\psi|_{C_0(X)}$ is injective, $\psi|_{C_0(X)}$ is isometric by minimality of the supremum norm, cf.\ \cite[Proposition 2.10]{BKM}. 
Let $f\in \mathfrak{C}_c(\G , \L)$ and take $\epsilon >0$. 
For $a,b\in C_0(X)$ as in Lemma~\ref{lem:pinching_property}, using the inequalities there, we get
\[
\begin{split}
    \| q_{X} (f|_{X}) \| &\leq \| a ( b \cdot f|_{X}) a \|_\infty +\epsilon = \| \psi(a ( b \cdot f|_{X}) a) \|_B + \epsilon \\
        &\leq \| \psi(a ( b \cdot f) a) \|_B +2\epsilon \leq \| \psi( f ) \| +2\epsilon,
\end{split}
\]
where the third step uses the reverse triangle inequality, contractiveness of $\psi$ and Lemma~\ref{lem:pinching_property}: 
\[
\begin{split}
    \| \psi \big( a (b \cdot f) a \big) \|_B - \| \psi \big( a (b \cdot f|_X) a \big) \|_B &\leq \| \psi \big( a (b \cdot f) a \big) - \psi \big( a (b \cdot f|_X) a \big) \|_B \\
       &\leq \| a (b \cdot f) a - a ( b \cdot f|_X) a \|_{\max}^{S} \leq \epsilon .
\end{split}
\]
Thus we conclude that $\| q_X(f|_X) \| \leq \| \psi(f) \|_B$. 
It follows that $\overline{\psi(\Falg{\G , \L}{S}{}{})}$ is exotic. 
\end{proof}

\begin{thm}\label{thm:Top_Free_Intersection_Property}
Let $S\subseteq \Bis(\G)$ be a unital inverse semigroup covering $\G$. 
Assume $\G$ is topologically free and consider an $S$-graded Banach algebra $\Falg{\G , \L}{}{\RR}{}$ such that $C_0(X)$ embeds isometrically into $\Falg{\G , \L}{}{\RR}{}$. 
\begin{enumerate}
    \item $\Falg{\G , \L}{}{\RR}{}$ is an exotic groupoid Banach algebra and the inclusion $C_0(X) \subseteq \Falg{\G , \L}{}{\RR}{}$ has the generalized intersection property with the hidden ideal $\ker j_{\RR}^{\ess}$.
    \item $\Falg{\G , \L}{}{\RR}{}$ is an essential groupoid Banach algebra if and only if $C_0(X)\subseteq \Falg{\G , \L}{}{\RR}{}$ has the intersection property.
    \item The Banach algebra $\Falg{\G , \L}{}{\RR}{}$ is simple if and only if $\Falg{\G , \L}{}{\RR}{}$ is an essential groupoid Banach algebra and $\G$ is minimal.
\end{enumerate}
\end{thm}
\begin{proof}
Applying Corollary~\ref{cor:top_free_implies_expectation} to the canonical map $\Lambda_\RR : \Falg{\G , \L}{S}{}{} \to \Falg{\G , \L}{}{\RR}{}$ we get that $\Falg{\G , \L}{}{\RR}{}$ is an exotic groupoid Banach algebra. 
For an ideal $J$ in $\Falg{\G , \L}{}{\RR}{}$ with $J\cap C_0(X) = \{ 0 \}$, Corollaries~\ref{cor:kernel_of_reps_with_expectation} and \ref{cor:top_free_implies_expectation} applied to the quotient map $\Falg{\G , \L}{}{\RR}{} \to \Falg{\G , \L}{}{\RR}{} / J$ give that $\Falg{\G , \L}{}{\RR}{} / J$ is an exotic groupoid Banach algebra and $J\subseteq \ker j_{\RR}^{\ess}$. 
Hence the inclusion $C_0(X) \subseteq \Falg{\G , \L}{}{\RR}{}$ has the generalized intersection property with the hidden ideal $\ker j_{\RR}^{\ess}$.
Thus it has the intersection property if and only if $\ker j_{\RR}^{\ess}=\{0\}$, which by definition means that $\Falg{\G , \L}{}{\RR}{}$ is an essential groupoid Banach algebra. 
Combining this with Lemmas~\ref{lem:simplicity_in_general}, \ref{lem:invariant_ideals} one gets that $\Falg{\G , \L}{}{\RR}{}$ is simple if and only if $\Falg{\G , \L}{}{\RR}{}$ is essential and $\G$ is minimal. 
\end{proof}

\subsection{Characterizations of topological freeness}

We start by linking  topological freeness to maximal abelianness of $C_0(X)$ in an essential groupoid algebra. 
In the setting of $C^*$-algebra crossed products this goes back to \cite[Proposition 4.14]{Zeller-Meier} and plays a crucial role in the theory of Cartan $C^*$-subalgebras, cf.\ \cite[Proposition II.4.7]{Renault_book}, \cite[Theorem 4.2]{Re}. 

\begin{prop}\label{prop:Cartan_subalgebra}
If $C_0(X)$ is maximal abelian in some exotic groupoid Banach algebra $\Falg{\G , \L}{}{\RR}{}$, then $\G$ is topologically free. 
If $\G$ is topologically free and Hausdorff, then $C_0(X)$ is maximal abelian in every reduced groupoid Banach algebra $\Falg{\G , \L}{}{\RR}{}$. 
\end{prop}
\begin{proof} 
Assume that $\G$ is not topologically free, so that there is a non-empty open bisection $V\subseteq \G \setminus X$ for which $r|_V = d|_V$. 
Thus every $b \in C_c(V , \L) \subseteq \Falg{\G , \L}{}{\RR}{}$ commutes with $C_0(X)$ for any algebra $\Falg{\G , \L}{}{\RR}{}$ that contain the spaces $C_c(U,\LL)$,  $U\in \Bis(\G)$ (so for instance an exotic Banach algebra). 
Hence $C_0(X)$ is not maximal abelian in $\Falg{\G , \L}{}{\RR}{}$. 

Now assume that $\G$ is Hausdorff, $\Falg{\G , \L}{}{\RR}{}$ is a reduced groupoid Banach algebra, and let  $b\in \Falg{\G , \L}{}{\RR}{} \setminus C_0(X)$  commute with all elements in $C_0(X)$. 
Since $j_{\RR} : \Falg{\G , \L}{}{\RR}{} \to C_0(\G , \L)$ is injective and $E^\RR_X(b) := j_{\RR}(b)|_{X}\in C_0(X)$, we get that $c = b - E^\RR_X(b)\in \Falg{\G , \L}{}{\RR}{}$ commutes with elements in $C_0(X)$ and $\supp(j_{\RR}(c))$ has non-empty interior and is contained in $\G\setminus X$. 
Thus there is a non-empty open bisection $V\subseteq \supp(j_{\RR}(c))\subseteq \G\setminus X$ and we have $a * j_{\RR}(c)=j_{\RR}(ac)=j_{\RR}(ca)=j_{\RR}(c) *a$ for all  $a\in C_c(V)$, because $j_{\RR}$ is a $C_0(X)$-bimodule map. 
This implies that $r|_{V}=d|_{V}$. 
Hence $\G$ is not topologically free.
\end{proof}

It is known that, in the untwisted case, topological freeness is equivalent to the generalized intersection property for essential groupoid $C^*$-algebras, see \cite[Theorem 7.29]{Kwa-Meyer}. 
This generalizes classical results for crossed products by Kawamura--Tomiyama~\cite{Kawa-Tomi} and Archbold--Spielberg~\cite{Arch_Spiel}. 
A corresponding result for $L^P$-operator algebra crossed products was proved in \cite[Theorem 5.9]{BK}. 
We now generalize all these results to essential groupoid Banach algebras associated to $P\subseteq [1,\infty]$. 
To this end, we use the following representation, which in the $\ell^2$-context is called the \emph{orbit representation} \cite{Kwa-Meyer}, \emph{augmentation representation} \cite{BCFS}, or \emph{trivial representation} \cite{Renault_Fourier}. 
As in \cite{BK}, we prefer to call it a \emph{$C_0(X)$-trivial representation}.

\begin{lem}\label{lem:orbit_representation} 
For any $p\in[1,\infty]$, we have a representation $\Lambda^{\tr}_p : \Falg{\G}{p}{}{} \to B(\ell^{p}(X))$ where
\[
      \Lambda^{\tr}_p (f)\xi (x) := \sum_{d(\gamma)=x} f(\gamma)\xi (r(\gamma)), \qquad f\in \mathfrak{C}_c(\G) ,\ \xi \in \ell^{p}(X).
\]
\end{lem}
\begin{proof}
Consider the canonical action $h : S \to \PAut(X)$ of any wide inverse semigroup $S\subseteq \Bis(\G)$, cf.\ \cite[Example 2.19]{BKM}. 
One readily sees that 
\[
    \pi^{\tr}(a)\xi(x) := a(x)\xi(x), \qquad 
    v_U^{\tr}\xi(x) := \begin{cases}
                            \xi(h_{U^*}(x)), & x\in r(U), \\
                            0 , & x \notin r(U),
                        \end{cases}
\]
for $a\in C_0(X)$, $\xi \in \ell^p(X)$, $x \in X$, $U\in S$, defines a covariant representation $(\pi^{\tr},v^{\tr})$ of $h$ on $\ell^{p}(X)$ in the sense of \cite[Definition 3.17]{BKM}. By \cite[Proposition 4.23 and Theorem 5.5]{BKM}, $(\pi^{\tr},v^{\tr})$ integrates to 
a $\|\cdot\|_{I}$-contractive homomorphism $\pi^{\tr}\rtimes v^{\tr}:  \mathfrak{C}_c(\G) \to B(\ell^{p}(X))$. 
Hence it extends to the desired representation $\Lambda^{\tr}_p : \Falg{\G}{p}{}{} \to B(\ell^{p}(X))$, see Example~\ref{ex:FullLpAlgebra}.
\end{proof}

Recall that for any non-empty $P\subseteq [1,\infty]$ we denoted by $\Lambda_{P} : F^P(\G , \L) \to F^P_{\red}(\G , \L)$  the canonical \emph{regular representation}, 
see Definition \ref{def:F^P_crossed_products}. 
We denote by $j_{P}^{\ess} : F^P(\G , \LL) \to \D(\G,\LL)$ the canonical contractive linear map, which is the identity on every $C_0(U,\LL)$, $U\in \Bis(G)$, cf.\ Remark~\ref{rem:exotic_and_graded}. 
Following \cite{BK}, we also define the \emph{$C_0(X)$-trivial representation} $\Lambda_{P}^{\tr}$ of $F^P(\G)$ as an extension of the $\ell^{\infty}$-direct sum $\oplus_{p\in P} \Lambda_p^{\tr}$ of representations from Lemma~\ref{lem:orbit_representation}.

\begin{thm} 
\label{thm:topological_freeness_untwisted} 
Let $\G$ be an \'etale groupoid. 
For any non-empty $P\subseteq [1,\infty]$, the following conditions are equivalent:
\begin{enumerate}
    \item\label{enu:topological_freeness_untwisted1} $\G$ is topologically free;
    \item\label{enu:topological_freeness_untwisted2} $C_0(X) \subseteq F^P(\G)$ has the generalized intersection property with the hidden ideal $\ker(j_P^{\ess})$;
    \item\label{enu:topological_freeness_untwisted3} $\ker(\Lambda_P^{\tr}) \subseteq \ker(j_P^{\ess})$ where $\Lambda_{P}^{\tr}$ is the $C_0(X)$-trivial representation of $F^P(\G)$. 
\end{enumerate}
If $\mathfrak{M}_0(\G)=\{0\}$, so for instance when $\G$ is Hausdorff, then $\ker(j_P^{\ess}) = \ker(\Lambda_P)$ and the above conditions are equivalent to
\begin{enumerate}\setcounter{enumi}{3}
    \item\label{enu:topological_freeness_untwisted4} $C_0(X)$ detects ideals in one (and hence all) of $F^1(\G)$, $F^{\infty}(\G)$, $F_{I}(\G)$. 
\end{enumerate}
\end{thm}
\begin{proof}
\ref{enu:topological_freeness_untwisted1} implies \ref{enu:topological_freeness_untwisted2} by Theorem~\ref{thm:Top_Free_Intersection_Property}, and \ref{enu:topological_freeness_untwisted2} implies \ref{enu:topological_freeness_untwisted3} because $\ker(\Lambda_P^{\tr})$ is an ideal in $F^P(\G)$ satisfying $\ker(\Lambda_P^{\tr})\cap C_0(X) = \{ 0 \}$. 
To show that \ref{enu:topological_freeness_untwisted3} implies \ref{enu:topological_freeness_untwisted1}, assume that $\G$ is not topologically free. 
Thus there is a non-empty open bisection $U\subseteq \G \setminus X$ with $r|_U = d|_U$. 
Since $U\cap X = \emptyset$, we get $\E_X(f) = 0$ for all $f\in C_0(U)$. 
Choose any non-zero $f\in C_c(U)$ and define $f_0\in C_c(d(U)) \subseteq C_0(X)$ by $f_0( d(\gamma) ) := f(\gamma)$ for $\gamma\in U$.
Since $r|_U = d|_U$, both $ \Lambda^{\tr}_P (f)$ and $\Lambda^{\tr}_P (f_0)$ act on each subspace $\ell^p(X)$, $p\in P$, by pointwise multiplication with the same function $f_0$. 
Hence $f - f_0\in \ker\Lambda^{\tr}_P$. 
For the canonical essential expectation $\E_X^{P} : F^P(\G)\to \D(X)$ we have $\E_X^{P}(f - f_0) = - f_0 \neq 0$.
Hence $f - f_0 \notin \ker(j_P^{\ess})$, cf.\ Proposition~\ref{prop:local_expectations_and_local_j_map}. 
Thus \ref{enu:topological_freeness_untwisted2} fails.

This proves that \ref{enu:topological_freeness_untwisted1}--\ref{enu:topological_freeness_untwisted3} are equivalent. 
These conditions are independent of the choice of $P$ because \ref{enu:topological_freeness_untwisted1} is. 
Assuming $\mathfrak{M}_0(\G)=\{0\}$, we have $\ker(j_P^{\ess}) = \ker j_{P}=\ker(\Lambda_P)$, by Remark~\ref{rem:essential_vs_reduced} and Proposition~\ref{prop:regular_disintegrated}. 
For a non-empty $Q\subseteq \{1,\infty\}$, we have $\ker(\Lambda_Q)=\{0\}$, by Remark~\ref{rem:L_P_definition_discussion}. 
Hence condition \ref{enu:topological_freeness_untwisted4} is equivalent to \ref{enu:topological_freeness_untwisted2} for $P=Q$, because $F^Q(\G) = F^Q_{\red}(\G) = F^Q_{\ess}(\G)$.  
\end{proof} 

\begin{rem}
When $\G$ is a countable union of open bisections, then $\ker(j_P^{\ess})$ coincides with the kernel of the \emph{essential representation}  $\widetilde{\Lambda}_{P}^{\ess}:F^P(\G , \L) \to \widetilde{F}^{P}_{\ess}(\G , \L)$ given by the $\ell^{\infty}$-direct sum $\oplus_{p\in P} \widetilde{\Lambda}_p$ of representations from Proposition~\ref{prop:essential_L^p_operator_algebra}, see  also Corollary~\ref{cor:essential_L^P_operator_algebra}. 
\end{rem}

\begin{cor}\label{cor:simplicity3}
Assume that $\mathfrak{M}_0(\G)=\{0\}$, which holds when $\G$ is Hausdorff or the topology of $\G$ has a basis of compact regular open sets.
The following conditions are equivalent:
\begin{enumerate}
    \item\label{enu:simplicity31} $\G$ is topologically free and minimal; 
    \item\label{enu:simplicity32} for every twist $\LL$ and every unital inverse subsemigroup $S\subseteq \Bis(G)$ that covers $\G$, every $S$-graded reduced Banach algebra $F_{\RR}(\G,\LL)$ is simple; 
    \item\label{enu:simplicity33} one of the algebras $F^1(\G)$, $F^{\infty}(\G)$, $F_{I}(\G)$ is simple. 
\end{enumerate}
Assume in addition that $\G$ is second countable and amenable, then for any non-empty set $P\subseteq [1,\infty]$, the above are further equivalent to: 
\begin{enumerate} \setcounter{enumi}{3}
    \item\label{enu:simplicity35} $F^P(\G)$ is simple. 
\end{enumerate}
\end{cor}
\begin{proof} 
If $\mathfrak{M}_0(\G)=\{0\}$, then all reduced algebras are essential, see Remark~\ref{rem:essential_vs_reduced}.
Hence \ref{enu:simplicity31} implies \ref{enu:simplicity32} by Theorem~\ref{thm:Top_Free_Intersection_Property}. 
Implication \ref{enu:simplicity32}$\Rightarrow$\ref{enu:simplicity33} is obvious because the algebras $F^1(\G)$, $F^{\infty}(\G)$, $F_{I}(\G)$ are reduced. 
If $\G$ is second countable and amenable, then $F^P(\G)=F^P_{\red}(\G)$, by \cite[Theorem 6.19]{Gardella_Lupini17}, and hence these algebras are reduced. 
Thus also \ref{enu:simplicity32}$\Rightarrow$\ref{enu:simplicity35} in the amenable case. 
Since simplicity implies both the intersection property and minimality, see Lemmas~\ref{lem:simplicity_in_general} and~\ref{lem:invariant_ideals}, we get that either of \ref{enu:simplicity33} or \ref{enu:simplicity35} implies \ref{enu:simplicity31} by Theorem~\ref{thm:topological_freeness_untwisted}. 
\end{proof}

In the twisted case, we have the following characterization of topological freeness that generalizes \cite[Theorem 5.13]{BK}, which is formulated for group actions, and \cite[Proposition 6.4]{Kwa-Meyer2} in the case of a line bundle. 

\begin{thm}
\label{enu:topological_freeness_twisted_intersection} 
Let $(\G,\LL)$ be a twisted \'etale groupoid and let $P\subseteq [1,\infty]$ by any non-empty set.
The following are equivalent: 
\begin{enumerate} 
    \item\label{enu:topological_freeness_twisted_intersection1} $\G$ is topologically free;
    \item\label{enu:topological_freeness_twisted_intersection3} the inclusion $C_0(X)\subseteq F^P_{\ess}(\H, \LL|_{\H})$ has the intersection property for  every open subgroupoid $\mathcal{H}\subseteq \G$ containing $X$;
	\item\label{enu:topological_freeness_twisted_intersection4} the inclusion $C_0(X)\subseteq F^P_{\ess}(\H, \LL|_{\H})$ has the intersection property for every open subgroupoid $\mathcal{H}\subseteq \G$ generated by $X$ and an open bisection $U\subseteq \G\setminus X$. 
\end{enumerate}
\end{thm}
\begin{proof}
\ref{enu:topological_freeness_twisted_intersection1} implies \ref{enu:topological_freeness_twisted_intersection3} by Theorem~\ref{thm:Top_Free_Intersection_Property}, because topological freeness of $\G$ passes to an open subgroupoid $\mathcal{H}\subseteq \G$ containing $X$. 
It is clear that \ref{enu:topological_freeness_twisted_intersection3} implies \ref{enu:topological_freeness_twisted_intersection4}.  
To prove that \ref{enu:topological_freeness_twisted_intersection4} implies \ref{enu:topological_freeness_twisted_intersection1}, assume that \ref{enu:topological_freeness_twisted_intersection1} fails.
Then there is a non-empty bisection $U\subseteq \G \setminus \overline{X}$ with $r|_U = d|_U$.
Then the union $\mathcal{H} := X\cup \bigcup_{k\in \Z} U^{k}$ is an open subgroupoid of $\G$ generated by $X$ and $U$. 
The open set  $U^{0} := d(U) = r(U)$ is $\mathcal{H}$-invariant. 
In particular, the groupoid $\mathcal{H}_0 = \bigcup_{k\in \Z} U^{k}$ is an ideal in $\mathcal{H}$ (that is, $\mathcal{H} \mathcal{H}_0 \subseteq \mathcal{H}_0$ and $\mathcal{H}_0 \mathcal{H} \subseteq \mathcal{H}_0$), and so if $B$ denotes the closure of the image of $\mathfrak{C}_c(\mathcal{H}_0, \L|_{\mathcal{H}_0})$ in $F^P_{\ess}(\H, \LL|_{\H})$, we see that $B$ is an ideal. 
In fact, $B$ is an ideal in $F^P_{\ess}(\H, \LL|_{\H})$ generated by $C_0(U^0)=B\cap C_0(X)$. 

We claim that if $C_0(U^0)\subseteq B$ does not have the intersection property, then $C_0(X)\subseteq F^P_{\ess}(\H, \LL|_{\H})$ does not have it either.
Indeed, assume that there is a non-zero ideal $I$ in $B$ with $I\cap C_0(U^0)=\{0\}$. 
Using that an approximate unit in $C_0(U_0)$ is an approximate unit in $B$, and that $B$ is an ideal in $F^P_{\ess}(\H, \LL|_{\H})$, one concludes that $I$ is an ideal in $F^P_{\ess}(\H, \LL|_{\H})$ and $I\cap C_0(X)=I\cap C_0(U^0)=\{0\}$. 
Hence $C_0(X)$ does not detect ideals in $F^P_{\ess}(\H, \LL|_{\H})$.
This proves our claim. 

Thus to show that \ref{enu:topological_freeness_twisted_intersection4} fails, which finishes the proof, it is enough to show that $C_0(U^0)\subseteq B$ does not have the intersection property. 
We have the following two cases:

1) If $U^{n} \cap U^0 = U^{n} \cap X = \emptyset$ for all $n>0$, then $\{ U^{n} \}_{n\in \Z}$ are pairwise disjoint and $\mathcal{H}_0\cong U^0\times \Z$ is the transformation groupoid for the trivial $\Z$-action (since $U$ is a bisection with $r|_U = d|_U$, we have $U^{n}:=\{\gamma^n: \gamma\in U\}$). 

2) If $U^{n}\cap U^0\neq\emptyset$ for some $n>0$ then, assuming $n$ is the smallest such number and replacing $U$ by $ U (U^{n}\cap U^0)$, we may assume that $\{ U^{k} \}_{k\in \Z_n}$ are pairwise disjoint and $U^{n} = U^0$. 
Then $\mathcal{H}_0 \cong U^0 \times \Z_n$ is the transformation groupoid for the trivial $\Z_n$-action.

Thus we may assume that $\mathcal{H}_0$ is the transformation groupoid for the trivial $G$-action where $G=\Z$ or $G=\Z_n$. 
Then $\mathcal{H}$ is a transformation groupoid for a partial $G$-action generated by $\textrm{id}_{U^0}$, cf.\ Example~\ref{ex:groupoid_model_for_partial_action_crossed_product} below. 
In particular, $\H$ is Hausdorff and $F^P_{\ess}(\H, \LL|_{\H}) = F^P_{\red}(\H, \LL|_{\H})$ is a reduced groupoid algebra. 
This implies that $B\subseteq F^P_{\red}(\H, \LL|_{\H})$ can be naturally identified with the reduced algebra $F^P_{\red}(\H_0, \LL|_{\H_0})$. 
Thus $B$ can be identified with a reduced $L^P$-operator algebra crossed product studied in \cite{BK}. 
Since the group $G$ is amenable, the trivial action is amenable and $B$ is in fact also a full $L^P$-operator algebra crossed product by \cite[Corollary 4.14]{BK}. 
Since the $G$-action is trivial, by passing to a subgroup of $G$ if necessary, we get that $C_0(U^{0})\subseteq B$ does not have the intersection property by \cite[Corollary 5.13]{BK}. 
Formally, in \cite{BK} it was assumed that $\F=\C$, but a quick inspection of the proofs of the invoked results show that they remain valid  for $\F=\R$, also due to our definition of the universal algebras in the real case. 
\end{proof}

\section{Pure infiniteness}
\label{sec:PureInfiniteness}

Cuntz~\cite{Cuntz} defined pure infiniteness for a simple $C^*$-algebra $B$ as the property that every nonzero hereditary $C^*$-subalgebra of $B$ has an infinite projection. 
The success of this notion inspired a number of attempts to transfer it to other contexts. 
In particular, Phillips~\cite{PhLp2a}, see also \cite{Daws_Horwath}, introduced pure infiniteness for unital simple Banach algebras, which agrees with the pure infiniteness for unital simple rings from \cite{AGP}. 
Recently, this definition was extended to arbitrary (not necessarily unital) Banach algebras in \cite{cortinas_Montero_rodrogiez} using an approximation condition mimicking the Cuntz preorder. 
Here we will define pure infiniteness for simple Banach algebras using an algebraic condition from \cite{AGP}, \cite{AAM}. 

Recall that idempotents $e, f$ in a ring $B$ are equivalent (written $e\sim f)$, provided there are $x,y\in B$ such that $e=xy$ and $f=yx$. 
For any idempotents, we write $e\leq f$ if $ef = fe = e$, and $e<f$ if in addition $e\neq f$. 
We call an idempotent $e\in B$ \emph{infinite} if there exists an idempotent $f\in R$ such that $e \sim f < e$. 
Note that $\sim$ is an equivalence relation that preserves infiniteness. 
A \emph{simple ring} $B$ is called \emph{purely infinite}, see \cite[Definition 3.1.8]{AAM}, \cite[Definition 1.2]{AGP}, if 
\begin{itemize}
    \item[($\dagger$)]\label{enu:pure_infiniteness_def} for every $y\in B\setminus\{0\}$, there are $x, z\in B$ such that $xyz$ is an infinite idempotent in $B$.
\end{itemize}

\begin{defn}\label{def:purely_infinite}
We say that a simple Banach algebra $B$ is \emph{purely infinite} if it satisfies ($\dagger$). 
\end{defn}

\begin{rem}\label{rem:on_condition_dagger}
Condition ($\dagger$) holds if and only if every non-zero left  ideal (not necessarily closed) in $B$ contains an infinite idempotent, see \cite[Proposition 3.1.7]{AAM}. 
Thus ($\dagger$) is equivalent to the apparently stronger and unsymmetrical condition that, for every $y\in B\setminus\{0\}$ there is $x\in B$  such that $xy$  is an infinite idempotent in $B$. A version of this remark where left ideals are replaced by right ideals is also true.
\end{rem}


\begin{prop}\label{prop:various_purely_infnite}
Let $B$ be a simple Banach algebra. 
\begin{enumerate}
    \item\label{enu:various_purely_infnite2} If $B$ is unital then the following conditions are equivalent:
        \begin{enumerate}
            \item\label{enu:various_purely_infnitea} $B$ is a purely infinite Banach algebra;
            \item\label{enu:various_purely_infnitec} $B$ is a purely infinite simple ring in the sense of \cite{AGP}, \cite{AAM};
		  \item\label{enu:various_purely_infniteb} $B$ is purely infinite in the sense of \cite{PhLp2a}, \cite{Daws_Horwath}, that is, $B\not\cong \F$ and for every $y\in B\setminus\{0\}$ there are $x,z\in B$ such that $xyz=1$.   
        \end{enumerate}
	\item\label{enu:various_purely_infnite2.5} If $B$ has a bounded (two-sided) approximate unit consisting of idempotents, then the following conditions are equivalent:
        \begin{enumerate}
            \item\label{enu:projection_purely_infnitea} $B$ is a purely infinite Banach algebra;
            \item\label{enu:projection_purely_infniteb} $B$ is purely infinite in the sense of \cite{cortinas_Montero_rodrogiez}, that is, $B\not\cong \F$ and for every $a,b\in B\setminus\{0\}$ there are sequences $(x_n)_{n=1}^\infty, ( z_n )_{n=1}^\infty$ in $B$ such that $x_n b z_n\to a$; 
		  \item\label{enu:projection_purely_infnitec} for every idempotent $e\in B\setminus\{0\}$, the unital Banach algebra $eBe$ is purely infinite in the sense of the equivalent conditions in \ref{enu:various_purely_infnite2}.
        \end{enumerate}
    \item\label{enu:various_purely_infnite1} If $B$ is a $C^*$-algebra, then $B$ is a purely infinite Banach algebra if and only if it is purely infinite in the sense of Cuntz \cite{Cuntz}, that is, every non-zero hereditary $C^*$-subalgebra of $B$ contains a partial isometry $u$ with $u u^* < u^*u$.
\end{enumerate}
\end{prop}
\begin{proof} 
\ref{enu:various_purely_infnite2}. 
The conditions \ref{enu:various_purely_infnitea} and \ref{enu:various_purely_infnitec} differ only in that one assumes algebraic simplicity and the other analytic simplicity, but for unital Banach algebras these are equivalent by openness of the set of invertible elements. 
The equivalence \ref{enu:various_purely_infnitec}$\Leftrightarrow$\ref{enu:various_purely_infniteb} holds by \cite[Theorem 1.6]{AGP}. 

\ref{enu:various_purely_infnite2.5}. 
The implications \ref{enu:projection_purely_infnitea}$\Rightarrow$\ref{enu:projection_purely_infnitec} and \ref{enu:projection_purely_infniteb}$\Rightarrow$\ref{enu:projection_purely_infnitec} hold in general.
Indeed, for \ref{enu:projection_purely_infniteb}$\Rightarrow$\ref{enu:projection_purely_infnitec} see \cite[Corollary 3.4]{cortinas_Montero_rodrogiez}, and for \ref{enu:projection_purely_infnitea}$\Rightarrow$\ref{enu:projection_purely_infnitec} assume \ref{enu:projection_purely_infnitea} and let $e\in B\setminus\{0\}$ be an idempotent. 
Then $eBe$ is a unital Banach algebra, and for any nonzero $y\in eBe$, there is $x\in B$ such that $xy\in Be$ is an infinite projection in $B e$. 
Then $xy=xye\sim exy=exey \in eBe$, and so for $x':=exe\in eBe$ the product $x'y$ is  an infinite projection in $eBe$.
Hence $eBe$ is purely infinite, see Remark~\ref{rem:on_condition_dagger}.

Now let us assume \ref{enu:projection_purely_infnitec} and that $B$ has a bounded two-sided approximate unit consisting of idempotents. 
By assumption $B\neq \F$. 
Let $a,b\in B\setminus\{0\}$. 
Then we may choose a sequence $( e_n )_n$ of idempotents in $B$ such that $e_n a e_n\to a$ and $e_n b e_n\to b$, and by \ref{enu:projection_purely_infnitec} for each $n$ we may find $x_n, z_n\in e_n B e_n$ such that $x_n (e_n b e_n) y_n = e_n a e_n$. 
Then $x_n b z_n\to a$, which proves \ref{enu:projection_purely_infniteb}.
Now let $y\in B\setminus\{0\}$. 
Then there is an idempotent $e\in B$ such that $eye\in eBe$ is nonzero. 
By \ref{enu:projection_purely_infnitec} there are $x,z\in eBe$ such that $xeyez$ is an infinite idempotent in $eBe$. 
This proves \ref{enu:projection_purely_infnitea}.

\ref{enu:various_purely_infnite1}. 
Assume $B$ is purely infinite in the sense of Cuntz~\cite{Cuntz}. 
Let $y\in B\setminus\{0\}$. 
We may assume that $\|y\|=1$.  
Put $a:=y^*y$ and let $f$ be a continuous function which takes the value $0$ on $(-\infty, 1/2]$ and the value $1$ on $[1, +\infty)$, and is linear on $[1/2, 1]$. 
Then  $af(a)\geq f(a)/2>0$  in $\overline{aBa}\subseteq \overline{y^*By}$. 
By assumption, there is a partial isometry $u\in \overline{f(a) Bf(a)}$ such that $u u^* < u^*u$. 
Writing $u=f(a)b$ with $b\in \overline{f(a) Bf(a)}$, we get $u^*a u=b^*f(a)a f(a)b\geq 1/2 \cdot b^*f(a) f(a)b = 1/2 \cdot u^*u$.
Thus $u^*a u$ is invertible in the unital algebra $\overline{u^*Bu}$ with $u^*u$ being a unit. 
Denote the inverse of $u^*a u$  by $z$, and put $w:=zu^*a=zu^*y^* y \in By$. 
Then $uw\in By$ is an idempotent equivalent to the infinite idempotent $wu=u^*u$. 
Hence $uw\in By$ is infinite.

Conversely, assume $B$ is purely infinite in the sense of Definition~\ref{def:purely_infinite} and let $y\in B\setminus\{0\}$. 
Then there is an infinite idempotent $e$ in $By$. 
Every idempotent $e$ in a $C^*$-algebra is similar (in fact even homotopy equivalent) to a projection $p$ with $p=p e$, see the proof of \cite[Proposition 4.6.2]{Bl3}. 
Since $p\in Be\subseteq By$ and $p$ is self-adjoint, we get $p\in \overline{y^*By}$. 
Since $e$ is an infinite element, it follows that $p$ is an infinite projection. 
Indeed, $p$ is an infinite element because it is equivalent to an infinite element $e$, and so there is an idempotent $f$ with $p \sim f < p$. 
Then, as mentioned above, there is a projection $q\sim f$ with $q=qf$. 
Thus we have $q\sim p$ in $\overline{y^*By}$ and so they are Murray--von Neumann equivalent, see \cite[Proposition 4.6.4]{Bl3}.
Using that $f < p$ and $q=qf$, we get $q = qf = q(fp) = (qf)p = qp$. 
Therefore $q\leq p$ (because $p$ and $q$ are self-adjoint) and in fact $q<p$ (because $q=p$ would imply $f=p$).
\end{proof}

\begin{rem} 
Without the standing assumption in  \ref{enu:various_purely_infnite2.5} above, we do not know the relationship between pure infiniteness in our sense and that of \cite{cortinas_Montero_rodrogiez}. 
When $B$ is a simple unital $C^*$-algebra, the equivalence between the Cuntz definition and the condition in \ref{enu:various_purely_infnite2}(b) is well known, see \cite[Proposition 6.11.5]{Bl3}. 
\end{rem}

We now consider two known pure infiniteness criteria introduced for $C^*$-algebras associated to Hausdorff groupoids. 
We generalize them to twisted Banach algebras associated to not necessarily Hausdorff groupoids. 
The first criterion was introduced for group actions on compact spaces \cite{Laca-Spielberg}, \cite{Jolissaint-Robertson} and then generalized to groupoids with compact unit spaces by Suzuki~\cite{Suzuki}.
Recall that $S(\LL)\subseteq \Bis(\G)$ is the wide inverse semigroup of bisections on which the bundle $\LL$ is trivial. 
Thus $S(\LL)= \Bis(\G)$ when the twist comes from a cocycle. 

\begin{defn}[{cf.\ \cite[Definition 3.2]{Suzuki}}] \label{def:n_feeling} 
Let $S\subseteq \Bis(\G)$ and $n\in \N$. 
We say that $\G$ is \emph{$n$-filling with respect to $S$} if $X$ is \emph{compact, infinite} and for every non-empty open set $U\subseteq X$ there are $W_1, \ldots , W_n\in S$ with $\bigcup_{i=1}^n r(W_i U) = X$. 
\end{defn}

\begin{lem}[{cf.\ \cite[Lemma 3.10]{Suzuki}}] \label{lem:Suzuki}
Let $\G$ be $n$-filling  with respect to a unital inverse semigroup $S\subseteq S(\LL)$ covering $\G$. 
For any $b\in C(X)^+$ and any $\varepsilon >0$, there is $c\in \mathfrak{C}_c(\G,\LL)$ such that $\| c \|_{\max}^S \leq n$, $c^* * b* c\in C(X)$ and $c^* *b *c \geq \|b\|_{\infty}-\varepsilon$. 
\end{lem}
\begin{proof}
We may assume that $\|b\|_{\infty}=1$. 
Using the $n$-filling property, exactly as in the proof of \cite[Lemma 3.10]{Suzuki}, one may construct mutually disjoint non-empty subsets $U_1,\ldots,U_n$ of $U := \{ x \in X: b(x)\geq 1-\varepsilon \}$, relatively compact bisections $W_1,\ldots , W_n\in \Bis(\G)$, where each $W_i$ is a subset of an element of $S$, and closed subsets $K_i\subseteq W_i$, $i=1, \ldots , n$, such that 
\[
    \s(W_i) \subseteq U_i,\ i=1,\ldots,n, \qquad \text{and} \qquad \bigcup_{i=1}^n r(K_i)=X.
\]
Now, since the bundle over each $W_i$ can be trivialized, we may pick sections $f_i\in C_c(W_i,\LL)$ such that $\| f_i \|_{\infty} = 1$ and $|f_i|_{K_i}|=1$, for $i=1, \ldots ,n$. 
Putting $c := \sum_{i=1}^{n} f_i^{*}$, we have $\| c \|_{\max}^S \leq \sum_{i=1}^{n} \| f_i^{*} \|_{\infty} = n$, and since $\{ W_i \}_{i=1}^n$ are bisections with mutually disjoint domains $\{d(W_i)\}_{i=1, \ldots , n}$, we get $c^* *b * c= \sum_{i=1} f_i^* * b* f_i\in C(X)$. 
Since $d(W_i)\subseteq U$, for $i=1, \ldots , n$ and $\bigcup_{i=1}^n r(K_i)=X$, we further obtain $c^* *b * c\geq 1-\varepsilon$.
\end{proof}


\begin{thm}\label{thm:filling_implies_pure_infinite}
Assume that $\G$ is topologically free and $n$-filling with respect to a unital inverse semigroup $S\subseteq S(\LL)$  covering $\G$. 
Then every $S$-graded essential Banach algebra $\FR(\G , \LL)$ is purely infinite simple (and unital).
\end{thm}
\begin{proof} 
Let $y\in \FR(\G , \LL) \setminus \{0\}$. 
By Proposition~\ref{prop:various_purely_infnite}\ref{enu:various_purely_infnite2}, it suffices to show that there are $x, z \in \FR(\G , \LL)$ such that $x y z$ is invertible. 
By Proposition~\ref{prop:local_expectations_and_local_j_map}, there is $U\in S(\LL)$ and $a_U\in C_c(r(U))$ such that $\E^{\RR}_U(a_U y)\neq 0$. 
Replacing $y$ with $a_U y$ and scaling, we may assume that $y\in C_c(r(U)) \FR(\G , \LL)$ and $\| \E^{\RR}_U(y) \| = 1$. 
Let $\varepsilon >0$ be such that $\varepsilon < 1/(2n^2+2)$. 
Let $i_{\RR} : \mathfrak{C}_c(\G, \LL)\to\FR(\G , \LL)$ be the canonical homomorphism, which is isometric on the subspaces $C_c(V,\LL)$, $V\in S$. 
The argument from the beginning of the proof of \cite[Theorem 1.2]{Jolissaint-Robertson} shows that we may find $f \in C_c(r(U)) \mathfrak{C}_c(\G, \LL)$ with $\| y - i_{\RR}(f)\|_{\RR} < \varepsilon$ and $\|\E^{\RR}_{U} (i_{\RR}(f )) \| = \| q_U ( f \circ r|_U^{-1} ) \| = 1$. 
By Lemma~\ref{lem:approx_conditional_exp_twisted}, there is a contractive $b \in C_c(U^*,\LL)$ such that $|b * f|_X |= |f\circ r|_{U}^{-1}|$. 
Thus, replacing $y$ with $b y$ and $f$ with $b f$, we may assume that $U=X$, that is, we have $f \in\mathfrak{C}_c(\G, \LL)$ with
\[
    \| y - i_{\RR}(f)\|_{\RR} < \varepsilon 
\]
and $\| \E^{\RR}_{X}(i_{\RR}(f)) \| = \| q_X(f|_X) \| = 1$.
By Lemma~\ref{lem:pinching_property} there are norm one functions $a,b \in C(X)$ such that $a, b  f|_X \in C(X)^+$, $1=\| q_{X} ( f|_{X} ) \| \leq \| a (b  f|_X) a \|_{\infty} + \varepsilon$ and 
\[
    \| a b i_{\RR}(f) a - a (b f|_X ) a \|_{\RR} \leq \| a (b f) a - a (b f|_X ) a \|_{\max}^S \leq \varepsilon.
\]
In particular, $\|a (b  f|_X ) a\|\geq 1-\varepsilon$. 
Applying Lemma~\ref{lem:Suzuki} to $a (b  f|_X )a$, we get $c\in \mathfrak{C}_c(\G, \LL)$ such that $\| c \|_{\max}^{S} \leq n$ and $g:=c^* (a (b  f|_X ) a) c \in C(X)$ satisfies $g \geq 1 - 2\varepsilon$. 
In particular, $d:=i_{\RR}(c)\in  \FR(\G , \LL)$ satisfies $\|d\|_{\RR}\leq \| c \|_{\max}^{S} \leq n$ and $g $ is invertible in $C(X) \subseteq  \FR(\G , \LL)$ where $\| g^{-1} \|_{\infty} \leq (1 - 2\varepsilon)^{-1}$. 
Thus every $h\in  \FR(\G , \LL)$ with $\| h - g \|_{\RR}< 1-2\varepsilon$ is invertible in $\FR(\G , \LL)$. 
Using the above displayed inequalities, together with $\|d\|_{\RR}\leq n$, $\|a\|, \|b\|\leq 1$, and our choice of $\varepsilon$, we get
\[
\begin{split}
    \big\| d^*  (a b y a)  d - g \big\|_{\RR} &\leq \big\| d^* \big( a b y a   - a b i_{\RR}(f) a \big) d \big\|_{\RR}+ \big\| d^* \big( a b i_{\RR}(f) a   - a b f|_X a\big) d \big\|_{\RR} \\
        &\leq n^2 \varepsilon + n^2 \varepsilon = 2 n^2 \varepsilon <  1-2\varepsilon.
\end{split}
\]
Hence $d^* ( a d y a )  d$ is invertible in $\FR(\G , \LL)$. 
\end{proof}

We now turn to the second criterion which is due to Anantharaman-Delaroche and is sometimes easier to apply. 
Namely, we generalize the classical result \cite[Proposition 2.4]{A-D} from (untwisted) groupoid $C^*$-algebras to the twisted Banach algebra case. 
Recall that $h_{W} = r\circ \s|_{W}^{-1} : \s(W)\to r(W)$ is the homeomorphism induced by $W\in\Bis(\G)$.

\begin{defn}[{cf.\ \cite[Definition 2.1]{A-D}}] \label{def:locally contractive}
Let $S\subseteq \Bis(\G)$. 
We call $\G$ \emph{locally contracting with respect to $S$} if for every non-empty open $U\subseteq X$ there are open $V \subseteq U$ and $W\in S$ with $\overline{V} \subseteq d(W)$ and $h_{W}(\overline{V}) \subsetneq V$ (equivalently $r(W\overline{V})\subsetneq V$).
\end{defn}

\begin{rem} 
If $S=S(\LL) = \Bis(\G)$, so for instance when the twist comes from a $2$-cocycle (equivalently $\LL$ is trivial as a line bundle, see Example~\ref{ex:cocycle_twist}), then the properties in Definitions~\ref{def:n_feeling} and~\ref{def:locally contractive} do not depend on the twist, and if they hold, we simply call the groupoid $\G$ \emph{$n$-filling} \cite{Suzuki} or \emph{locally contractive} \cite{A-D}, respectively. 
\end{rem}

\begin{rem}
The $n$-filling property implies that $\G$ is minimal, while locally contracting does not. 
Assuming minimality, the relationship between $n$-filling and locally contracting is not completely clear. 
However, when $\G$ is ample, it is known that $n$-filling does not depend on $n$ (is equivalent to $1$-filling) and it implies locally contracting, see \cite[Theorem 5.8]{Ma} (it also implies paradoxicality of all non-empty compact open sets, see \cite[Proposition 7.7]{Rainone-Sims}). 
\end{rem}

\begin{lem}\label{lem:Anantharaman-Delaroche}
Assume that $\G$ is locally contracting with respect to a unital inverse semigroup $S\subseteq S(\LL)$ covering $\G$.  
For any bump function $a\in C_c(X)$, there is an idempotent $p\in \mathfrak{C}_c(\G, \LL)$ with $p=ap=pa$ and $\|p\|_{\max}^{S}\leq 4$, and there is $u\in \mathfrak{C}_c(\G, \LL)$, with $\|u\|_{\max}^{S}\leq 8$, such that $u^*u = p$ and $uu^* < p$ (so $p$ is infinite). 
Moreover, these relations remain valid for the images of $u$ and $p$ in any $S$-graded essential Banach algebra $\FR(\G , \LL)$ of $(\G,\LL)$. 
\end{lem}
\begin{proof} 
We adapt the argument from the proof of \cite[Proposition~2.2]{A-D}, see also \cite[Theorem 4.4.2]{Sims}. 
If $X$ is not compact, we let $\widetilde{\mathfrak{C}_c(\G, \LL)}$ be the minimal unitization of the algebra $\mathfrak{C}_c(\G, \LL)$. 
Then $1\C+C_c(X)=\widetilde{C_c(X)}\subseteq \widetilde{\mathfrak{C}_c(\G, \LL)}$ is naturally a subalgebra and $\mathfrak{C}_c(\G, \LL)\subseteq \widetilde{\mathfrak{C}_c(\G, \LL)}$ is a two-sided ideal.
Moreover, every positive function in $\widetilde{C_c(X)}$ has a square root in $\widetilde{C_c(X)}$. 

Let $a\in C_c(X)$ be a bump function. 
Let $U$ be a non empty open set such that $a|_{U}\equiv 1$. 
Take $V \subseteq U$ and $W\in S$ with $\overline{V} \subseteq \s(W)$ and $h_{W}(\overline{V}) \subsetneq V$. 
Find an open $V_0$ such that $h_{W}(\overline{V}) \subseteq V_0\subseteq \overline{V_0}\subsetneq V$.
Fix a positive norm one function $b\in C_c(V_0)$ which is equal to $1$ on $h_{W}(\overline{V})$. 
Then $1-b^2\in \widetilde{C_c(X)}$ vanishes on  $h_{W}(\overline{V})$ and $b\circ  h_{W^*} \in C_c(h_{W}(V_0))$, so in particular $(1-b^2)^{1/2}(b \circ h_{W^*})= 0$. 
Since the bundle $\LL$ can be trivialized on $W$, we may pick a unitary section $c_W\in C(W,\LL)$. 
Put $x := c_W b= ( b \circ h_{W^*})c_W$. 
Then $x^*x= b^2 \in C_c(V)$ and $x x^*= b^2\circ h_{W^*} \in C_c(h_{W}(V))$. 
It follows that $(1-x^*x)^{1/2}x=0$ ($x$ is the so called \emph{scaling element}). 
Therefore
\[
    v := x + ( 1 - x^*x )^{1/2} \in \widetilde{\mathfrak{C}_c(\G,\LL)}
\]
satisfies $v^*v=1$. 
Thus $vv^*$ is an idempotent and also
\[
    p := 1 - vv^* = x^*x - xx^* -  x(1 - x^*x)^{1/2} - ( 1-x^*x )^{1/2}x^*\in \mathfrak{C}_c(\G,\LL)
\] 
is an idempotent with $\|p\|_{\max}^S\leq 4$ (as the four elements in the above expression are $\|\cdot\|_{\max}^S$ contractive). 
Since $a|_{U}\equiv 1$  and $\s(\supp(x)), r(\supp(x))\subseteq V\subseteq U$, we have $ax = xa = x$. 
By the above formula for $p$, this implies that $p=ap=pa\in a\cdot \mathfrak{C}_c(\G,\LL)\cdot a$. 

Since $h_{W}(\overline{V_0}) \subsetneq h_{W}(V)$, the open set $O := h_W(V)\setminus h_{W}(\overline{V_0})$ is non-empty, and $b\circ h_{W^*}$ vanishes on $O$. 
The function $x^*x-xx^*=b^2 - b^2\circ h_{W^*}$ on $O$ is constant  equal to $1$, and $(1-x^*x)^{1/2}$ vanishes on $O$. 
Thus $p|_O \equiv 1$ and therefore for any $S$-graded essential Banach algebra $\FR(\G , \LL)$, using the canonical homomorphism $i_{\RR}: \Falg{\G , \LL}{S}{}{} \to \FR(\G , \LL)$, we have 
\[
    \| \E_X^{\RR} (i_{\RR}(p) ) \| = \| q_X(p|_X) \| \geq \| q_O(p|_O) \| = 1 ,
\]
and so $i_{\RR}(p)$ is non-zero. 

To prove that in fact $i_{\RR}(p)$ is infinite we take any bump function $c \in C_c(X)$ with support in $O$, so that we get $c x^*x=c$ and $c x=0$.  
Applying the above reasoning to $c$ (instead of $a$), we get a scaling element $y$ with $cy=yc=y$. 
This implies $y x^*x= x^*xy=y$ and $yx=0$. 
Hence $py=yp=y$. 
As a consequence $p$ commutes with the elements in the Banach subalgebra of $\widetilde{C_c(X)}$ generated by the positive function $1-y^*y$, and in particular with $(1-y^*y)^{1/2}$. 
Thus $p$ commutes with $w:=y+(1-y^*y)^{1/2}\in \widetilde{\mathfrak{C}_c(\G,\LL)}$. 
Therefore putting $u:=p w\in \mathfrak{C}_c(\G,\LL)$, and using that $w^*w=1$ (which holds because $y$ is scaling), we get $u^*u=p$ and 
\[
    uu^*= p - ( y^*y-yy^*-y(1-y^*y)^{1/2}-(1-y^*y)^{1/2}y^*) .
\]
The reasoning showing that $i_{\RR}(p)\neq 0$ shows that $i_{\RR}$ applied to the term in the brackets is non-zero. 
Hence the image of $p$ in $\FR(\G , \LL)$ is infinite. 
Clearly, $\| u \|_{\max}^S \leq \|p\|_{\max}^S \|w\|_{\max}^S \leq 8$. 
\end{proof}

\begin{prop}\label{prop:Anantharaman-Delaroche}
Assume $\G$ is topologically free and locally contracting with respect to a unital inverse semigroup $S\subseteq S(\LL)$ covering $\G$. 
Then every non-zero left ideal in every $S$-graded essential Banach algebra $\FR(\G , \LL)$ contains an infinite idempotent. 
\end{prop}
\begin{proof}
Let $y\in \FR(\G , \LL) \setminus \{0\}$. 
We show that $\FR(\G , \LL) y$ contains an infinite idempotent. 
Let $0<\varepsilon < \frac{1}{192}$. 
Arguing as in the beginning of the proof of Theorem~\ref{thm:filling_implies_pure_infinite}, we may assume that $\| \E^{\RR}_X(y )\| = 1$ and that there is $f\in \mathfrak{C}_c(\G,\LL)$ with $\| y - i_{\RR}(f) \|_{\RR} < \varepsilon$ and $\| E^{\RR}_{X} (i_{\RR}(f)) \| = \| q_X(f|_X) \| = 1$. 
By Lemma~\ref{lem:pinching_property} there is a contractive $b\in C_c(X)$ and a bump function $a\in C_c(X)$ such that $b  f|_X \in C_c(X)^+$ and
\[
    \| a^2 - a (b f|_X) a \|_{\infty} \leq \varepsilon  \quad  \text{ and } \quad \| a (b f) a-  a (b f|_X)  a \|_{\max}^S \leq \varepsilon. 
\]
By Lemma~\ref{lem:Anantharaman-Delaroche}, there is $u\in \mathfrak{C}_c(\G,\LL)$ such that $\|u\|_{\max}^S\leq 8$, $p:=u^*u$ is an idempotent with $p=ap=pa$, $\|p\|_{\max}^S\leq 4$, $uu^* <p$ and $i_{\RR}(uu^*) < i_{\RR}(p)$. 
In particular, $i_{\RR}(p)$ is an infinite idempotent in $a \FR(\G , \LL) a$. 
Using these relations we get 
\[
\begin{gathered}
    \| p- u^* (b  f|_X) u \|_{\max}^S = \| u^* a^2 u - u^* a (b f|_X) a u \|_{\max}^S \leq  64 \varepsilon < 1/3, \\
    \| u^* b f u - u^*(b f|_X)u \|_{\max}^S = \| u^* a (b f)  a  u- u^* a (b f|_X)a  u \|_{\max}^S \leq 64 \varepsilon <1/3, \\ 
    \| i_{\RR}(u^*) (by) i_{\RR}(u) - i_{\RR}(u^* b f u) \|_{\RR} \leq 16 \| y - i_{\RR}(f) \|_{\RR} < 64 \varepsilon < 1/3 .
\end{gathered}
\]
This implies that $\| i_{\RR}(p) - i_{\RR}(u^*)(b y)i_{\RR}(u) \|_{\RR} <1$ and therefore $i_{\RR}(u^*)(b y)i_{\RR}(u)$ is invertible in the Banach algebra $i_{\RR}(p) \FR(\G , \LL) i_{\RR}(p)$. 
We denote its inverse by $c$ and we put $w := c\cdot i_{\RR}(u^*)(b y) \in \FR(\G , \LL) y$. 
Then $w i_{\RR}(u) = i_{\RR}(p)$ is an infinite idempotent and therefore the equivalent idempotent $i_{\RR}(u) w \in \FR(\G , \LL) y$ is also infinite. 
\end{proof}

\begin{thm}\label{thm:Anantharaman-Delaroche}
Let $(\G,\LL)$ be a twisted \'etale groupoid where $\G$ is topologically free and minimal. If $\G$ is locally contracting with respect to a unital subsemigroup $S\subseteq S(\LL)$ covering $\G$, then every $S$-graded essential Banach algebra $\FR(\G , \LL)$ of $(\G,\LL)$ is purely infinite simple.
\end{thm}
\begin{proof}
Combine Proposition~\ref{prop:Anantharaman-Delaroche} and Theorem~\ref{thm:Top_Free_Intersection_Property}.
\end{proof}

\section{Applications and examples}
\label{sec:ApplicationsExamples}

\subsection{Crossed products by twisted partial group actions}
\label{ssec:ExamplesPartialActions}

We use a transformation groupoid model to apply our results to twisted partial group actions. 
In particular, we will generalize and extend some of the results of \cite{BK} from global to partial actions. 

We fix a (topological) \emph{partial action} of a discrete group $G$ on a locally compact Hausdorff space $X$. 
This is a map $\theta : G\to \PHomeo(X)$ such that $\theta_1 = \id_X$, $\theta_{t}^{-1} = \theta_{t^{-1}}$ and the partial homeomorphism $\theta_{ts}$ extends $\theta_{t} \circ \theta_s$  for all $t,s\in G$. 
Thus for each $t\in G$ we have a homeomorphism $\theta_t : X_{t^{-1}} \to X_{t}$ between two open subsets of $X$ and for each $s\in G$ the natural domain $\theta_{t^{-1}}(X_{s^{-1}}\cap X_{t}) = X_{t^{-1}}\cap X_{(st)^{-1}}$ for $\theta_{s} \circ \theta_t$ is contained in $X_{(st)^{-1}}$. 
Fixing $\theta$ is equivalent to fixing an algebraic \emph{partial action} $\alpha : G\to \PAut(C_0(X))$ where $\alpha_t(a) := a \circ \theta_{t}^{-1}$, $a \in C_0( X_{t^{-1}})$, $t\in S$. 
We refer to $\theta$ as \emph{dual} to $\alpha$.

We also fix a  \emph{2-cocycle of the partial action $\alpha$} (or of $\theta$), which is a family of maps $u(s,t) : G \times G \to C_u(X_{s}\cap X_{st})$, $s,t\in G$, satisfying $u(1,t) = u(t,1) = 1$ for $t\in G$, and
\[
    \alpha_r \big( a u(s,t) \big) u(r,st) = \alpha_r(a ) u(r,s) u(rs,t), \qquad a \in C_0(X_{r^{-1}}\cap X_{s}\cap X_{st}),\ r,s,t\in G,
\]
see \cite{Exel_twisted_partial}.
\begin{defn}[{cf.\ \cite[Definitions 2.1, 4.12]{BK}}] 
The \emph{universal Banach algebra crossed product} for the twisted partial action $(\alpha,u)$ on $C_0(X)$ is the Banach $*$-algebra 
\[
    F(\alpha,u) := \{ f \in \ell^1(G,C_0(X)) : f(t)\in C_0(X_{t}),\ t\in G \}
\]
with $\ell^{1}$-norm from $\ell^1(G,C_0(X))$ and operations 
\[
    (f * g)(r) := \sum_{st=r} \alpha_s \Big( \alpha_s^{-1} \big( f(s) \big) g(t) \Big) u(s,t) , \qquad 
    f^*(t) := \alpha_{t} \big( f (t^{-1})^* \big) u(t , t^{-1})^*,
\]
for $f,g \in F(\alpha,u)$ and $s,t,r\in G$. 
For any class $\RR$ of representations of $F(\alpha,u)$, we denote the Hausdorff completion of $F(\alpha,u)$ in the seminorm $\|f \|_{\RR} := \sup \{ \| \psi(f) \| : \psi \in \RR \}$, $f\in F(\alpha,u)$, by $F_{\RR}(\alpha,u)$. 
For any non-empty $P\subseteq [1,\infty]$ we define the \emph{$L^P$-crossed product} for  $(\theta,u)$ as the completion $F^P(\alpha,u) := \overline{F(\alpha,u)}^{\|\cdot\|_{L_P}}$, where $\|f\|_{L^P}$ is the supremum of $\| \psi(f) \|$ over representations $\psi : F(\alpha,u) \to B(L^p(\mu))$, where $p\in P$ and $\mu$ is a measure, and if $\F=\R$ we also assume that $\psi|_{C_c(X)}$ is positive. 
\end{defn}

We briefly describe how to view the algebras $F_{\RR}(\alpha,u)$ as groupoid Banach algebras.
More details, as well as description of representations of $F(\alpha,u)$  in terms of covariant representations of $(\alpha,u)$, can be found in an earlier preprint version of \cite{BKM}, preserved as arXiv:2303.09997v3, see subsection 7.1 therein.
\begin{ex}\label{ex:groupoid_model_for_partial_action_crossed_product} 
 
The \emph{transformation groupoid for $\theta$} is  
\[
    \G_{\theta} := \bigsqcup_{t\in G} \{t\}\times X_{t^{-1}}
\] 
equipped with the product topology from $G\times X$ and the algebraic structure given by 
\[
    r(t,x) := \theta_t(x), \quad d(t,x) := x , \quad \big( s, \theta_t(x) \big) \cdot (t,x) := (s t,x) , \qquad s,t \in G,\ x \in X_{t^{-1}} . 
\]
The \emph{twist associated to $u$} is given by the trivial line bundle $\LL_{u} = \G_{\theta} \times \F$ and the $2$-cocycle $\sigma_{u}$ for $\G_{\theta}$ given by 
$$
\sigma_{u}((s,\theta_t(x)),(t,x))=u(s,t)(\theta_{st}(x)).
$$ 
Let $S_\theta$ be the inverse semigroup consisting of bisections $\{ t \} \times U \subseteq \G_\theta$ where  $U\subseteq X_{t^{-1}}$ is open. 
Then one sees that the map $C_c(\G_{\theta},\LL_{u}) \ni \hat{f} \mapsto f \in C_c(G, C_c(X))$ where $f(t)(x) := \hat{f}(t,\theta_{t^{-1}}(x))$ extends to an isometric isomorphism
\begin{equation}\label{eq:groupoid_model_for_partial_crossed_prod}
    F^{S_\theta}(\G_{\theta} , \LL_{u}) \cong F(\alpha,u).
\end{equation}
Thus for any class $\RR$ of representations of $F(\alpha,u)$ there is a class $\RR_0$ of representations of $F^{S_\theta}(\G_{\theta} , \LL_{u})$ such that \eqref{eq:groupoid_model_for_partial_crossed_prod} induces an isometric isomorphism  $F^{S_\theta}_{\RR_0}(\G_{\theta} , \LL_{u}) \cong F_{\RR}(\alpha,u)$.
\end{ex}

\begin{lem}\label{lem:crossed_products_by_partial_actions}
For any Banach algebra of the form $F_{\RR}(\alpha,u)$, the following are equivalent:
\begin{enumerate}
    \item $\|f(1)\|_{\infty}\leq \|f\|_{\RR}$ for all $f\in F(\alpha,u)$, and so $F(\alpha,u) \ni f\mapsto f(1) \in C_0(X)$ extends to a contractive map $E^{\RR} : \FR(\alpha,u)\to C_0(X)$; 
    \item $\max_{t\in G}\|f(t)\|_{\infty} \leq \|f\|_{\RR}$ for all $f\in F(\alpha,u)$, and so for any $t\in G$, $F(\alpha,u)\ni f\mapsto f(t) \in C_0(X_t)$ extends to a contractive map $E^{\RR}_t : F_{\RR}(\alpha,u) \to C_0(X_t)$; 
    \item the inclusion $F(\alpha,u) \subseteq C_0(\G_{\theta})$ extends to a contractive linear map $j_{\RR} : F_{\RR}(\alpha,u) \to C_0(\G_{\theta})$. 
\end{enumerate}
If the above equivalent conditions hold, then $\bigcap_{t\in G}\ker(E_t^{\RR})=\ker j_{\RR}$ is an ideal in $F_{\RR}(\alpha,u)$ (and $F_{\RR}(\alpha,u)$ is necessarily a completion of $F(\alpha,u)$). 
\end{lem}
\begin{proof}
Using the groupoid model from Example~\ref{ex:groupoid_model_for_partial_action_crossed_product}, the assertion follows from Proposition~\ref{prop:expectations_and_j_map}, see also Remark~\ref{rem:Fourier_decomposition}. 
\end{proof}

\begin{defn} 
If the equivalent conditions in Lemma~\ref{lem:crossed_products_by_partial_actions} hold, then we call $F_{\RR}(\alpha,u)$ a \emph{Banach algebra crossed product} for $(\alpha,u)$. 
If in addition $\ker j_{\RR} = \bigcap_{t\in G} \ker E_{t}^{\RR}=\{0\}$, we say that $F_{\RR}(\alpha,u)$ is a \emph{reduced Banach algebra crossed product}.
\end{defn}

For each $p\in[1,+\infty]$, the representation \eqref{eq:regular_for_p} composed with the isomorphism \eqref{eq:groupoid_model_for_partial_crossed_prod} gives a representation $\Lambda_p : F(\alpha,u) \to B(\ell^p(\G_{\theta}))$, where
\[
    \Lambda_p(b)\xi(t,x) := \sum_{s\in G\text{ if } x\in X_s} b(s)(x) u(s,s^{-1}t)(x) \xi \big( s^{-1}t,\theta_{s^{-1}}(x) \big) , \qquad \xi \in \ell^p(\G_{\theta}), \ b \in F(\alpha,u).
\]

\begin{defn} 
For any non-empty $P\subseteq [1,+\infty]$, we call $F^{P}_{\red}(\alpha,u) := F_{ \{ \Lambda_p : p \in P \} }(\alpha,u)$ the \emph{reduced $L^P$-crossed product}. 
We define the \emph{full $L^P$-crossed product} $F^{P}(\alpha,u)$ as $F_{\RR}(\alpha,u)$ where $\RR$ is the class of all representations $\psi$ of $F(\alpha,u)$ on $L^p$-spaces $L^p(\mu)$, for $p\in P$, and if $\F=\R$ we additionally assume that $\psi(C_0(X))$ consists of multiplication operators. 
We let $\Lambda_P : F^{P}(\alpha,u)\to F^{P}_{\red}(\alpha,u)$ be the canonical representation.
\end{defn}

By \eqref{eq:groupoid_model_for_partial_crossed_prod} and Proposition~\ref{prop:regular_disintegrated}, the algebras $F^{P}_{\red}(\alpha,u)$ are reduced crossed products isometrically isomorphic to $F^P_{\red}(\G_{\theta},\LL_{u})$. 
The above definitions agree with those introduced in \cite{BK} for full actions and $\F=\C$.


\begin{defn}[\cite{ELQ}]
A partial action $\theta$ is \emph{topologically free} if $\{x\in X_{t^{-1}}: \theta_{t}(x)=x\}$ has empty interior for every $t\in G\setminus \{1\}$. 
A set $U\subseteq X$ is \emph{$\theta$-invariant} if $h_{t}(U\cap X_{t^{-1}})\subseteq U$ for all $t\in G$. 
A partial action $\theta$ is \emph{minimal} if there are no non-trivial $\theta$-invariant open sets.
\end{defn} 


\begin{rem}\label{re:MinimalTopFreePartialActionGroupoid}
It is obvious that a partial action $\theta$ is minimal if and only if the associated transformation groupoid $\G_\theta$ is minimal. 
It is also easy to see from Remark~\ref{re:TopologicallyFreeGroupoid} that the partial action $\theta$ is topologically free if and only if $\G_\theta$ is topologically free. 
\end{rem}

If the twist $u$ is trivial, then we omit writing it. 
For a subgroup $H\subseteq G$ we denote the obvious restriction of the twisted action $(\alpha,u)$ of $G$ to a twisted action of $H$ by $(\alpha|_{H},u|_{H})$. 

\begin{thm}\label{pr:SimplicityPartialActions}
Let $(\alpha , u)$ be a twisted partial action of a discrete group $G$ on the algebra $C_0(X)$, with dual partial action $\theta$ of $G$ on $X$. 
Let $P\subseteq [1,\infty]$ be a non-empty set and let $F_{\RR}(\alpha , u)$ be any  reduced crossed product Banach algebra. 
The following are equivalent: 
\begin{enumerate}
    \item\label{enu:twisted_generalised_intersection_property1} $\theta$ is topologically free; 
    \item $C_0(X)$ is a maximal abelian subalgebra in $F_{\RR}(\alpha,u)$;
	\item\label{enu:twisted_generalised_intersection_property3} $C_0(X)$ detects ideals in all $F_{\red}^P (\alpha|_{H},u|_{H})$ for any subgroup $H\subseteq G$; 
	\item\label{enu:twisted_generalised_intersection_property4} $C_0(X)$ detects ideals in all $F_{\red}^P (\alpha|_{H},u|_{H})$ for any cyclic subgroup $H\subseteq G$; 
	\item $C_0(X)\subseteq F^P(\alpha)$ has the generalized intersection property with hidden ideal $\ker\Lambda_P$;
    \item $C_0(X)$ detects ideals in one (and hence all) of the algebras $F^1(\alpha)$, $F^{\infty}(\alpha)$, $F^{\{1,\infty\}}(\alpha)$. 
\end{enumerate}
If the above equivalent conditions hold, then $C_0(X)$ detects ideals in $F_{\RR}(\alpha,u)$ and so $F_{\RR}(\alpha,u)$ is simple if and only if $\theta$ is minimal. 
\end{thm}
\begin{proof}
Extending isomorphism \eqref{eq:groupoid_model_for_partial_crossed_prod}, we may view a reduced crossed product Banach algebra $F_{\RR}(\alpha,u)$ of $(\alpha , u)$ as an $S_{\theta}$-graded reduced groupoid Banach algebra of $(\G_{\theta},\LL_{u})$. 
As noted in Remark~\ref{re:MinimalTopFreePartialActionGroupoid}, topological freeness and minimality of $\theta$ are equivalent to topological freeness and minimality of $\G_\theta$. The groupoid  $\G_{\theta}$ is Hausdorff. Hence the assertion follows from Proposition~\ref{prop:Cartan_subalgebra} and Theorems~\ref{thm:Top_Free_Intersection_Property}, \ref{thm:topological_freeness_untwisted}, \ref{enu:topological_freeness_twisted_intersection}, cf.\ also Theorem~\ref{thm:topological_freeness_intro} and Corollary~\ref{cor:simplicity_intro}. 
Formally, to get from \ref{enu:twisted_generalised_intersection_property4} to \ref{enu:twisted_generalised_intersection_property1}, one needs to modify the proof of Theorem~\ref{enu:topological_freeness_twisted_intersection} by replacing $U$ by its intersection with some $\{ t \} \times X_{t^{-1}}$, $t\in G$. 
\end{proof}

Moving towards pure infiniteness, we generalize the definitions of $n$-filling and local boundary actions, introduced in \cite{Jolissaint-Robertson} and \cite{Laca-Spielberg} respectively, to the partial action case. 

\begin{defn}
For $n\geq 2$, we say that a partial action $\theta$ is \emph{$n$-filling} if $X$ is compact, infinite and for any non-empty open set $U \subseteq X$ there are $t_1, \ldots ,t_n \in G$ with $\bigcup_{i=1}^{n} \theta_{t_i}(X_{t_i^{-1}}\cap U_i) = X$. 
\end{defn}

\begin{rem} 
For full actions, $2$-filling is equivalent to being a \emph{strong boundary action} introduced by Laca and Spielberg~\cite{Laca-Spielberg}.
In general, the $n$-filling condition for the action $\theta$ is stronger than the $n$-filling condition for the groupoid $\G_{\theta}$, see \cite{Suzuki}. 
\end{rem}

\begin{defn}
We call $\theta$ a \emph{local boundary action} if for every non-empty open $U\subseteq X$ there is an open $V \subseteq U$ and $t\in G$ such that $\overline{V} \subseteq X_t$ and  $\theta_{t}(\overline{V}) \subsetneq V$.
\end{defn}

\begin{rem}
If  $\theta$ is a local boundary action, then the groupoid $\G_{\theta}$ is locally contracting (in fact the latter notion was inspired by the former, see \cite{A-D}).
\end{rem}

We arrive at the following generalization of the classical results of \cite{Laca-Spielberg}, \cite[Example 2.1]{Jolissaint-Robertson}, to the twisted, partial action and Banach algebra case.

\begin{thm}\label{th:PureInfinitePartialAction}
Let $(\alpha,u)$ be a twisted partial action of a discrete group $G$ on $C_0(X)$. 
If the dual action $\theta$ is topologically free, minimal and either $n$-filling or a local boundary action, then every reduced twisted crossed product Banach algebra $F_{\RR}(\alpha,u)$ is purely infinite simple. 
\end{thm}
\begin{proof} 
Extending isomorphism \eqref{eq:groupoid_model_for_partial_crossed_prod}, we may view $F_{\RR}(\alpha,u)$ as a reduced groupoid Banach algebra of $(\G_{\theta},\LL_{u})$. 
Since $\G_{\theta}$ is Hausdorff, this is also an essential groupoid Banach algebra, and $\G_\theta$ is topologically free and minimal. 
Clearly,   if $\theta$ is $n$-filling  then $\G_{\theta}$ is $n$-filling  with respect to $S_\theta$, and if $\theta$ is a local boundary action then $\G_{\theta}$ is locally contracting with respect to $S_\theta$. 
The assertion therefore follows from Theorem~\ref{thm:filling_implies_pure_infinite} and Theorem~\ref{thm:Anantharaman-Delaroche}. 
\end{proof}

\begin{rem} 
Theorem~\ref{th:PureInfinitePartialAction} applies to a number of natural group actions on boundaries of various type, cf.\ \cite{A-D} and \cite{Jolissaint-Robertson}, to actions modeling Ruelle algebras arising from certain Smale spaces, see \cite[Examples 14.1, 14.2]{Laca-Spielberg}, and to many actions of any countable non-amenable exact group on a product of any connected closed topological manifold with a Cantor set, see  \cite[Theorem 4.1]{Suzuki}. 
If $X$ is totally disconnected and every compact open set is $\theta$-paradoxical, in the sense of \cite[Definition 4.3]{Giordano_Sierakowski}, then $\theta$ is a local boundary action. 
Thus Theorem~\ref{th:PureInfinitePartialAction} applies to partial actions described in \cite[Section 5]{Giordano_Sierakowski} such  as  those modeling Cuntz algebras, simple Cuntz--Krieger algebras and algebras of integral domains. 
\end{rem}

\subsection{Banach Roe-type operator algebras}
\label{ssec:BanachRoeAlgebras}

In this subsection, we associate to any uniformly locally finite coarse structure $(X,\E)$ a family of Banach algebras that cover Roe-type Banach algebras with compact coefficients considered by \v{S}pakula and Willett~\cite{Spakula_Willett}, as well as $\ell^p$ Roe-type algebras of Chung and Li~\cite{Chung_Li}. 
We prove that they can be naturally viewed as a reduced Banach algebra of a principal \'etale groupoid, so parts of our theory apply. 

Let $\E$ be a (unital) \emph{coarse structure} on a set $X$ in the sense of \cite{STY}, that is, $\E$ is a collection of subsets of the pair groupoid $X\times X$ which contains all finite subsets, and is closed under inverses, products, unions and formation of subsets. 
Elements of $\E$ are called \emph{entourages}. 
The coarse space $(X,\E)$ is called \emph{unital} if $\E$ contains the diagonal $\{(x,x):x\in X\}$.
We will assume that $(X,\E)$ is \emph{uniformly locally finite}, which means that for any $E\in \E$ the range and domain of $E$ is finite, that is, the following number
\[
    n(E) := \sup_{x\in X} |\{y\in X: (x,y)\in E \text{ or }(y,x)\in E\}| =\sup_{x\in X}|(r^{-1}(x) \cup d^{-1}(x))\cap E|
\]
is finite. 
As explained in \cite[page 810]{STY}, up to coarse equivalence, coarse structures coming from metric spaces are exactly countably generated unital coarse structures. 
In this case uniform local finiteness corresponds to what is called \emph{bounded geometry} for metric spaces. 
In this sense, our setting covers those in \cite{Spakula_Willett} or \cite{Chung_Li}.

Let $\EE$ be a Banach space with a fixed normalized (not necessarily countable) Schauder basis $\{e_i\}_{i\in I}$. 
For every finite $I_{0}\subseteq I$, we denote by $M_{I_0}(\EE)\cong M_{|I_0|}(\F)$ the set of operators on the finite dimensional space $\spane\{e_i:i\in I_{0}\}$. 
We view $M_{I_0}(\EE)$ as a subaglebra of $\BB(\EE)$ and we put $M_{I}(\EE):=\bigcup_{\stackrel{I_0\subseteq I}{\text{finite}} } M_{I_0}\subseteq \BB(\EE)$. 
Up to an equivalence of the norm in $\EE$, we may and shall assume that the unit $1_{I_{0}}$ in $M_{I_0}(\EE)$ has norm one as an element of $\BB(\EE)$, cf.\ \cite{Daws0}. 
We will refer to such a Schauder basis as \emph{standardized}.

\begin{defn}\label{defn:controlled_propagation_matrix}
A \emph{controlled propagation matrix} on $(X,\E)$ with coefficients in $\BB(\EE)$ is an $X$-by-$X$ matrix $T=\{T_{x,y}\}_{x,y\in X}\subseteq \BB(\EE)$ such that: 
\begin{enumerate}
    \item\label{enu:controlled_propagation_matrix1} $T_{x,y}\subseteq  M_{I_0}(\EE)$ for some finite $I_0\subseteq I$ and all $x,y \in X$;
    \item\label{enu:controlled_propagation_matrix2} the norms $\| T_{x,y} \|$, $x,y\in X$ are uniformly bounded; 
    \item\label{enu:controlled_propagation_matrix3} the support $\supp(T) := \{(x,y)\in X\times X: T_{x,y}\neq 0\}$ is an entourage ($\supp(T)\in \E$).
\end{enumerate}
The collection of all such controlled propagation matrices is denoted by $\F[X;\EE]$.
\end{defn}

\begin{rem}
The above definition differs from \cite[Definition 2.2]{Spakula_Willett} in that we impose the restriction \ref{enu:controlled_propagation_matrix1}. 
Obviously, \ref{enu:controlled_propagation_matrix1} is automatic when $\EE$ is finite dimensional. 
\end{rem}

\begin{ex}[cf.\ {\cite[Example 2.3]{Spakula_Willett}}] 
Fix a finite $I_0\subseteq I$ and an entourage $E\in \E$. 
Bounded maps from $X$ to $M_{I_0}(\EE)$ supported on $r(E):=\{x:(x,y)\in E \}$, that is elements of $\ell^{\infty}(r(E), M_{I_0}(\EE))$, when identified with diagonal matrices, are elements of $\F[X;\EE]$, referred to as \emph{multiplication operators}. 
If in addition the entourage $E$ is a bisection in $X\times X$, the characteristic function of $E$ multiplied by the unit $1_{I_{0}}$ in $M_{I_0}(\EE)$ is an element of $\F[X;\EE]$, referred to as a \emph{partial translation}. 
Namely, for a bisection entourage $E$ and any $f\in \ell^{\infty}(r(E), M_{I_0}(\EE))$, the formulas
\[
    f_{x,y}=\begin{cases}f(x), & x=y\in r(E),\
        \\
        0, & \text{otherwise},
        \end{cases}
    \qquad 
    E^{I_0}_{x,y}=\begin{cases}1_{I_{0}}, & (x,y)\in E,\
        \\
        0, & \text{otherwise},
        \end{cases}
\]
define basic elements of $\F[X;\EE]$.
\end{ex}

The uniform local finiteness of $(X,\E)$ guarantees that $\F[X;\EE]$ becomes an algebra under the usual matrix operations and the algebra structure of $\BB(\EE)$. 
Moreover, for every $p\in [1,\infty]$, $\F[X;\EE]$ can be naturally represented on  the space $\ell^p(X,\EE)$ of $\EE$ valued functions, with the usual norm $\|\xi\|_p=(\sum_{x\in X}\|\xi(x)\|^p)^{1/p}$. 
More specifically, we have:

\begin{lem}\label{lem:bisection_entourages_decomposition}
Let $T\in \F[X;\EE]$ and put $N := n(\supp(T))^2+1 < \infty$. 
Then there are pairwise disjoint bisection entourages $E_1 , \ldots , E_N \subseteq \supp(T)$ and functions $f_1, \ldots , f_N \in \ell^\infty(X, M_{I_0}(\EE))$, for some finite $I_0\subseteq I$, with $\supp_{x\in X, k=1, \ldots ,N} \| f_k(x) \| \leq \sup_{x,y} \| T_{x,y} \|$ and such that
\[
    T=\sum_{k=1}^{N} f_k  E_k^{I_0},
\]
where we use the usual matrix operations. 
Thus $T$ is a finite combination of partial translations with coefficients in multiplication operators.
\end{lem}
\begin{proof}
Let $I_0\subseteq I$ be a finite set such that $T_{x,y}\subseteq  M_{I_0}(\EE)$ all  $x,y \in X$. 
By \cite[Lemma 4.10]{Roe}, see also \cite[Lemma 2.7]{STY}, there are pairwise disjoint bisection entourages $E_1, \ldots ,E_N$ such that $\supp(T)=\bigsqcup_{i=1}^N E_i$. 
For each $i=1, \ldots ,n$ define $f_i\in \ell^{\infty}(r(E_i), M_{I_0}(\EE))$ by $f_{i}(x) := T_{x,y}$ if $(x,y)\in E$ for some $y\in X$ and $f_{i}(x)=0$ otherwise. 
Then the displayed formula holds. 
\end{proof}

\begin{prop}\label{prop:regular_reps_Roe_algebra}
The usual matrix operations and the algebra structure of $\BB(\EE)$ turn $\F[X;\EE]$ into an algebra.
Moreover, for every $p\in [1,+\infty]$, a matrix $T \in \F[X;\EE]$ defines a bounded operator $T_p\in\BB(\ell^p(X,\EE))$ via 
\[
    (T_p\xi)(x) := \sum_{y\in X} T_{x,y} \xi(y), \qquad \xi \in \ell^p(X,\EE).
\]
Then $\|T_p\|\leq  \sup_{x,y} \|T_{x,y}\| \cdot n(\supp(T))^2+1$.
The assignment $\F[X;\EE]\ni T\mapsto T_p\in \BB(\ell^p(X,\EE))$ is an injective algebra homomorphism.
\end{prop}
\begin{proof} 
Let $E\in \E$ be a bisection entourage and let $f\in \ell^{\infty}(r(E), M_{I_0}(\EE))$, for a finite $I_0\subseteq I$. 
The norm of $(f E^{I_0})_{p}$, as an operator on $\ell^p(X,\EE)$, is not greater than $\sup_{x\in X}\|f(x)\|$ due to our standing normalization assumption that $1_{I_0}\in \B(\EE)$ has norm one, which implies that $(E^{I_0})_{p}$ has norm one. 
Thus, applying Lemma~\ref{lem:bisection_entourages_decomposition} to $T\in \F[X;\EE]$, we get that $T_p\in\BB(\ell^p(X,\EE))$ is bounded with the prescribed estimate on the norm.  

It is clear that the map $\F[X;\EE]\ni T\mapsto T_p\in \BB(\ell^p(X,\EE))$ is linear and injective. 
To see that it is multiplicative, in view of Lemma~\ref{lem:bisection_entourages_decomposition}, it suffices to consider only two elements of the form $f_1 E_1^{I_1}, f_2 E_2^{I_2}\in \F[X;\EE]$ where $f_1\in \ell^{\infty}(X, M_{I_1}(\EE))$, $f_2\in \ell^{\infty}(X, M_{I_2}(\EE))$, $I_1, I_2\subseteq I$ are finite, and  $E_1,E_2\in \E$ are bisections. 
Then we have 
\begin{equation}\label{eq:product_of_basic_elements}
    (f_1  E_1^{I_1}) (f_2 E_2^{I_2})= f_1 (f_2\circ h_{E_1}) (E_1 E_2)^{I_1\cap I_{2}} , 
\end{equation} 
where $f_2\circ h_{E_1}\in \ell^{\infty}(d(E_1), M_{I_2}(\EE))$ is given by $f_2\circ h_{E_1} (y):=f_2(x)$ for $(x,y)\in E_{1}$. 
This readily implies that $(f_1  E_1^{I_1})_p (f_2 E_2^{I_2})_p = (f_1  E_1^{I_1} f_2 E_2^{I_2})_p$.
\end{proof}

\begin{defn}
For any non-empty $P\subseteq [1,\infty]$ we define the \emph{$L^P$-$\EE$ Roe-type Banach algebra} of $(X,\E)$ as the completion $F^{P}[X;\EE] = \overline{\F[X;\EE]}^{\|\cdot\|_{L^P,\EE}}$ in the norm $\|T\|_{L^P,\EE} := \sup_{p\in P} \| T_p \|$.
When $P = \{p\}$ is a singleton, we omit writing braces in the above notation.
\end{defn}

The \emph{translation groupoid} associated to the coarse space $(X,\E)$ is defined as the union of closures of entourages in the Stone--\v{C}ech compactifaction of the discrete space $X\times X$, see \cite[3.2]{STY}, \cite[Definition 10.19]{Roe}. 
That is, as a topological space, we put
\[
    \G(X) := \bigcup_{E\in \E} \overline{E}\subseteq \beta(X\times X) .
\]
The range and domain maps $r,d : X\times X \to X$ extend uniquely to continuous maps $r,d: \beta(X\times X)\supseteq \G(X)\to \beta(X)$ and with this structure $\G(X)$ becomes an \'etale Hausdorff ample principal groupoid. 
Also the bisection entourages $S^X := \{ E\in \E : \text{$E$ is a bisection in $X\times X$} \}$ form an inverse semigroup whose canonical action by partial bijections on $X$ induces an action by partial homeomorphisms of $\beta(X)$, and we have a natural groupoid isomorphism $\G(X)\cong \beta(X)\rtimes S^X$, see \cite[Proposition 3.2]{STY}.

An \emph{ideal in the coarse space} $(X,\E)$ is a coarse space $(X,\E_0)$ where $\E_0\subseteq \E$ is such that $E_0 E, E E_0\in \E_0$ for all $E \in \E$ and $E_0\in\E_0$, cf.\ \cite[Definition 4.1]{Chena_Wang}. 
We will say that $(X,\E)$ is \emph{simple} if there are no non-trivial ideals in $(X,\E)$.

\begin{thm}
Let $(X,\E)$ be a uniformly locally finite coarse structure. 
Let $\EE$ be a Banach space with a standardized Schauder basis $\{e_i\}_{i\in I}$, and 
let $\G := \G(X)\times \RR_{I}$ where $\RR_I:=I\times I$ is the full equivalence relation on $I$. 
Then there is a canonical algebraic isomorphism 
\[
    C_c(\G) \cong \F[X;\EE] ,
\]
under which for any non-empty $P\subseteq [1,\infty)$ (or $P\subseteq [1,\infty]$ when $\EE$ is reflexive), the Roe-type algebra $F^{P}[X;\EE]$ becomes a reduced $S$-graded groupoid Banach algebra of $\G$ where $S$ consists of $\overline{E}\times \{(i,j)\}$, $E\in S^X$, $i,j\in I$, and the unit space $\beta(X)$. 

In particular, the diagonal $C_0(\beta(X)\times I)$ is a maximal abelian subalgebra,  $C_0(\beta(X)\times I)\subseteq F^{P}[X;\EE]$ has the intersection property, and $F^{P}[X;\EE]$ is simple if and only if the coarse space $(X,\E)$ is simple. 
\end{thm}
\begin{proof}
Clearly, $S$ described in the assertion is a wide unital inverse subsemigroup of $\Bis(\G)$.
In particular, every element in $C_c(\G)$ is a finite linear combination of continuous functions with supports in elements in $S$. 
Moreover, for every $E\in S^X$ and finite $I_0\subseteq I$ we have a linear isomorphism $C_c(\G)\cong \F[X;\EE]$, 
\[
    C(\bigcup_{i,j\in I_0}\overline{E}\times \{(i,j)\})\ni \overline{f}\longmapsto f \in \ell^{\infty}(X, M_{I_0}(\EE)) 
\]
where $f(x)_{i,j} := \overline{f}(x,y,i,j)$ if $(x,y)\in E$ for some $y\in X$, and $f(x)= 0$ otherwise. 
Using this and Lemma~\ref{lem:bisection_entourages_decomposition}, one sees that the maps
\[
    C\big(\bigcup_{i,j\in I_0}\overline{E}\times \{(i,j)\}\big)\ni \overline{f}\longmapsto f E^{I_0}\in \F[X;\EE], \qquad E\in S^X,
\]
integrate to a linear isomorphism $C_c(\G)\cong \F[X;\EE]$. 
In view of \eqref{eq:product_of_basic_elements}, it readily follows that this isomorphism is multiplicative. 

Let $p\in [1,\infty]$. 
Composing $C_c(\G)\cong \F[X;\EE]$ with the monomorphism from Proposition~\ref{prop:regular_reps_Roe_algebra}, we get a monomorphism $\Lambda_p : C_c(\G)\to \BB(\ell^p(X,\EE))$ such that 
\[
    [\Lambda_p (f)\xi](x) = \sum_{y\in X} f(x,y,i,j) e_i\otimes e_j^* (\xi (y)), \qquad f\in C_c(\G), \xi\in \ell^p(X,\EE),
\]
where  $e_i\otimes e_j^*\in \B(\EE)$ is the contractive one-dimensional projection given by $(e_i\otimes e_j^*)(\sum_{k\in I}\xi_ke_k)= \xi_j e_i$. 
For $f$ supported on $\overline{E}\times \{(i,j)\}\in S$, we clearly have $\| \Lambda_p(f) \| = \|f\|_{\infty}$, and so $\Lambda_p$ extends to a representation $\Lambda_p:F^S(\G)\to F^{p}[X;\EE]$. 
For $(y,j)\in X\times I$, we let $1_{(y,j)} \in \ell^p(X,\EE)$ be such that $1_{(y,j)}(y):=e_j$ and $1_{(y,j)}(x) :=0$ for $x \in X\setminus \{ y \}$. 
Then putting 
\begin{equation}\label{eq:Roe_j_map}
    j(T)(x,y,i,j) := e_i^*\big(T 1_{(y,j)} \big)  , \qquad T\in F^{P}[X;\EE] ,\ (x,y,i,j) \in \G_0:=\bigcup_{E\in \E} E \times \RR_{I} ,
\end{equation} 
we get a contractive linear map $j$ from $F^{P}[X;\EE]$ to $C_{b}(\G_0)$, equipped   with the norm $\|\cdot\|_{\infty}$. 
For $f\in C_c(\G)$, we have $j(\Lambda_p(f))=f|_{\G_0}$ which has a unique continuous extension to $\G$ (which is $f$) and therefore it may be treated as an element of $C_c(\G)$. 
Thus we get that on a dense subalgebra $j$ takes values in $C_c(\G)$ and therefore we may assume that $j$ maps $F^{p}[X;\EE]$ to $C_{0}(\G)$. 
In particular, $F^{p}[X;\EE]$ is a groupoid Banach algebra. 
If $p<\infty$, then $\{1_{(y,j)}\}_{(y,j)\in X\times I}$ is a Schauder basis for $\ell^p(X,\EE)$, and then \eqref{eq:Roe_j_map} implies that the map $j$ is injective on $F^{p}[X;\EE]$ and hence $F^{P}[X;\EE]$ is reduced. 
For $p=\infty$, this is not so clear. 
However, when $\EE$ is reflexive, then $\{ e_{i}^* \}_{i\in I}$ is a Schauder basis for the dual space $\EE^*$, and the involution on $C_c(\G)$ induces an isometric anti-linear isomorphism $F^{\infty}[X;\EE]\cong F^1[X;\EE^*]$. 
Thus $F^{\infty}[X;\EE]$ is reduced because we have just proved that  $F^1[X;\EE^*]$ is. 
The assertion for the set $P\subseteq [1,\infty]$ follows now from Lemma~\ref{lem:reduced_from_other_reduced}.

Since $\G$ is Hausdorff and topologically free (even principal), the last part of the assertion follows from Theorem~\ref{thm:topological_freeness_intro} and Corollary~\ref{cor:simplicity_intro}, because $\G$ is minimal if and only if $\G(X)$ is minimal, which holds if and only if $(X,\EE)$ is simple \cite[Theorem 6.3]{Chena_Wang}. 
\end{proof}

\begin{rem} 
Using that $C_0(\beta(X)\times I)\subseteq F^{P}[X;\EE]$ has the intersection property, one can show that ideals in  $F^{P}[X;\EE]$ in which controlled propagation operators are dense, are exactly those generated by open $\G$-invariant subsets of $\beta(X)\times I$ (which in turn correspond to ideals in the coarse structure $(X,\E)$). 
In this way, one could generalize \cite[Theorems 6.3, 6.4]{Chena_Wang} from uniform Roe $C^*$-algebras to algebras of the form $F^{P}[X;\EE]$. 
\end{rem}

\subsection{Tight algebras of inverse semigroups}
\label{ssec:TightBanachAlgebrasInverseSgrp}

In this subsection, we define reduced and essential tight $L^P$-operator algebras associated to inverse semigroups, and phrase our main results in terms of the inverse semigroup. 
This extends, and sometimes improves, the corresponding results for $C^*$-algebras. 

We fix an inverse semigroup $S$ with zero $0$. We recall our notation from \cite[Section 6]{BKM} for the tight groupoid associated to  $S$ 
in \cite{Exel}. 
The set of idempotents $E := E(S)$ is a semilattice with the `meet' and `preorder' defined by $e \wedge f := e f$ and $e\leq f$ if and only if $ef = e$. 
The set of filters $\widehat{E}$ in $E$, with the product topology inherited from the Cantor space $\{0,1\}^{E}$, is a totally disconnected Hausdorff space. 
A \emph{cover of $e \in E$} is a finite set $F \subseteq  E$ such that, for every non-zero $z \leq e$, we have $z f \neq 0$ for some $f \in F$. 
A filter $\phi\in \hat{E}$ is \emph{tight} if for every $e\in \phi$ and any cover $F \subseteq  eE$  of $e$, we have $\phi \cap F\neq \emptyset$. 
The set of tight filters $\widehat{E}_{\tight}$ is called the \emph{tight spectrum} of $E$. 
It is the closure in $\hat{E}$ of the set of all ultrafilters (if $E$ is a Boolean algebra, then the tight filters are precisely the ultrafilters). 
For each $e\in E$, the set $Z_e := \{ \phi\in\widehat{E}_{\tight} : e \in \phi \}$ is compact open in $\widehat{E}_{\tight}$. 
For each $t\in S$, we have a homeomorphism $h_t : Z_{t^* t} \to Z_{t t^*}$ given by 
\[
    h_t(\phi) := \{ e \in E : t^* e t\in \phi \} = \{ e \in E : \text{$f \geq t ft^*$  for some $f\in \phi$} \}. 
\]
This yields an inverse semigroup action $h : S \to \PHomeo(\widehat{E}_{{\tight}})$ and we can form the transformation groupoid 
\[
    \G_{{\tight}}(S) := S \ltimes_h \widehat{E}_{{\tight}}.
\]  
Let $u=\{u(s, t)\}_{s,t\in S}$ be a twist of the action $h : S \to \PHomeo(\widehat{E}_{{\tight}})$ described above. 
That is, $u(s, t)\in C_{u}(Z_{t^*s^*st})$, $t,s \in S$, are maps satisfying the conditions in \cite[Definition 4.1]{Buss_Exel}, or in \cite[Definition 3.4]{BKM}. 
Also let $\LL_u$ be the twist associated to $u$, as in \cite[Subsection 4.3]{BKM}. 
Then $(\G_{{\tight}}(S),\LL_{u})$ is an ample twisted groupoid, and every ample twisted groupoid is of this form, see Remark~\ref{rem:every_ample_from_inverse_semigroup} below. 

We will now represent the twisted inverse semigroup $(S,u)$ via spatial partial isometries on the spaces $\ell^p(\G_{{\tight}}(S))$, $p\in [1,\infty]$.
We recall that for any localizable measure $\mu$, in particular for a counting measure, we have a natural inverse semigroup $\SPIso(L^p(\mu))\subseteq B(L^p(\mu))$ of contractive weighted composition operators called \emph{spatial partial isometries} on $L^p(\mu)$, see \cite[Definition 2.26]{BKM}, \cite[Definition 6.4]{PhLp1}, \cite[Definition 6.2]{Gardella}. 

\begin{prop}\label{prop:tight_reduced_twisted_inverse_semigroup_algebra}
Let $S$ be an inverse semigroup with zero and let $u=\{u(s, t)\}_{s,t\in S}$ be a twist of the associated action $h : S \to \PHomeo(\widehat{E}_{{\tight}})$. 
For each $p\in [1,\infty]$, we have a map $v : S \to \SPIso\big(\ell^p(\G_{{\tight}}(S))\big)$ into spatial isometries on $\ell^p(\G_{{\tight}}(S))$ given by 
\[
      v_t\xi[s,\phi] = \begin{cases}
                            u(t,t^*s) \big( h_{s}(\phi) \big) \xi \big( [t^*s,\phi] \big), &  \phi\in h_{s}^{-1}(Z_{tt^*}), \\
                            0 , & \text{otherwise}, 
                    \end{cases}
\]
where  $\xi \in \ell^p(\G)$, $\phi\in Z_{s^*s}$, $t,s \in S$. 
Denoting by $v_{t}^*$ the unique generalized inverse of $v_t$ in $\SPIso \big( \ell^p(\G_{{\tight}}(S)) \big)$, we get 
\[
    F^p_{\red,{\tight}}(S,u) := \clsp\{ v_{t} v_{s}^{*}: t,s\in S \}
\]
is a Banach algebra naturally isometrically isomorphic to the twisted reduced groupoid algebra $F_{\red}^p(\G_{{\tight}}(S), \LL_u)$ via an isomorphism that restricts to $\clsp\{ v_{t} v_{t}^{*}: t\in S\}\cong C_0(\widehat{E}_{\tight} )$. 
\end{prop}
\begin{proof} 
By \cite[Remark 5.2]{BKM}, we have an isometric representation $\pi : C_0 ( \widehat{E}_{\tight} ) \to B\big(\ell^p(\G_{{\tight}}(S))\big)$ given by the formula $\pi(a) \xi[s,x] = a \big( h_s(x) \big) \xi([s,x])$ and then the pair $(\pi, v)$ forms a covariant representation of the twisted action that integrates to the regular representation of $F(\G_{{\tight}}(S), \LL_u)$ on $\ell^p(\G_{{\tight}}(S))$. 
In particular, $v_t$ is a spatial partial isometry, with the generalized inverse being $v_t^* = v_{t^*} \pi(\overline{u(t,t^*)})$, and $v_{t} v_{t}^* = \pi(1_{Z_{tt^*}})$ for each $t\in S$. 
This implies that we have an isometric isomorphism $C_0( \widehat{E}_{\tight}) \cong \pi(C_0( \widehat{E}_{\tight})) = \clsp\{ v_{t} v_{t}^{*}: t\in S\}$ that extends to an isometric isomorphism $F_{\red}^p(\G_{{\tight}}(S), \LL_u)\cong F^p_{\red,{\tight}}(S,u)$ induced by the integrated form $\pi\rtimes v$ of the covariant representation $(\pi, v)$.
\end{proof}

\begin{rem}
When the twist $u$ is trivial, we write $F^p_{\red,{\tight}}(S)$ for $F^p_{\red,{\tight}}(S,u)$. 
Then the map $v$ in Proposition~\ref{prop:tight_reduced_twisted_inverse_semigroup_algebra} is a tight representation of $S$ in the sense of \cite[Definition 6.12]{BKM}. 
In particular, $F^p_{\red,{\tight}}(S)$ is a reduced version of the tight $L^p$-operator algebra $F^p_{\tight}(S)$ of $S$ introduced in \cite[Definition 6.14]{BKM}: $v$ induces a natural representation $F^p_{\tight}(S)\to  F^p_{\red,{\tight}}(S)$ (which, by \cite[Theorem 6.19]{Gardella_Lupini17}, is an isometric isomorphism if $S$ is countable and the groupoid $\G_{{\tight}}(S)$ is amenable).
\end{rem}

\begin{rem}\label{rem:every_ample_from_inverse_semigroup}
By \cite[Corollary 6.18 and Lemma 4.23]{BKM}, every twisted ample groupoid $(\G,\LL)$ is isomorphic to $(\G_{{\tight}}(S), \LL_u)$ for some inverse semigroup $S$ (one may take $S$ to be the collection of all compact open bisection in $S(\LL)$) and some twist $u$. 
Thus the algebras $F^p_{\red,{\tight}}(S,u)$ introduced in Proposition~\ref{prop:tight_reduced_twisted_inverse_semigroup_algebra} model all reduced twisted groupoid $L^p$-operator algebras associated to ample groupoids.
\end{rem}

\begin{cor}\label{cor:essential_inverse_semigroup_algebra}
Retain the notation and assumptions from Proposition~\ref{prop:tight_reduced_twisted_inverse_semigroup_algebra} and fix $p\in [1,\infty]$.
Denote by $\G_{\Hau}(S)$ the set of Hausdorff points in $\G_{{\tight}}(S)$. 
Then the obvious contractive projection $P_{\Hau}$ from $\ell^p(\G_{{\tight}}(S))$ onto $\ell^p(\G_{\Hau}(S))\subseteq \ell^p(\G_{{\tight}}(S))$ commutes with elements of $F^p_{\red,{\tight}}(S,u)$, and if $S$ is countable, then the compressed Banach algebra
\[
    F^p_{\ess,{\tight}}(S,u) := F^p_{\red,{\tight}}(S,u) P_{\Hau} \subseteq B(\ell^p(\G_{\Hau}(S)))
\]
is isometrically isomorphic to the essential groupoid $L^p$-operator algebra $\FFalg{\G_{{\tight}}(S), \LL_u}{\ess}{p}{}$ 
introduced in Proposition~\ref{prop:essential_L^p_operator_algebra}.
\end{cor}

\begin{defn}\label{def:L_p_inverse_semigroup_algebras}
Let $P\subseteq [1,\infty]$ be a non-empty set. 
We denote the Banach algebra generated by the $\ell^{\infty}$-direct sum of representations $S \to \SPIso\big(\ell^p(\G_{{\tight}}(S))\big)$, $p\in P$, from Proposition~\ref{prop:tight_reduced_twisted_inverse_semigroup_algebra}, by $F^P_{\red,{\tight}}(S,u)$. 
We denote the compression of $F^P_{\red,{\tight}}(S,u)$ to the subspace $\bigoplus_{p\in P}\ell^p(\G_{\Hau}(S))$ by $F^P_{\ess,{\tight}}(S,u)$.
We also write $F^P_{{\tight}}(S,u) := F^P(\G_{{\tight}}(S), \LL_u)$ and skip writing $u$ when $u$ is trivial. 
\end{defn}
 
\begin{defn}\label{def:InverseSemigroupSimplicityProperties}
Let $S$ be an inverse semigroup with zero. 
We say that $e\in E$ is \emph{fixed} by $t\in S$ if we have $ft^* f\neq 0$, for every nonzero idempotent $f\leq e$. 
Elements in $F_t = \{ e\in E : e\leq t \}$ are fixed by $t$ and we call them \emph{trivially fixed} by $t$. 
We also say that: 
\begin{enumerate}
    \item $S$ is \emph{closed} if for every $t\in S$ there is finite $F\subseteq F_t$ that covers every $e\in F_t$. 
    \item $S$ is \emph{topologically free} if for every $e\in E$ fixed by $t\in S$ there is  $f\in F_t$ with $f  e\neq 0$.
    \item $S$ is \emph{minimal} if for every $e, f \in E\setminus\{0\}$, 
		there is a finite $T\subseteq S$, such that $\{ t f t^* \}_{t\in T}$ is a cover for $e$. 
    \item $S$ is \emph{locally contracting} if for every $e \in E\setminus\{0\}$ there is $s\in S$ and a finite set $F\subseteq E\setminus\{0\}$ with a distinguished element $f_0\in F$ such that for every $f\in F$ we have $f\leq e s^*s$, $F$ is a cover of both $sfs^*$ and $f_0 s f$. 
\end{enumerate}
\end{defn}

\begin{rem} 
In Definition~\ref{def:InverseSemigroupSimplicityProperties} we slightly changed the terminology from \cite{EP16}. 
Namely, the authors of \cite{EP16} call weakly-fixed what we call fixed, and they call fixed what we call trivially fixed (our naming is consistent with \cite[Definition 4.1]{EP16}, though). 
Local contractiveness is as in \cite[Definition 6.1]{EP16}, closedness and minimality appear as conditions in \cite[Theorems 3.16 and 5.5]{EP16}, respectively. 
Topological freeness is our invention and it improves a condition appearing in \cite[Theorem 4.10]{EP16}, see Remark~\ref{rem:Exel_Pardo_topological_freeness} below. 
\end{rem}

\begin{lem}\label{lem:towards_topological_freeness_tight}
An idempotent $e\in E$ is fixed by $t\in S$ if and only if $Z_{e} \subseteq \{ \xi \in Z_{t^*t} : h_{t}(\xi) = \xi \}$.
\end{lem}
\begin{proof} 
If $e\in E$ is fixed by $t$, then for every $0\neq f\leq e$ we have $0\neq ft^*f =  ft^*tt^*f = t^*tft^*f$, which implies that $t^*tf \neq 0$. 
Hence $e$ is covered by $t^*t$ (formally by $\{ t^* t \}$), which is equivalent to the inclusion $Z_{e} \subseteq Z_{t^*t}$, see \cite[Proposition 3.8]{EP16}. 
Thus the equivalence in the assertion is proved in \cite[Lemma 4.9]{EP16} (as the assumption that $e\leq t^*t$, made there, is redundant). 
\end{proof}

\begin{prop}\label{prop:properties_of_tight_groupoid} 
Let $S$ be an inverse semigroup with zero. 
Then $S$ is closed (resp.\ topologically free or minimal) if and only if $\G_{\tight}(S)$ is Hausdorff (resp.\ topologically free or minimal). 
If $S$ is locally contractive, then $\G_{\tight}(S)$ is. 
\end{prop}
\begin{proof} 
The claims about $S$ being closed, minimal, and locally contractive follow by \cite[Theorems 3.16, 5.5 and 6.5]{EP16}. 
For the claim about topological freeness, recall that topological freeness of $\G_{\tight}(S)$ is equivalent to topological freeness of $h$ in the sense of 
\cite[Definition 2.20]{Kwa-Meyer}, by \cite[Lemma 2.23]{Kwa-Meyer}. 
This means that for every $t\in S$, the set $\{\phi\in Z_{t^*t}: h_{t}(\phi) = \phi \} \setminus \bigcup_{f\in F_t} Z_f$ has empty interior. 
By \cite[Proposition 2.5]{EP16}  and density of ultrafilters in $\widehat{E}_{\tight}$, every non-empty open set in $\widehat{E}_{\tight}$ contains a set of the form $Z_e$, for some $e\in E$.
Thus topological freeness of $\G_{\tight}(S)$ is equivalent to the condition that for any $t\in S$ and $e\in E$, if $Z_e \subseteq \{ \phi \in Z_{t^*t} : h_{t}(\phi) = \phi \}$, then $Z_e \cap \bigcup_{f\in F_t} Z_f \neq \emptyset$. 
In view of Lemma~\ref{lem:towards_topological_freeness_tight} and the fact that $Z_e \cap Z_f \neq \emptyset$ if and only if $e f \neq 0$, this in turn is equivalent to saying that for any $t\in S$ and $e\in E$, if $e$ is fixed by $t$, then $ef\neq 0$ for some $f\in F_t$, which is exactly what we called topological freeness of $S$.
\end{proof}

\begin{rem}\label{rem:Exel_Pardo_topological_freeness}
The condition (ii) in \cite[Theorem 4.10]{EP16} tries to characterize effectiveness of  $\G_{\tight}(S)$ in terms of $S$. 
It says that for every $e\in E$ fixed by $t\in S$, $e$ can be covered by elements in  $F_t$. 
When $\G_{\tight}(S)$ is Hausdorff (equivalently $S$ is closed), this condition is equivalent to $S$ being topologically free, but in general topological freeness is  is weaker. 
This agrees with the fact that effectiveness is equivalent to topological freeness for Hausdorff groupoids, but in general  is stronger. 
\end{rem}

\begin{thm}\label{thm:SimplicityTightAlgebras}
Let $S$ be an inverse semigroup with zero, let $u$ be a twist for the associated action on $\widehat{E}_{\tight}$ and let $P\subseteq [1,\infty]$ be non-empty. 
Assume $S$ is topologically free and minimal. 
If $S$ is closed, then $F^P_{\red,{\tight}}(S,u)$ is simple and it is purely infinite if $S$ is locally contractive. 
If $S$ is countable, then $F^P_{\ess,{\tight}}(S,u)$ is simple and it is purely infinite if $S$ is locally contractive. 
\end{thm}
\begin{proof} 
When $S$ is closed, then $F^P_{\red,{\tight}}(S,u) = F^P_{\ess,{\tight}}(S,u)$ because the underlying groupoid is Hausdorff. 
Using the groupoid models for $F^P_{\red,{\tight}}(S,u)$ and $F^P_{\ess,{\tight}}(S,u)$, see Proposition~\ref{prop:tight_reduced_twisted_inverse_semigroup_algebra} and Corollary~\ref{cor:essential_inverse_semigroup_algebra} (that extends from $p$ to $P$ by Lemma~\ref{lem:reduced_from_other_reduced}), the assertion follows from Theorems~\ref{thm:Top_Free_Intersection_Property} and \ref{thm:Anantharaman-Delaroche}. 
\end{proof}

\begin{rem}\label{rem:reduces_as_essential_tight}
It may happen that the reduced algebra $F^P_{\red,{\tight}}(S,u)$ is also an essential Banach algebra, even though the inverse semigroup is not closed. 
Then the above assertion also applies to $F^P_{\red,{\tight}}(S,u)$. 
This happens, for instance, if $S$ satisfies condition (S) introduced in \cite[Definition 5.4]{CEPSS}, and all tight filters in $E$ are ultrafilters, see \cite[Lemma 5.6]{CEPSS} and Lemma~\ref{lem:ample_essential_reduced_coincide}, cf.\ Remark~\ref{rem:essential_vs_reduced}. 
\end{rem}

\subsection{Graph algebras}
\label{subsect:graph_algebras}

Let $Q = (Q^0,Q^1, r, s)$ be a directed graph. 
So $Q^0$, $Q^1$ are countable sets whose elements are interpreted as vertices and edges respectively, and $r, s : Q^{1}\to Q^0$ are the range and source maps. 
For the standard inverse semigroup $S_{Q}$ associated to $Q$, the associated tight groupoid is the standard Deaconu--Renault groupoid $\G_Q$ for the one sided shift on the boundary path space of $Q$, see, for instance, \cite[Subsection 6.1]{BKM}. 
The groupoid $\G_Q$ is amenable, Hausdorff and second countable.  
For any $p\in [1,\infty]$, by Proposition~\ref{prop:tight_reduced_twisted_inverse_semigroup_algebra}, \cite[Corollary 6.25]{BKM} 
and \cite[Theorem 6.19]{Gardella_Lupini17}, the algebra $F^p_{\red,{\tight}}(S_{Q})$ is naturally isometrically isomorphic to the graph $L^p$-operator algebra $F^p(Q)$ considered in \cite[Definition 6.22]{BKM}, \cite{cortinas_rodrogiez}, \cite{cortinas_Montero_rodrogiez}, \cite{PhLp1}. 
Therefore, for every non-empty $P\subseteq [1,\infty]$, we call
\[
    F^P(Q) := F^P_{\red,{\tight}} (S_{Q}) \cong F^P(\G_Q) = F^P_{\red}(\G_Q) 
\]
the \emph{graph $L^P$-operator algebra} of $Q$. 

A \emph{path} in $Q$ is a sequence $\mu=\mu_1 \mu_2 \cdots $ (finite or infinite) where $s(\mu_i) = r(\mu_{i+1})$ for all $i$. 
We also treat vertices in $Q^0$ as paths of length zero. 
A finite path $\mu=\mu_1 \mu_2 \cdots \mu_n$ in $Q$ such that $s(\mu_n) = r(\mu_1)$ is called a \emph{cycle}. 
A vertex $x\in Q^0$ is called \emph{singular} if it is a source or an infinite receiver, that is, if $|r^{-1}(v)| \in \{0,\infty\}$. 
A path in $Q$ is called \emph{boundary} if either it is infinite or if it starts in a singular vertex. 
The set of boundary paths $\partial Q$ is naturally equipped with the topology that makes it homemomorphic to the tight spectrum of $S_{Q}$ and it is the unit space of $\G_Q$. In particular, for any vertex $v\in Q^0$, the cylinder
$Z_{v}:=\{\mu\in \partial Q: r(\mu)=b\}$ is compact open in $\partial Q$ and so its indicator $1_{Z_{v}}$ is a projection in $C_0(\partial Q)\subseteq  F^P(Q)$.
For two vertices $v,w\in Q^0$, we write $v\leq w$ if there is a path $\mu$ in $Q$ that ends in $v$ and starts in $w$. 

\begin{defn}\label{defn:cofinal}
We say that the graph $Q$ is \emph{cofinal} if every boundary path $\mu = \mu_1 \mu_2 \cdots $ in $Q$ is cofinal in $(Q^0,\leq)$ in the sense that for every $v\in Q^0$ there is $i$ such that $v\leq s(\mu_i)$. 
We say that a cycle $\mu=\mu_1 \mu_2 \cdots \mu_n$ has an \emph{entry} if for some $i=1, \ldots ,n$ there is an edge $e\in Q^1$ such that $r(e) = r(\mu_i)$ but $e \neq \mu_i$.
\end{defn}


The main results of \cite{cortinas_Montero_rodrogiez} and \cite{Gardella_Tinghammar}, proved using more direct methods, can be generalized using the groupoid methods as follows. 

\begin{thm}\label{thm:graph_algebras} Let $Q = (Q^0,Q^1, r, s)$ be a directed graph and let  $P\subseteq [1,\infty]$ be non-empty.
\begin{enumerate}
    \item\label{enu:graph_algebras1} If every cycle in $Q$ has an entry, then every representation of $F^P(\G)$ (into any Banach algebra) is injective whenever it is non-zero on projections $1_{Z_{v}}\in C_0(\partial Q)$, $v\in Q^0$. 
    \item \label{enu:graph_algebras2} The graph $L^P$-operator algebra $F^P(Q)$ is simple if and only if $Q$ is cofinal and every cycle in $Q$ has an entry. 
    \item \label{enu:graph_algebras3} If $F^P(Q)$ is simple, then either it is purely infinite, which holds when $Q$ contains a cycle, or it is approximately finite in the sense that it is the inductive limit of finite direct sums of full finite dimensional matrix algebras over $\F$.
\end{enumerate}
\end{thm}
\begin{proof}
We apply Corollary~\ref{cor:simplicity_intro2} from the introduction to $F^P(Q) \cong F^P(\G_Q)$.  It is well known, see, for instance, \cite[Propositions 8.2, 8.3]{Nyland_Ortega}, that $\G_Q$ is topologically free if and only if every cycle in $Q$ has an entry, and $\G_Q$ is minimal if and only if $Q$ is cofinal. 
Thus \ref{enu:graph_algebras2} in the assertion follows from  Corollary~\ref{cor:simplicity_intro2}\ref{enu:simplicity_intro2}.
Assume now that every cycle in $Q$ has an entry. 
By Corollary~\ref{cor:simplicity_intro2}\ref{enu:simplicity_intro1}, any representation $\pi:F^P(Q)\to B$ is injective whenever its restriction to $C_0(\partial Q)$ is injective.
But $\pi$ is injective on  $C_0(\partial Q)$ if it is non-zero   on projections $1_{Z_{v}}\in C_0(\partial Q)$, $v\in Q^0$.
This can be proved using that $\pi$ is non-zero on projections spanning finite-dimensional algebras turning $C_0(\partial Q)$ into an AF-algebra, cf.\ the proof of \cite[Theorem 3.7]{Kumjian_Pask_Reaburn}, or by applying the criterion in the last part of \cite[Theorem 6.15]{BKM} to $S_{Q}$. 
This shows \ref{enu:graph_algebras1}.
Assume now $Q$ is cofinal and that every cycle in $Q$ has an entry. 
It is well known and easy to see that then $\G_Q$ is locally contracting if and only if $Q$ contains a cycle. 
If this holds, then $F^P(Q)$ is purely infinite by Theorem~\ref{thm:pure_infinteness_intro}. 
If $Q$ has no cycles, then $F^P(Q)$ is an inductive limit of finite-dimensional graph algebras, 
which are $\F$-matricial algebras, see \cite[Proposition 2.6.20]{AAM}.    
This proves \ref{enu:graph_algebras3}.
\end{proof}


%
\subsection{Algebras associated to self-similar group actions on graphs}
\label{ssec:AlgebrasSelfSimilarActionsGraphs}

In this subsection, we  use our theory  to extend a number of non-trivial, recent results for $C^*$-algebras associated to self-similar group actions on graphs, as defined by  Exel and Pardo \cite{Exel-Pardo:Self-similar}, to the context of $L^P$-operator algebras.

We fix a self-similar group action on a directed graph as defined in \cite{Exel-Pardo:Self-similar}. 
Thus we let $Q = (Q^{0} , Q^{1} , r , s )$ be a \emph{finite directed graph without sources}. 
Hence $Q^{0}$, $Q^{1}$ are finite, and $r^{-1}(v)\neq \emptyset$ for all $v\in Q^0$.
We also fix a countable group $G$. 
An \emph{automorphism of $Q$} is a bijective map $\sigma : Q^0 \sqcup Q^1 \to Q^0 \sqcup Q^1$ such that $\sigma(Q^0)=Q^0$, $\sigma(Q^1)=Q^1$, and $r \circ \sigma = \sigma \circ r$ and $s \circ \sigma = \sigma \circ s$. 
A \emph{self-similar action} of $G$ on $Q$ is a group homomorphism $\sigma$ from $G$ into the group of automorphisms of $Q$ and a map $\varphi : G \times Q^{1} \to G$, called a one-cocycle or a \emph{restriction map}, whose values are sometimes denoted by $g|_e:=\varphi(g,e)$, such that
\[
    gh|_{e}  = g|_{\sigma_h(e)} h|_{e}, \qquad \sigma_{g|_{e}} (s(e)) = \sigma_g\big(s(e) \big) \qquad \text{ for all $g,h \in G,\ e\in E^{1}$.}
\]
Both the $G$-action and the restriction maps extend uniquely to paths in $Q$, in such a way that $\sigma_g (e\mu) = \sigma_g(e) \sigma_{g|_e} (\mu)$ and $g|_{e\mu} = (g|_e)|_{\mu}$ for all $g\in G$, $e\in Q^1$ and $\mu \in Q^*$ with $s(e) = r(\mu)$, see \cite[Proposition 2.4]{Exel-Pardo:Self-similar}. 
There is a natural inverse semigroup associated to these data, \cite[Sections 4, 8]{Exel-Pardo:Self-similar}. 
Namely, the set 
\[
    S_{G , Q} := \big\{ ( \alpha , g , \beta ) \in Q^* \times G \times Q^* : s(\alpha) = \sigma_g \big( s(\beta) \big) \big\} \cup \{ 0 \} 
\]
with the semigroup law defined by 
\[
    (\alpha , g , \beta) (\gamma , h , \delta) := 
        \begin{cases}  
            \big( \alpha \sigma_g(\epsilon) , g|_{\epsilon} h , \delta \big) & \text{if $\gamma = \beta \epsilon$}, \\
            \big( \alpha , g (h^{-1}|_{\epsilon})^{-1} , \delta \sigma_{h^{-1}}(\epsilon) \big) & \text{if $\beta = \epsilon \gamma$}, \\
            0 & \text{otherwise},  
        \end{cases} 
\]
and $(\alpha , g , \beta)^* := (\beta , g^{-1} , \alpha)$, is an inverse semigroup with zero \cite[Proposition 4.3]{Exel-Pardo:Self-similar}. 
We use it to define $L^P$-analogues of Exel--Pardo $C^*$-algebras:

\begin{defn} 
For non-empty $P\subseteq [1,\infty]$, using notation from Definition~\ref{def:L_p_inverse_semigroup_algebras}, we define the \emph{full}, \emph{reduced} and \emph{essential $L^P$-operator algebra of the self-similar action} $(G,Q)$ as 
\[
    F^P(G , Q) := F^P_{{\tight}}(S_{G , Q}), \quad F^P_{\red}(G , Q) := F^P_{\red,{\tight}}(S_{G , Q})\quad\text{and} \quad F^P_{\ess}(G , Q) := F^P_{\ess,{\tight}}(S_{G , Q}),
\]
respectively.
\end{defn}

\begin{rem}
When the group $G=\{e\}$ is trivial, the above algebras coincide with the graph algebra $F^P(Q)$ discussed in the previous subsection.
\end{rem}

\begin{rem} 
Let $\F=\C$ and $P=\{2\}$. 
Then $F^P(G , Q)$ is isomorphic to the universal $C^*$-algebra $\OO_{G,Q}$ introduced in \cite[Definition 3.2]{Exel-Pardo:Self-similar}, because $\OO_{G,Q}\cong C^*(\G_{{\tight}}(S_{G,Q}))= F^2_{\tight}(S_{G,Q})$ by \cite[Corollary 6.4]{Exel-Pardo:Self-similar}. 
This construction can be specialized to give the $C^*$-algebras constructed by Katsura~\cite{Katsura}. 
When the graph has only a single vertex, the algebra $F^2(G , Q)$ becomes the self-similar group $C^*$-algebra of Nekrashevych, who also considered its reduced and essential versions $F^2_{\red}(G , Q)$ and $F^2_{\ess}(G , Q)$, see~\cite[3.6, 5.4]{Nekrashevych2} and \cite{Nekrashevych}. 
\end{rem}

\begin{lem}
If $G$ is amenable, then $F^P(G , Q) = F^P_{\red}(G , Q)$.  
\end{lem}
\begin{proof}
By \cite[Corollary 10.18]{Exel-Pardo:Self-similar}, the underlying groupoid is amenable and we may apply \cite[Theorem 6.19]{Gardella_Lupini17}. 
\end{proof}

We combine the relation $\leq$ on $Q^0$ induced by edges, as described before Definition~\ref{defn:cofinal}, with the orbit equivalence $\sim$ on $Q^0$ induced by the action $\sigma$: $v\sim w$ if and only if $v=\sigma_g(w)$ for some $g\in \G$. 
We write $v \ll w$ if $v \leq u \sim w$ for some $u \in Q^0$, which holds if and only if $v \sim u' \leq w$ for some $u' \in Q^0$, see \cite[Proposition 13.2]{Exel-Pardo:Self-similar}. 

\begin{defn}\label{defn:cofinal_self-similar}
We say that the self-similar graph $(G,Q)$ is \emph{cofinal} if every infinite path $\mu = \mu_1 \mu_2 \cdots$ in $Q$ is cofinal in $(Q^0,\ll)$ in the sense that for every $v\in Q^0$ there is $i$ such that $v\ll s(\mu_i)$. 
A path $\mu$ in $Q$ is \emph{strongly fixed} by $g \in G$ if $\sigma_g(\mu) = \mu$ and $g|_{\mu} = 1_G$. 
In addition, if no proper prefix of $\mu$ is strongly fixed by $g$, we say that $\mu$ is a minimal strongly fixed path for $g$. 
An element $g \in G$ is \emph{slack at $v\in Q^0$} if there is $n\in\N$ such that all paths $\mu = \mu_1 \cdots \mu_m$ of length $m\geq n$ that ends in $v$ are strongly fixed by $g$. 
\end{defn}

\begin{thm}\label{thm:self_similar_graphs_simple}
Let $(G,Q)$ be the self-similar action as above. 
Assume that $(G,Q)$ is cofinal, every cycle in $Q$ has an entry, and if $g\in G$  fixes all infinite paths ending in $v\in Q^0$, then $g$ is a slack at $v$. 
For any non-empty $P\subseteq [1,\infty]$, $F^P_{\ess}(G , Q)$ is a simple purely infinite Banach algebra.

The reduced Banach algebra $F^P_{\red}(G , Q)$ is simple purely infinite if either every $g\in G$ admits at most finitely many minimal strongly fixed paths, or condition  \cite[(5.9)]{CEPSS} holds.
\end{thm}
\begin{proof} 
The assumption that $(G,Q)$ is cofinal is equivalent to $\G_{{\tight}}(S_{G , Q})$ being minimal by \cite[Theorem 13.6]{Exel-Pardo:Self-similar}. 
The remaining assumption on $(G,Q)$ is equivalent to $\G_{{\tight}}(S_{G , Q})$ being effective, see \cite[Theorem 14.10]{Exel-Pardo:Self-similar} and \cite[Theorem 4.7]{EP16}.
Also, under our assumptions, $\G_{{\tight}}(S_{G , Q})$ is locally contracting, see \cite[Theorem 15.1]{Exel-Pardo:Self-similar}. 
Thus since $F^P_{\ess}(G , Q)$ is an essential Banach algebra for $\G_{{\tight}}(S_{G , Q})$, cf.\ Corollary~\ref{cor:essential_inverse_semigroup_algebra}, $F^P_{\ess}(G , Q)$ is simple purely infinite by Theorem~\ref{thm:Anantharaman-Delaroche}. 
Every $g\in G$ admits at most finitely many minimal strongly fixed paths if and only if $\G_{{\tight}}(S_{G , Q})$ is Hausdorff, see \cite[Theorem 12.2]{Exel-Pardo:Self-similar}. 
Thus then $F^P_{\ess}(G , Q) = F^P_{\red}(G , Q)$. 
If \cite[(5.9)]{CEPSS} holds, then $S_{G , Q}$ satisfies condition (S) mentioned in Remark~\ref{rem:reduces_as_essential_tight}, and so then $F^P_{\red}(G , Q)$ is an essential Banach algebra. 
\end{proof}

\begin{ex}[Katsura algebras] 
Let $A$ and $B$ be two $N\times N$ matrices with integer entries. 
Katsura \cite{Katsura} associated to them a $C^*$-algebra, and Exel--Pardo noticed in  \cite[Example 3.4]{Exel-Pardo:Self-similar} that there is a natural self-similar action $(\Z , Q_A)$ associated to $A$ and $B$ that yield the same algebra. 
Thus, for any non-empty $P\subseteq [1,\infty]$, one could view $F^P(\Z , Q_A) = F^P_{\red}(\Z , Q_A)$ as $L^P$-versions of Katsura algebras. 
In particular, if the underlying groupoid is Hausdorff, which is characterized in terms of $A$ and $B$ in \cite[Theorem 18.6]{Exel-Pardo:Self-similar}, then $F^P(\Z , Q_A)$ is simple (and purely infinite) if and only if the matrix $A$ is irreducible, see \cite[Theorem  18.12]{Exel-Pardo:Self-similar}. 
An example of $A$ and $B$ such that the underlying groupoid is non-Hausdorff, but satisfies \cite[(5.9)]{CEPSS}, is analyzed in \cite[Subsection 5.4]{CEPSS}. 
Thus Theorem~\ref{thm:self_similar_graphs_simple} applied to this example gives that $F^P(\Z , Q_A)$ is simple and purely infinite.
\end{ex}

\begin{ex}[Self-similar groups] 
Let $X$ be a set with at least two elements. 
Self-similar groups $(G,X)$, cf.\ \cite{Nekrashevych}, \cite{Nekrashevych2}, are equivalent to self-similar group actions $(G,Q_X)$ on the graph $Q_X$ that has only one vertex, the set of edges $E^1=X$, and such that the extended action of $G$ on the set of finite paths $X^*$ is faithful. 
Then the associated groupoid is necessarily minimal and topologically free, cf.\  \cite[Corollary 14.14]{Exel-Pardo:Self-similar}, and so the first part of Theorem~\ref{thm:self_similar_graphs_simple} applies. 
Thus, for any self-similar group $(G,X)$ and any non-empty $P\subseteq [1,\infty]$, the associated essential algebra $F^P_{\ess}(G,X) := F^P_{\ess}(G,Q_X)$ is always purely infinite and simple. 
If $(G,X)$ is $\omega$-faithful in the sense of \cite[Definition 5.20]{CEPSS}, then also \cite[(5.9)]{CEPSS} holds, and so $F^P_{\red}(G,X) := F^P_{\red}(G,Q_X)$ is also simple purely infinite in this case. 
The second part of Theorem~\ref{thm:self_similar_graphs_simple} does not apply to the self-similar action $(G,X)$ modeling the \emph{Grigorchuk group}. 
Nevertheless, the proof of \cite[Theorem 5.22]{CEPSS} shows that $F^P(G,X) := F^P(G,Q_X)$ is simple purely infinite in this case as well. 
\end{ex}

\end{document}